\documentclass[12pt,a4paper,twoside]{article}

\usepackage[a4paper,margin=1in]{geometry}
\usepackage{amsmath,amssymb,amsthm,mathtools}
\usepackage{mathrsfs}
\usepackage{booktabs}
\usepackage{enumitem}
\usepackage{xcolor}
\usepackage{appendix}
\usepackage[
  colorlinks=true,
  linkcolor=blue!60!black,
  citecolor=blue!60!black,
  urlcolor=blue!60!black
]{hyperref}
\allowdisplaybreaks
\numberwithin{equation}{section}
\newtheorem{theorem}{Theorem}[section]
\newtheorem{proposition}[theorem]{Proposition}
\newtheorem{lemma}[theorem]{Lemma}
\newtheorem{corollary}[theorem]{Corollary}
\newtheorem{definition}[theorem]{Definition}

\newtheorem{remark}[theorem]{Remark}

\newcommand{\D}{\mathbb D}
\newcommand{\E}{\mathbb E}

\newcommand{\Pp}{\mathbb P}

\newcommand{\CPS}{\mathrm{CPS}}
 \newcommand{\dd}{\,\mathrm{d}}

\renewcommand{\boxed}[1]{#1}

\title{Poissonian Pair Correlation for 
Bernoulli Orbits}

\author{Bo Tan, \ \  Zhen Zhang, \ \ Qing-Long Zhou}

\date{}

\begin{document}

\maketitle

\begin{abstract}
Let \(\alpha, \beta \in \mathbb{T} = \mathbb{R}/\mathbb{Z}\), and let \((\xi_n)_{n \ge 1}\) be independent random variables taking the values \(\alpha\) and \(\beta\) with equal probability. We study the random walk \[X_n = \xi_1 + \cdots + \xi_n \pmod 1\] and present  a deterministic Diophantine condition on \((\alpha, \beta)\) that implies almost sure Poisson pair correlation. To the best of our knowledge, this is the first result providing explicit examples in this framework where the increment measure is purely discrete. The classical Fourier contraction method fails in this setting; a different approach is required. Our proof uses the parabolic geometry of returns to \(1\) and an exact covariance formula for overlapping blocks.
\end{abstract}

\medskip


\medskip


\section{Introduction}
\subsection{Bernoulli orbits}

Let \(\mathbb{T}=\mathbb{R}/\mathbb{Z}\) denote the circle group. For \(\alpha,\beta\in\mathbb{T}\), a nonempty closed set \(K\subseteq\mathbb{T}\) is called an \(\alpha\beta\)-set if
\[
K\subset (K-\alpha)\cup(K-\beta),
\]
or equivalently, if for every \(x\in K\), either \(x+\alpha\in K\) or \(x+\beta\in K\). A sequence \((t_n)_{n\in\mathbb{N}}\) is called an \(\alpha\beta\)-orbit if
\[
t_{n+1}-t_n\in\{\alpha,\beta\},
\]
where all sums and differences are taken in \(\mathbb{T}\), i.e., modulo one.

The notions of \(\alpha\beta\)-sets and \(\alpha\beta\)-orbits were introduced by Engelking \cite{E61} and Katznelson \cite{K79}. The two notions are closely related; indeed, every \(\alpha\beta\)-orbit is readily seen to be an \(\alpha\beta\)-set. Katznelson proved that if \(\alpha\) and \(\beta\) are rationally independent modulo one, then there exists a nowhere dense \(\alpha\beta\)-set; moreover, for certain choices of \(\alpha, \beta\), such a set may have Hausdorff dimension zero.

Motivated by problems on affine embeddings of Cantor sets, Feng and Xiong \cite{FX18} showed that for rationally independent \(\alpha,\beta\), every \(\alpha\beta\)-set \(K\) satisfies either \(K-K=\mathbb{T}\) or \(K\) has nonempty interior; consequently, its box dimension satisfies
\[
\dim_{\mathrm{B}}K\ge \frac12.
\]
Subsequently, Yu \cite{Y19} extended these estimates to settings involving \(k\) algebraic parameters, and Baker \cite{B21} proved that if \(\alpha\) is Diophantine and \(\beta\) is Liouvillian, then every \(\alpha\beta\)-set is of box dimension one. These findings led to the conjecture, explicitly stated in Yu's work, that every \(\alpha\beta\)-set has box dimension one. However, Hochman \cite{H25b} recently showed that this is not the case, demonstrating instead that the Feng--Xiong lower bound \(1/2\) is sharp.

We focus on the statistical structure of \(\alpha\beta\)-orbits. To this end, we define the symbol space
\[
\Sigma=\bigl\{\boldsymbol{\omega}=(\omega_1,\omega_2,\dots):\omega_n\in\{0,1\}\text{ for }n\in\mathbb{N}\bigr\}.
\]
For \(\boldsymbol{\omega}=(\omega_1,\omega_2,\dots)\in\Sigma\), the Bernoulli \(\alpha\beta\)-orbit \(\{t_n(\boldsymbol{\omega})\}_{n\ge1}\) is defined by
\[
t_n(\boldsymbol{\omega})=\sum_{i=1}^n \xi_i(\boldsymbol{\omega}),
\]
where
\[
\xi_i(\boldsymbol{\omega})=
\begin{cases}
\alpha, & \omega_i=0,\\
\beta, & \omega_i=1.
\end{cases}
\]
Thus \(t_n(\boldsymbol{\omega})\) records the sum of the first \(n\) steps, each of which is either \(\alpha\) or \(\beta\), according to the symbol \(\omega_i\). We endow \(\Sigma\) with the product probability measure \(\mathbb{P}\) under which the random variables \(\{\xi_i\}_{i\in\mathbb{N}}\) are independent and identically distributed, with
\[
\mathbb{P}(\xi_i=\alpha)=\mathbb{P}(\xi_i=\beta)=\frac12.
\]
Thus each infinite sequence \(\boldsymbol{\omega}\in\Sigma\) corresponds to a random walk on \(\mathbb{T}\) with steps \(\alpha\) or \(\beta\), chosen independently with equal probability.

The study of Bernoulli \(\alpha\beta\)-orbits is related to the Collatz conjecture.
Recall the \(3x+1\) function \(T\colon\mathbb{N}\to\mathbb{N}\), defined by
\[
T(n)=
\begin{cases}
\frac{3n+1}{2}, & n\text{ odd},\\[6pt]
\frac{n}{2}, & n\text{ even}.
\end{cases}
\]
The Collatz conjecture asserts that for every positive integer \(n\), there exists some \(k\) such that \(T^{(k)}(n)=1\); despite extensive study, the conjecture remains open. Motivated by this conjecture, a substantial body of work has been devoted to the behavior of its iterates, including various statistical investigations.

For a fixed \(N\), Lagarias and Soundararajan \cite{LagariasSoundararajan2006} studied the first \(N\) iterates \(x_k=T^{(k)}(x_0)\), \(1\le k\le N\), and showed that, for most initial values \(x_0\), the finite sequence \((x_k)_{1\le k\le N}\) approximately satisfies Benford's law\footnote{Benford's law to base \(B\) for an infinite sequence \(\{x_k:k\ge1\}\) asserts that \(\{\log_B x_k:k\ge1\}\) is uniformly distributed modulo one.} to integer base \(B\). A key ingredient in their proof is the following dichotomy: the Bernoulli \(\alpha\beta\)-orbit \(\{t_n(\boldsymbol{\omega})\}\) is, with probability one, uniformly distributed modulo one (or equidistributed modulo one; see Section \ref{sub1.2} for the definition) if and only if at least one of \(\alpha\) and \(\beta\) is irrational. This criterion is fundamental, as it precisely characterizes when the additive structure of the random walk ensures equidistribution on the torus \(\mathbb{T}\), thereby linking the probabilistic dynamics to the logarithmic equidistribution required for Benford's law.

From the perspective of Hausdorff dimension, Chen, Wang and Wen \cite{ChenWangWen2021} refined this picture by showing that the exceptional set in the probability space where equidistribution fails actually has full Hausdorff dimension. Peng, Wu and Xu \cite{PengWuXu2024} strengthened this result by proving that certain non-dense sets associated with random \(\alpha\beta\)-orbits also have Hausdorff dimension one. This provides a finer quantitative description of the failure of equidistribution beyond the almost-sure setting.

Beyond global equidistribution and the size of exceptional sets, a more refined understanding of Bernoulli \(\alpha\beta\)-orbits requires the study of their fine-scale local statistics. In this paper, we further investigate such local properties, aiming to capture the microscopic gap structure and spacing distribution of the sequence \(\{t_n(\boldsymbol{\omega})\}_{n\ge1}\) on \(\mathbb{T}\).

\subsection{Equidistribution and Poissonian pair correlation}\label{sub1.2}

A real sequence \(\{x_n\}_{n\in\mathbb{N}}\) is said to be equidistributed (or uniformly distributed modulo one) if for every interval \([a,b]\subseteq[0,1)\),
\begin{equation*}
 \lim_{N\to\infty} \frac{\# \{1\le n\le N : x_n \bmod 1 \in [a,b]\}}{N}=b-a.
\end{equation*}
This concept, introduced by Weyl \cite{W16}, has attracted considerable attention due to its connections with number theory, fractal geometry, dynamical systems, and probability theory.
A fundamental result is Weyl's criterion \cite[Theorem 3.2]{KN74}, which implies, for instance,  that the orbit of every irrational rotation is equidistributed.   For a comprehensive treatment, we refer to \cite{B12,KN74}.

A finer-scale understanding of the distribution of a sequence is provided by the pair correlation statistics. For \(s>0\), the pair correlation function is defined by
\[
R_N(s,\{x_n\})
=\frac1N\#\left\{1\le m,n\le N:\ m\ne n,\ 
\|x_n-x_m\|\le \frac{s}{N}\right\},
\]
where \(\#\) denotes the cardinality of a finite set, and \(\|\cdot\|\) denotes the distance to the nearest integer. We say that \(\{x_n\}\) has Poissonian pair correlation (PPC for short) if, for every \(s>0\),
\begin{equation}\label{ppc}
\lim_{N\to\infty}R_N(s,\{x_n\})=2s.
\end{equation}
This notion has its origins in quantum physics and plays a central role in the Berry--Tabor conjecture; we refer to \cite{Marklof2001} for further background.

Like equidistribution, PPC may be viewed as a pseudorandomness property: indeed, an i.i.d. sequence of random variables uniformly distributed on \(\mathbb{T}\) satisfies \eqref{ppc} almost surely. Nevertheless, the two properties are fundamentally different in nature. Equidistribution is a large-scale statistic, in the sense that the test interval is fixed, whereas PPC is a local statistic, since the interval size shrinks proportionally to \(1/N\). The two notions are related: PPC implies equidistribution \cite{ALP18,GL17,S20}, while the converse is generally false. A standard example is the sequence \((n\alpha)_{n\ge1}\), which is equidistributed if and only if \(\alpha\notin\mathbb{Q}\), but fails to have PPC for every \(\alpha\); see \cite{HK21,PS19} for related results.

The study of pair correlations for monomial sequences of the form
\[
(\alpha n^\theta \bmod 1)_{n\ge1}, \quad \theta>0,\quad \alpha>0,
\]
has attracted considerable attention, yet verifying the PPC property under explicit conditions on \(\alpha\) and \(\theta\) has proved challenging. A notable case is the sequence \((n^2\alpha)_{n\ge1}\), which is conjectured to have PPC under mild Diophantine conditions on \(\alpha\); however, only partial results in this direction are known \cite{H10,M18,RSZ01,T10}. From the perspective of Lebesgue measure, Rudnick and Sarnak \cite{RudnickSarnak1998} initiated the metric theory by proving that for almost all \(\alpha\), the sequence \((a_n\alpha)_{n\ge1}\) has PPC whenever \((a_n)\) is a polynomial of degree at least two with integer coefficients. This framework was later extended, via additive-energy criteria, to both integer sequences \cite{ALL17} and real sequences \cite{AEM21}. For monomial sequences \((\alpha n^\theta \bmod 1)_{n\ge1}\) with \(\theta>0\) and \(\alpha>0\), explicit results have also been obtained: El-Baz, Marklof and Vinogradov \cite{EMV15} established PPC for \((\sqrt{n}\bmod 1)_{n\in\mathbb{N}\setminus\square}\), where \(\square\) denotes the set of perfect squares; Technau and Yesha \cite{TY23} proved PPC for \((n^\theta \bmod 1)_{n\ge1}\) for \(\theta>7\); and for sequences \((\alpha n^\theta \bmod 1)_{n\ge1}\) with \(\alpha>0\), the admissible range of \(\theta\) was extended from \(0<\theta<14/41\) \cite{LST25} to \(0<\theta<43/117\) \cite{RS24}. Lutsko and Technau \cite{LT26} showed that for every \(m\ge3\), the sequence \((\alpha n^\theta \bmod 1)_{n\ge1}\) has Poissonian \(m\)-correlations for all \(0<\theta<1/(m^2+m-1)\) and all \(\alpha>0\). Here, Poissonian \(m\)-correlations are a higher-order generalization of PPC: for \(m=2\) they reduce to the usual pair correlation condition. More recently, Hauke \cite{H25a} proved that for sequences \((a_n)_{n\in\mathbb{N}}\) arising from the set of rough numbers with explicit roughness parameters, and for any badly approximable \(\alpha\), the sequence \((a_n\alpha)_{n\in\mathbb{N}}\) has PPC.

Turning to the random setting, the natural probabilistic counterpart is a sequence of independent and identically distributed random variables uniformly distributed on the torus, which almost surely has PPC. However, for random walks on the torus with discrete step distributions---such as the Bernoulli \(\alpha\beta\)-orbits considered in this paper---the situation is far more subtle.

\subsection{Exceptional points and Diophantine condition}
\label{subsec:DC}

Before stating our main results, we first introduce a Diophantine condition on the parameter \(\boldsymbol{\theta}\), which will control the arithmetic structure of the orbit.

We write $e(x) = e^{2\pi i x}$ for the complex character. 
Define
\[
Z(\mathbf{x}) := e(x)\cos(2\pi y), \quad \mathbf{x} = (x,y) \in \mathbb{T}^2,
\]
and set
\[
h\mathbf{x}=(hx,hy), \quad h\in\mathbb{Z}.
\]
The function $Z$ attains its peak  value $1$ precisely at the two  points
\[
 \mathbf{k}_1=(0,0), \quad \mathbf{k}_2=\left( \frac{1}{2}, \frac{1}{2} \right).
\] 
We call $\mathcal K:=\left\{\mathbf{k}_1,\mathbf{k}_2 \right\}$ the exceptional set. Note that $\mathcal K$ is a subgroup of $\mathbb T^2$ (of order $2$), and consequently,
$\mathcal K+\mathcal K
=
\mathcal K-\mathcal K
=
\mathcal K.
$
This group property will be used later to simplify certain arithmetic identities involving returns to the exceptional set.

Let $\boldsymbol{\theta} = (\sigma, \delta) \in \mathbb{T}^2$. For $h \in \mathbb{Z}$, we define
\[
\rho_h(\boldsymbol{\theta})
= \min_{m,n \in \mathbb{Z}} \left\{ \left| h\sigma - \frac{n}{2} - m \right|+|2h\delta - n|^2 \right\};
\]
this quantity measures the distance from \(h\boldsymbol{\theta}\) to the exceptional set \(\mathcal K\) in the anisotropic metric \(|x| + y^2\). The following elementary comparison shows that \(\rho_h(\boldsymbol{\theta})\asymp|1 - Z(h\boldsymbol{\theta})|\)\footnote{
Throughout, we use Vinogradov's notation: \(A\ll B\) means that \(A\le CB\) for some   constant \(C>0\); if the implied constant \(C\) depends on a parameter, say \(\varepsilon\), we indicate this by writing \(A\ll_{\varepsilon}B\). We write \(A\gg B\) if \(B\ll A\), and \(A\asymp B\) if \(A\ll B\ll A\). We also occasionally use the standard Landau notation \(O\) and \(o\).} whenever the latter is small.

\begin{theorem}
\label{thm:local}
There exist constants $\eta_{\mathrm{loc}}>0,$
such that whenever
\[h\boldsymbol{\theta}  \in B(\mathbf{k}_1, \eta_{\mathrm{loc}})
\quad \text{or}\quad 
h\boldsymbol{\theta}  \in B(\mathbf{k}_2, \eta_{\mathrm{loc}})\]
we have
\begin{equation*}
|1-Z(h\boldsymbol{\theta})|\asymp  \rho_h(\boldsymbol{\theta}).
\end{equation*}

\end{theorem}

\begin{remark}\label{rem:local}
We note that the conclusion of Theorem~\ref{thm:local} remains valid if the condition \(h\boldsymbol{\theta}  \in B(\mathbf{k}_1, \eta_{\mathrm{loc}})
\text{~or~}
h\boldsymbol{\theta}  \in B(\mathbf{k}_2, \eta_{\mathrm{loc}})\) is replaced by \(|1 - Z(h\boldsymbol{\theta})| \le \eta_{\mathrm{loc}}\).
\end{remark}

\begin{definition}[$\lambda$-Diophantine type]
Let $\lambda > 0$. We say that $\boldsymbol{\theta} \in \mathbb{T}^2$ is of \emph{$\lambda$-Diophantine type} if for every $\varepsilon > 0$, there exists a constant $c(\boldsymbol{\theta},\varepsilon ) > 0$ such that 
\[
\rho_h(\boldsymbol{\theta}) \ge c(\boldsymbol{\theta},\varepsilon )\cdot |h|^{-(\lambda+\varepsilon)} \quad (|h| \ge 1).
\]
\end{definition}

We now study the above Diophantine condition. The first result shows that it rules out linear dependencies; the second shows that it holds for a full-measure set of parameters.

\begin{theorem}\label{LI}
Suppose  $0 < \lambda < 1$, and   $\boldsymbol{\theta} = (\sigma, \delta)$ is of $\lambda$-Diophantine type. Then
\[
1, ~\sigma,~ \delta
 \text{ are linearly independent over } \mathbb Q.
\]
\end{theorem}

\begin{theorem}\label{full}
For $\lambda \ge \frac{2}{3}$, the set
\[
D_\lambda := \left\{ \boldsymbol{\theta} \in \mathbb{T}^2\colon \boldsymbol{\theta} \text{ is of } \lambda\text{-Diophantine type} \right\}
\]
has full Lebesgue measure in \(\mathbb{T}^2\).
\end{theorem}

This \(\lambda\)-Diophantine condition is the only arithmetic requirement in the variance argument. It rules out non-trivial exact returns to \(\mathcal{K}\), and is sufficient to establish the expectation and variance estimates used in the proof below. 

The proofs of all theorems in this subsection are deferred to Appendix.

\subsection{Main results}
We are now ready to state our main results.

For \(\alpha,\beta\in\mathbb{T}\), we write
\[
\alpha=\sigma-\delta,\quad \beta=\sigma+\delta
\]
for the symmetric parameters \(\sigma,\delta\in\mathbb{T}\), and set \(\boldsymbol{\theta}=(\sigma,\delta)\).

\medskip

Our main result establishes the existence of deterministic parameters $\boldsymbol{\theta}$ for which the Bernoulli $\alpha\beta$-orbits almost surely exhibit PPC.

\begin{theorem}
\label{thm:intro-deterministic-PPC}
If $\boldsymbol{\theta} = (\sigma, \delta)$ is of $2/3$-Diophantine type, then the Bernoulli $\alpha\beta$-orbit $\{t_n(\boldsymbol{\omega})\}$ has PPC with probability one, where the probability is taken over the Bernoulli sequence $(\xi_i)_{i \ge 1}$.
\end{theorem}

We make the following remarks regarding Theorem \ref{thm:intro-deterministic-PPC}$\colon$
\begin{itemize}

\item
Fiedler, Gnewuch and Wei\ss\ \cite{FiedlerGnewuchWeiss2026} proved almost sure PPC for torus random walks with absolutely continuous increment laws. Their argument uses the Fourier contraction method from \cite[Theorem 3]{S84}. Let \(Y\) be a random variable on \(\mathbb{T}\) with cumulative distribution function \(F(\cdot)\), and let \((c_r)_{r\in\mathbb{Z}}\) be its Fourier coefficients, defined by \(c_r=\int_{\mathbb{T}}e^{2\pi i r x}\dd F(x)\). If \(Y\) has an integrable Lebesgue density, then \(\sup_{r\ne0}|c_r|<1\), which implies exponential mixing of the convolution powers and, in turn, controls the pair correlation statistic via a second-moment argument (see \cite[Theorems 1.1 and 2.1]{FiedlerGnewuchWeiss2026} for more details).

In our setting, however, the increment distribution is purely discrete:
\[
\mu=\frac12\delta_\alpha+\frac12\delta_\beta,
\]
where \(\delta_x\) denotes the Dirac measure at \(x\). For such measures, the Fourier coefficients are
\[
c_r=\frac{e^{2\pi i r\alpha}+e^{2\pi i r\beta}}{2}
     =e^{2\pi i r\sigma}\cos(2\pi r\delta).
\]
Hence \(|c_r|=|\cos(2\pi r\delta)|\). When \(\delta\) is irrational, this quantity can return arbitrarily close to \(1\); when \(\delta\) is rational, it equals \(1\) for some \(r\ne0\). In either case, \(\sup_{r\ne0}|c_r|=1\), so the uniform contraction mechanism is unavailable. The PPC problem in this setting therefore requires a different approach, one in which the arithmetic nature of \(\sigma\) and \(\delta\) becomes decisive.

\item We further note that PPC is sensitive to the probability weights of the Bernoulli mixture. In the degenerate case where one step is chosen with probability \(1\)---for instance, \(\mathbb{P}(\xi_i=\alpha)=1\)---the Bernoulli \(\alpha\beta\)-orbit reduces to the Kronecker sequence \((n\alpha)_{n\ge1}\), which is known to fail PPC. It is therefore rather surprising that in the non-degenerate setting, where the weights are equal, the orbit does exhibit PPC.

\item For simplicity, we restrict our attention to the symmetric case \(\mathbb{P}(\xi_i=\alpha)=\mathbb{P}(\xi_i=\beta)=1/2\). This is the standard setting for Bernoulli \(\alpha\beta\)-orbits and already captures the main arithmetic difficulties. The asymmetric case, while of interest, introduces additional technicalities and is left for future work.


   \end{itemize}

 We  recall some preliminaries needed for the   proof.
Let \(f\in C_c^\infty(\mathbb{R})\) be real-valued, and denote its Fourier transform by
\[
\widehat f(\xi)=\int_{\mathbb{R}} f(x)e(-\xi x)\dd x.
\]
The Poisson summation formula yields the periodised function
\[
F_N(t):=\sum_{k\in\mathbb{Z}}f\bigl(N(t+k)\bigr)
=\frac1N\sum_{h\in\mathbb{Z}}\widehat f\!\left(\frac{h}{N}\right)e(ht).
\]
For a sequence \((x_n)_{n\ge1}\subset\mathbb{T}\), define the smooth pair-correlation statistic
\begin{equation}\label{pc}
R_N(f,\{x_n\}):=\frac1N\sum_{\substack{1\le m,n\le N\\ m\ne n}}F_N(x_n-x_m)
=\frac{1}{N^2}\sum_{h\in\mathbb{Z}}\widehat f\!\left(\frac{h}{N}\right)
\sum_{\substack{1\le m,n\le N\\ m\ne n}}e\bigl(h(x_n-x_m)\bigr).
\end{equation}

For an integrable function \(F\) on a probability space, we denote its expectation by \(\mathbb{E}F\) and its variance by \(\operatorname{Var}(F)\). The following theorem provides the key moment estimates needed in the proof of Theorem~\ref{thm:intro-deterministic-PPC}.

\begin{theorem}
\label{thm:intro-smooth-moments}
Let \(\boldsymbol{\theta}\) satisfy the \(2/3\)-Diophantine condition. Then for  \(0 < \varepsilon < 1/18\) and a real-valued even function \(f \in C_c^\infty(\mathbb R)\), we have that
\[
\mathbb{E} R_N(f, \{t_n(\boldsymbol{\omega})\}) \to \int_{\mathbb R} f(x) \dd x \text{ as }N\to\infty,\]
and 
\[\operatorname{Var}\bigl(R_N(f, \{t_n(\boldsymbol{\omega})\})\bigr) \ll_{f,\boldsymbol{\theta}} N^{-\varepsilon}.
\]
\end{theorem}

\begin{remark}
Put \(N_j=\lfloor j^K\rfloor\) with \(K\varepsilon>1\). Theorem~\ref{thm:intro-smooth-moments} implies
\[
\sum_{j=1}^{\infty}\operatorname{Var}\bigl(R_{N_j}(f,\{t_n(\boldsymbol{\omega})\})\bigr)<\infty,
\]
and Chebyshev's inequality together with the first Borel--Cantelli lemma yields
\[
R_{N_j}(f,\{t_n(\boldsymbol{\omega})\})\longrightarrow \int_{\mathbb{R}}f(x)\dd x
\]
with probability one. 

Since \(N_{j+1}/N_j\to1\), we may pass from this subsequence to all \(N\). To this end, we approximate the indicator function \(\mathbf{1}_{[-s,s]}\) by smooth compactly supported functions, verify the interpolation step for such smooth functions, and then pass to the limit (see \cite[Section 7]{TY23} for the standard interpolation argument). This proves Theorem~\ref{thm:intro-deterministic-PPC}.
\end{remark}

We  provide a concrete family of parameters satisfying the assumptions of our main results.

\begin{theorem}[Algebraic parameters]
\label{thm:ap}
Let $\sigma, \delta$ be real algebraic numbers such that $1, \sigma, \delta$ are linearly independent over $\mathbb Q$. Then $\boldsymbol{\theta} = (\sigma, \delta) \in \mathbb{T}^2$ is of ${2}/{3}$-Diophantine type.
\end{theorem}

\paragraph{Example.}
In particular, take
\[
\boldsymbol{\theta}=(\sigma,\delta)=(\sqrt{3},\sqrt{2}),\quad 
\alpha=\sqrt{3}-\sqrt{2},\quad 
\beta=\sqrt{3}+\sqrt{2},
\]
where all quantities are understood modulo \(1\). The conditions of Theorem~\ref{thm:intro-deterministic-PPC} are then satisfied.

\medskip

The remainder of the paper is organized as follows. The next two sections are devoted to proving the expectation and variance estimates in Theorem~\ref{thm:intro-smooth-moments}, respectively. In the final section, we give the proof of Theorem~\ref{thm:ap}.  The proofs of the theorems in Subsection ~\ref{subsec:DC}   are deferred to the appendix, as mentioned above.

\section{Expectation estimate}
\label{FM}
Let \(\boldsymbol{\theta}=(\sigma,\delta)\in\mathbb{T}^2\) be of \(2/3\)-Diophantine type. Then, for  $0<\varepsilon<\frac{1}{3}$,
\[
\rho_h(\boldsymbol{\theta}) \gg_{\boldsymbol{\theta},\varepsilon} |h|^{-(2/3+\varepsilon)} \quad (|h| \ge 1).
\]
To simplify the notation, we write \(\lambda = {2}/{3} + \varepsilon\). %
Recall that 
\begin{equation*}
\alpha=\sigma-\delta,
\quad
\beta=\sigma+\delta,
\end{equation*}
and that $\{\xi_n(\boldsymbol{\omega})\}_{n\ge1}$  is independent identically distributed random
variables satisfying
\begin{equation*}
\Pp(\xi_n(\boldsymbol{\omega})=\alpha)
=
\Pp(\xi_n(\boldsymbol{\omega})=\beta)
=
\frac12,
\end{equation*}
and 
\begin{equation*}
t_n(\boldsymbol{\omega})
=
\xi_1(\boldsymbol{\omega})+\cdots+\xi_n(\boldsymbol{\omega})
\pmod1.
\end{equation*}
For $m < n$, the independence of the steps yields
\[
\mathbb{E}\, e\bigl(h(t_n(\boldsymbol{\omega}) - t_m(\boldsymbol{\omega}))\bigr)
=
\left( \frac{e(h\alpha) + e(h\beta)}{2} \right)^{n-m}
=
\big(Z(h\boldsymbol{\theta})\big)^{n-m}.
\]
Combining this estimate with (\ref{pc}) gives 
\begin{align*}
   \E R_N(f,\{t_n(\boldsymbol{\omega})\})
   &=\frac2{N^2}\sum_{h\in\mathbb Z}\widehat f\left(\frac hN\right)\sum_{1\le m<n\le N}\big(Z(h\boldsymbol{\theta})\big)^{n-m}\\
   &=\frac2{N^2}\sum_{h\in\mathbb Z}
\widehat f\left(\frac hN\right)
\sum_{r=1}^{N-1}
(N-r)
\big(Z(h\boldsymbol{\theta})\big)^{r}\\
&=\frac{N-1}{N}\widehat f(0)+\frac2{N^2}
\sum_{h\in\mathbb Z\setminus\{0\}}
\widehat f\left(\frac hN\right)\sum_{r=1}^{N-1}
(N-r)
\big(Z(h\boldsymbol{\theta})\big)^{r}. 
\end{align*}
Thus, to obtain the expectation estimate, it suffices to prove that the error term
\begin{equation*}
\mathcal{E}_N(f)
:= \frac{2}{N^2}\sum_{h\in\mathbb{Z}\setminus\{0\}}\widehat{f}\left(\frac{h}{N}\right)
\sum_{r=1}^{N-1}
(N-r)
\bigl(Z(h\boldsymbol{\theta})\bigr)^{r}
\longrightarrow 0
\quad\text{as }N\to\infty.
\end{equation*}

Let \(N\) be sufficiently large that \(1/N<\eta_{\mathrm{loc}}\), where \(\eta_{\mathrm{loc}}>0\) is the constant from Theorem~\ref{thm:local}. We define the near-exceptional set
\[
\Omega_{N}^{(1)}
:= \left\{ h\in\mathbb{Z}\setminus\{0\} : |1-Z(h\boldsymbol{\theta})|<\frac{1}{N} \right\},
\]
and the regular set
\[
\Omega_{N}^{(2)}
:= \left\{ h\in\mathbb{Z}\setminus\{0\} : |1-Z(h\boldsymbol{\theta})|\ge\frac{1}{N} \right\}.
\]
We now split  the summation in $\mathcal E_N(f)$ into two parts, according as $h\in \Omega_{N}^{(1)}$ or $\Omega_{N}^{(2)},$ i.e.,
\begin{align*}
   \mathcal E_N(f)&=\frac2{N^2}\sum_{h\in \Omega_{N}^{(1)}}\widehat f\left(\frac hN\right)
\sum_{r=1}^{N-1}
(N-r)
\big(Z(h\boldsymbol{\theta})\big)^{r} + \frac2{N^2}\sum_{h\in \Omega_{N}^{(2)}}\widehat f\left(\frac hN\right)
\sum_{r=1}^{N-1}
(N-r)
\big(Z(h\boldsymbol{\theta})\big)^{r}\\
&=:\mathcal{E}_N^{(1)}(f) + \mathcal{E}_N^{(2)}(f).
\end{align*}

\subsection{Bounding the near-exceptional part}

For $z \in \mathbb C$ with $|z| \le 1$, define the Fejér-type kernel
\[
\mathcal{F}_N(z):= \frac{1}{N} \sum_{r=1}^{N-1} (N-r) z^r.
\]
Then 
\[\mathcal{E}_N^{(1)}(f)=\frac{1}{N}\sum_{h\in \Omega_{N}^{(1)}}\widehat f\left(\frac hN\right)\mathcal{F}_N\big(Z(h\boldsymbol{\theta})\big).\]

We begin with an auxiliary lemma that provides both an exact identity and a uniform bound for the Fejér-type kernel.

\begin{lemma}\label{FTK}
For \(z\ne1\),
\begin{equation*}
\mathcal{F}_N(z)
=\frac z{1-z}-\frac{z(1-z^N)}{N(1-z)^2}.
\end{equation*}
Moreover, uniformly for \(|z|\le1\),
\begin{equation*}
\boxed{
|\mathcal{F}_N(z)|
\ll
\min\left\{
N,\frac1{|1-z|}
\right\},
}
\end{equation*}
 with the convention that $\frac{1}{|1-z|} = +\infty$ when $z = 1$.

\end{lemma}

\begin{proof}
For \(z\ne1\), a direct computation yields 
\[
\mathcal{F}_N(z)
=
\frac{(N-1)z-Nz^2+z^{N+1}}
{N(1-z)^2}
=
\frac z{1-z}
-
\frac{z(1-z^N)}
{N(1-z)^2},
\]
and
\[
\sum_{r=1}^{N-1}(N-r)z^r
=
\sum_{k=1}^{N-1}\sum_{r=1}^{k}z^r.
\]
From this, we obtain that
\[\left|\mathcal{F}_N(z)\right|\le \frac{1}{N}\sum_{k=1}^{N-1}\sum_{r=1}^{k}1=\frac{N-1}{2} \quad 
\text{and} \quad 
\left|\mathcal{F}_N(z)\right|\le\frac{1}{N}\sum_{k=1}^{N-1}\frac{2}{|1-z|}=\frac{2}{|1-z|}.\]

\end{proof}

Since \(f\in C_c^\infty(\mathbb R)\), for every \(M > 0\) there exists a constant \(C_M > 0\) such that
\[
\left| \widehat f\!\left( \frac{h}{N} \right) \right|
\le C_M \left( 1 + \frac{|h|}{N} \right)^{-M}.
\]
Furthermore, if \(h \in \Omega_N^{(1)}\), then by Remark~\ref{rem:local},
\[
|h|^{-\lambda} \ll_{\boldsymbol{\theta}} \rho_h(\boldsymbol{\theta}) \ll \frac{1}{N}.
\]
Hence
\[
|h| \gg_{\boldsymbol{\theta}} N^{1/\lambda}.
\]
Consequently, by Lemma~\ref{FTK} and the rapid decay of \(\widehat f\),
\[
\mathcal{E}_N^{(1)}(f)
\ll \frac{1}{N}
\sum_{|h| \gg_{\boldsymbol{\theta}} N^{1/\lambda}}
\left| \widehat f\!\left( \frac{h}{N} \right) \right| \cdot N
\ll_M N^{(1 - \frac{1}{\lambda})M + \frac{1}{\lambda}}.
\]
Choosing \(M > 1/(1-\lambda)\) makes the exponent negative, and therefore
\[
\mathcal{E}_N^{(1)}(f) \longrightarrow 0 \quad (N \to \infty).
\]

\subsection{Bounding the regular part}
 
In this subsection, we devote to estimating
\[\mathcal{E}_N^{(2)}(f)=\frac{1}{N}\sum_{h\in \Omega_{N}^{(2)}}\widehat f\left(\frac hN\right)\mathcal{F}_N\big(Z(h\boldsymbol{\theta})\big).\]
By Lemma \ref{FTK}, we get 
\begin{align*}
\mathcal{E}_N^{(2)}(f)&=\frac{1}{N}\sum_{h\in \Omega_{N}^{(2)}}\widehat f\left(\frac hN\right)\mathcal{F}_N\big(Z(h\boldsymbol{\theta})\big)\\
&=
\frac{1}{N}\sum_{h\in \Omega_{N}^{(2)}}\widehat f\left(\frac hN\right)\left(\frac {Z(h\boldsymbol{\theta})}{1-Z(h\boldsymbol{\theta})}-\frac{Z(h\boldsymbol{\theta})(1-Z(h\boldsymbol{\theta})^N)}{N(1-Z(h\boldsymbol{\theta}))^2}\right)\\
&=\frac{1}{N}\sum_{h\in \Omega_{N}^{(2)}}\widehat f\left(\frac hN\right)\frac {Z(h\boldsymbol{\theta})}{1-Z(h\boldsymbol{\theta})}-\frac{1}{N^{2}}\sum_{m=1}^{N}\sum_{h\in \Omega_{N}^{(2)}}\widehat f\left(\frac hN\right)\frac{Z(h\boldsymbol{\theta})^{m}}{1-Z(h\boldsymbol{\theta})}\\
&=\frac{1}{N}
\sum_{h \in \mathbb Z}
\widehat f\!\left(\frac{h}{N}\right)
\mathbf{1}_N(h\boldsymbol{\theta})
\frac{Z(h\boldsymbol{\theta})}{1 - Z(h\boldsymbol{\theta})}-\frac{1}{N^{2}}\sum_{m=1}^{N}\sum_{h\in \mathbb{Z}}\widehat f\left(\frac hN\right)\mathbf{1}_N(h\boldsymbol{\theta})\frac{Z(h\boldsymbol{\theta})^{m}}{1-Z(h\boldsymbol{\theta})},
\end{align*}
where
\[
\mathbf{1}_N(\mathbf{x})
=
\begin{cases}
1,
& |1 - Z(\mathbf{x})| \ge \dfrac{1}{N}, \\[6pt]
0,
& \text{otherwise}.
\end{cases}
\]

The following lemma allows us to replace the discontinuous cutoff \(\mathbf{1}_N\) by a smooth one.

\begin{lemma}\label{smooth}
Let \(N\in\mathbb{N}\) be sufficiently large that \(1/N+1/N^2<\eta_{\mathrm{loc}}\). Then there exists a smooth cutoff function \(g_N\in C^\infty(\mathbb{T}^2)\) such that
\[
\left|
\frac{1}{N}
\sum_{h\in\mathbb{Z}}
\widehat{f}\!\left(\frac{h}{N}\right)
\mathbf{1}_N(h\boldsymbol{\theta})
\frac{Z(h\boldsymbol{\theta})^m}{1-Z(h\boldsymbol{\theta})}
-
\frac{1}{N}
\sum_{h\in\mathbb{Z}}
\widehat{f}\!\left(\frac{h}{N}\right)
g_N(h\boldsymbol{\theta})
\frac{Z(h\boldsymbol{\theta})^m}{1-Z(h\boldsymbol{\theta})}
\right|
\ll_M
N^{(1-1/\lambda)M+1/\lambda},
\]
where the implied constant depends on the parameter \(M\) governing the rapid decay of \(\widehat{f}\).
\end{lemma}

\begin{proof}

Let \(\psi \in C^\infty(\mathbb R)\) be a fixed function satisfying
\[
\psi(t) =
\begin{cases}
0, & t \le 0, \\
1, & t \ge 1,
\end{cases}
\quad \text{ and }
0 \le \psi(t) \le 1 \quad (t \in \mathbb R).
\]
Define
\begin{equation}\label{ap}
 g_N(\mathbf{x}) := \psi\!\left( \frac{N |1 - Z(\mathbf{x})| - 1}{1/N} \right)
= \psi\!\left( N^2 |1 - Z(\mathbf{x})| - N \right).   
\end{equation}
Then \(g_N \in C^\infty(\mathbb T^2)\) and
\[
g_N(\mathbf{x}) =
\begin{cases}
1, & |1 - Z(\mathbf{x})| \ge \dfrac{1}{N} + \dfrac{1}{N^2}, \\[6pt]
0, & |1 - Z(\mathbf{x})| \le \dfrac{1}{N} - \dfrac{1}{N^2},
\end{cases}
\]
with \(0 \le g_N \le 1\) everywhere. In the transition region
\[
T_N:=\left\{\mathbf{x}\in\mathbb{T}^{2}:\frac{1}{N}-\frac{1}{N^2}<|1-Z(\mathbf{x})|<\frac{1}{N}+\frac{1}{N^2}\right\},
\]
we have \(|g_N(\mathbf{x})-\mathbf{1}_N(\mathbf{x})|\le1\); outside this region, \(g_N(\mathbf{x})=\mathbf{1}_N(\mathbf{x})\).

Since \(1/N+1/N^2<\eta_{\mathrm{loc}}\), Remark~\ref{rem:local} implies that the difference between the sharp and smooth cutoffs is supported only in the transition region \(T_N\). Consequently,
\begin{align*}
&{}~~~~\left|
\frac{1}{N}
\sum_{h \in \mathbb Z}
\widehat f\!\left(\frac{h}{N}\right)
\mathbf{1}_N(h\boldsymbol{\theta})
\frac{Z(h\boldsymbol{\theta})^m}{1 - Z(h\boldsymbol{\theta})}
-
\frac{1}{N}
\sum_{h \in \mathbb Z}
\widehat f\!\left(\frac{h}{N}\right)
g_N(h\boldsymbol{\theta})
\frac{Z(h\boldsymbol{\theta})^m}{1 - Z(h\boldsymbol{\theta})}
\right|
\\
& = \frac{1}{N}
\left|
\sum_{h : h\boldsymbol{\theta} \in T_N}
\widehat f\!\left(\frac{h}{N}\right)
\Big(\mathbf{1}_N(h\boldsymbol{\theta})-g_N(h\boldsymbol{\theta})\Big)
\frac{Z(h\boldsymbol{\theta})^m}{1 - Z(h\boldsymbol{\theta})}
\right|
\\
&\ll_M
\frac{1}{N}
\sum_{h : |h| \gg N^{1/\lambda}}
\left( 1 + \frac{|h|}{N} \right)^{-M}
\cdot N
\ll_M
N^{(1 - 1/\lambda)M + 1/\lambda}.
\end{align*}

\end{proof}

Before proceeding with the estimates, we define the truncated Abel function, which will be used in the two auxiliary lemmas below. For \(m, N \in \mathbb{N}\), set
\[
\Phi_{N,m}(\mathbf{x})
=
\begin{cases}
g_N(\mathbf{x}) \, \dfrac{Z(\mathbf{x})^m}{1 - Z(\mathbf{x})},
& \mathbf{x} \in \mathbb T^2 \setminus \mathcal K, \\[6pt]
0,
& \mathbf{x} \in \mathcal K,
\end{cases}
\]
where \(g_N \in C^\infty(\mathbb T^2)\) is the smooth cutoff function defined in (\ref{ap}). Then
\[
\Phi_{N,m} \in C^\infty(\mathbb T^2).
\]

\begin{lemma}\label{Ab} We have 
\begin{equation*}
\lim_{N\to\infty}\int_{\mathbb{T}^{2}}\Phi_{N,m}(\mathbf{x})\dd \mathbf{x}=0.
\end{equation*}
\end{lemma}

\begin{proof}
Consider the untruncated Abel function
\begin{equation*}
G_{N,m}(\mathbf{x})
=
\frac{Z(\mathbf{x})^m}{1-Z(\mathbf{x})}.
\end{equation*}

We claim that \(G_{N,m} \in L^1(\mathbb T^2)\). Indeed, away from the two exceptional points \((0,0)\) and \((1/2,1/2)\), the function \(G_{N,m}\) is bounded. Near \((0,0)\), the Taylor expansion
\[
|1 - Z(\mathbf{x})| \asymp |x| + y^2
\]
yields
\[
|G_{N,m}(\mathbf{x})| \ll \frac{1}{|x| + y^2}.
\]
The singularity \(1/(|x|+y^2)\) is locally integrable near \((0,0)\) in \(\mathbb{T}^2\). By translation symmetry, the same local integrability holds near \((1/2,1/2)\). Hence \(G_{N,m}\in L^1(\mathbb{T}^2)\).

For almost every fixed \(y\), we have \(|\cos(2\pi y)|<1\), hence \(|Z(\mathbf{x})|<1\) almost everywhere. Expanding \(1/(1-Z)\) as a geometric series gives
\[
G_{N,m}(\mathbf{x})
=
\frac{Z(\mathbf{x})^m}{1-Z(\mathbf{x})}
=
\sum_{r=0}^\infty Z(\mathbf{x})^{m+r}
=
\sum_{r=m}^\infty Z(\mathbf{x})^r.
\]
Therefore,
\[
G_{N,m}(x,y)
=
\sum_{r=m}^\infty e(rx)\cos(2\pi y)^r.
\]

For each fixed \(y\) with \(|\cos(2\pi y)|<1\), termwise integration in \(x\) gives
\[
\int_0^1 G_{N,m}(x,y)\,dx
=
\sum_{r=m}^\infty \cos(2\pi y)^r \int_0^1 e(rx)\dd x
=
0,
\]
since \(\int_0^1 e(rx)\dd x=0\) for \(r\ge1\). Since \(G_{N,m}\in L^1(\mathbb{T}^2)\), Fubini's theorem yileds 
\[
\int_{\mathbb{T}^2} G_{N,m}(\mathbf{x})\dd \mathbf{x}
=
\int_0^1 \left(\int_0^1 G_{N,m}(x,y)\dd x\right)\dd y
=
0.
\]
Recalling that \(\Phi_{N,m}=g_NG_{N,m}\), we deduce that
\[
\int_{\mathbb{T}^2} \Phi_{N,m}(\mathbf{x})\dd{x}
=
\int_{\mathbb{T}^2} \bigl(g_N(\mathbf{x})-1\bigr)G_{N,m}(\mathbf{x})\dd{x}.
\]
The integrand on the right-hand side is supported in the region
\[
\left\{\mathbf{x}\in\mathbb{T}^2:|1-Z(\mathbf{x})|\le\frac{1}{N}+\frac{1}{N^2}\right\}.
\]
By the local comparison estimate in Theorem~\ref{thm:local}, we have
\[
\left|
\int_{\mathbb{T}^2} \Phi_{N,m}(\mathbf{x})\dd {x}
\right|
\ll
\iint_{|x|+y^2\le 2/N}
\frac{\dd x\dd y}{|x|+y^2}
\ll N^{-1/2}.
\]
Letting \(N\to\infty\) completes the proof.
\end{proof}

\begin{lemma}\label{0}
We have
\begin{equation*}
\boxed{
\lim_{N\to\infty}\frac1N
\sum_{h\in\mathbb Z}
\widehat f\left(\frac hN\right)
\Phi_{N,m}(h\boldsymbol{\theta})=0.
}
\end{equation*}
Removing the \(h=0\) term from the sum does not affect the limit. 
\end{lemma}

\begin{proof}
As usual, for \(\mathbf{x}=(x_1,x_2)\in\mathbb{T}^2\) and \(\mathbf{n}=(n_1,n_2)\in\mathbb{Z}^2\), set \(\mathbf{n}\cdot\mathbf{x}:=n_1x_1+n_2x_2\). Since \(\Phi_{N,m}\in C^\infty(\mathbb{T}^2)\), it has an absolutely convergent Fourier series
\[
\Phi_{N,m}(\mathbf{x})
=
\sum_{\mathbf{n}\in\mathbb{Z}^2} c_{\mathbf{n}} e(\mathbf{n}\cdot\mathbf{x}),
\qquad
\sum_{\mathbf{n}\in\mathbb{Z}^2} |c_{\mathbf{n}}|<\infty,
\]
with
\[
c_{\mathbf{0}}=\int_{\mathbb{T}^2}\Phi_{N,m}(\mathbf{x})\dd \mathbf{x}.
\]
Therefore,
\begin{align*}
\frac{1}{N}
\sum_{h\in\mathbb{Z}}
\widehat{f}\!\left(\frac{h}{N}\right)
\Phi_{N,m}(h\boldsymbol{\theta})
&=
\sum_{\mathbf{n}\in\mathbb{Z}^2} c_{\mathbf{n}}
\left[
\frac{1}{N}
\sum_{h\in\mathbb{Z}}
\widehat{f}\!\left(\frac{h}{N}\right)
e\bigl(h\mathbf{n}\cdot\boldsymbol{\theta}\bigr)
\right]\\
&=
c_{\mathbf{0}}\cdot\frac{1}{N}
\sum_{h\in\mathbb{Z}}
\widehat{f}\!\left(\frac{h}{N}\right)
+
\sum_{\mathbf{n}\neq\mathbf{0}} c_{\mathbf{n}}
\left[
\frac{1}{N}
\sum_{h\in\mathbb{Z}}
\widehat{f}\!\left(\frac{h}{N}\right)
e\bigl(h\mathbf{n}\cdot\boldsymbol{\theta}\bigr)
\right].
\end{align*}

By the elementary bound \(\frac{1}{N}\sum_{h}|\widehat{f}(h/N)|\ll_f1\) and Lemma~\ref{Ab}, the first term satisfies
\[
\lim_{N\to\infty}
c_{\mathbf{0}}\cdot\frac{1}{N}
\sum_{h\in\mathbb{Z}}
\widehat{f}\!\left(\frac{h}{N}\right)
=0.
\]

Now fix \(\mathbf{n}\neq\mathbf{0}\). By Theorem~\ref{LI}, we have
\[
\omega(\mathbf{n}):=\mathbf{n}\cdot\boldsymbol{\theta}
=n_1\sigma+n_2\delta\notin\mathbb{Z}.
\]
Applying Poisson summation gives
\[
\frac{1}{N}
\sum_{h\in\mathbb{Z}}
\widehat{f}\!\left(\frac{h}{N}\right)e(h\omega(\mathbf{n}))
=
\sum_{k\in\mathbb{Z}} f\bigl(N(k+\omega(\mathbf{n}))\bigr).
\]
Since \(\omega(\mathbf{n})\notin\mathbb{Z}\), we have
\[
\operatorname{dist}(\omega(\mathbf{n}),\mathbb{Z})>0.
\]
As \(f\) has compact support, the sum on the right-hand side vanishes for all sufficiently large \(N\). Therefore, for each fixed \(\mathbf{n}\neq\mathbf{0}\),
\[
\lim_{N\to\infty}
\frac{1}{N}
\sum_{h\in\mathbb{Z}}
\widehat{f}\!\left(\frac{h}{N}\right)e(h\omega(\mathbf{n}))
=0.
\]
Together with the absolute summability \(\sum_{\mathbf{n}\neq\mathbf{0}}|c_{\mathbf{n}}|<\infty\), the dominated convergence theorem yields
\[
\lim_{N\to\infty}
\sum_{\mathbf{n}\neq\mathbf{0}} c_{\mathbf{n}}
\left[
\frac{1}{N}
\sum_{h\in\mathbb{Z}}
\widehat{f}\!\left(\frac{h}{N}\right)
e\bigl(h\mathbf{n}\cdot\boldsymbol{\theta}\bigr)
\right]
=0.
\]

Since \(\Phi_{N,m}(\mathbf{0})=0\), the \(h=0\) summand in the original sum is zero and may therefore be omitted without changing the limit. The proof is complete.
\end{proof}

To bound the regular part, we apply Lemma~\ref{smooth}, which gives
\[
\lim_{N\to\infty}\mathcal{E}_N^{(2)}(f)
=
\lim_{N\to\infty}\frac{1}{N}
\sum_{h\in\mathbb{Z}}
\widehat{f}\!\left(\frac{h}{N}\right)
\Phi_{N,1}(h\boldsymbol{\theta})
-
\lim_{N\to\infty}\frac{1}{N}\sum_{m=1}^{N}
\left(
\frac{1}{N}
\sum_{h\in\mathbb{Z}}
\widehat{f}\!\left(\frac{h}{N}\right)
\Phi_{N,m}(h\boldsymbol{\theta})
\right).
\]
We now apply Lemma~\ref{0} to handle both terms, which yields
\[
\lim_{N\to\infty}\mathcal{E}_N^{(2)}(f)=0,
\]
as required.

\section{Variance estimate}
\label{SM}
In this subsection, we devote to estimating the variance. 
Recall that \[
t_n(\boldsymbol{\omega}) = \sum_{i=1}^n \xi_i(\boldsymbol{\omega}),
\]
where
\[
\xi_i(\boldsymbol{\omega}) =
\begin{cases}
\sigma-\delta, & \omega_i = 0, \\
\sigma+\delta, & \omega_i = 1.
\end{cases}
\]
By (\ref{pc}),
\begin{align*}
R_N(f,\{t_n(\boldsymbol{\omega})\})=&\frac{1}{N^{2}}\sum_{h\in\mathbb{Z}}\widehat{f}(\frac{h}{N})\sum_{\substack{1 \le m,n \le N \\ m \neq n}} e\bigl(h(t_n(\boldsymbol{\omega}) - t_m(\boldsymbol{\omega}))\bigr)\\
\asymp&\frac{1}{N^{2}}\sum_{h\in\mathbb{Z}}\widehat{f}(\frac{h}{N})\sum_{1\le m<n\le N} e\bigl(h(t_n(\boldsymbol{\omega}) - t_m(\boldsymbol{\omega}))\bigr)\\
=&\frac{1}{N^{2}}\sum_{h\in\mathbb{Z}}\widehat{f}(\frac{h}{N})\sum_{1\le m<n\le N} e\bigl(h(\xi_{m+1}(\boldsymbol{\omega})+\cdots \xi_{n}(\boldsymbol{\omega}))\bigr).
    \end{align*}

Let \(\mathcal I_N\) denote the family of all nonempty intervals
\[
I = \{m+1, \dots, n\}
\quad (1 \le m < n \le N).
\] 
For each such interval $I$, define
\[
Y_I := \sum_{k \in I} \xi_k(\boldsymbol{\omega}).
\]
Then
\begin{equation*}
R_N(f)
\asymp
\frac{1}{N^2}
\sum_{h \in \mathbb Z}
\widehat f\!\left(\frac{h}{N}\right)
\sum_{I \in \mathcal I_N}
e(h Y_I).
\end{equation*}
 Since \(f\) is even, its Fourier transform satisfies \(\widehat f(-\xi) = \widehat f(\xi)\), and hence
\begin{equation*}
\operatorname{Var}(R_N(f))
\asymp
\frac{1}{N^4}
\sum_{h_1, h_2 \in \mathbb Z}
\widehat f\!\left(\frac{h_1}{N}\right)
\widehat f\!\left(\frac{h_2}{N}\right)
\sum_{I, J \in \mathcal I_{N}}
\operatorname{Cov}\bigl(e(h_1 Y_I), e(-h_2 Y_J)\bigr),
\end{equation*}
where  \(\operatorname{Cov}(X,Y) := \mathbb E[XY] - \mathbb E[X]\,\mathbb E[Y]\) denotes the covariance of two random variables \(X\) and \(Y\). 
Note that if \(h_1=0\), then \(e(h_1Y_I)=1\) is deterministic, so
\[
\operatorname{Cov}\bigl(1,e(-h_2Y_J)\bigr)=0.
\]
The case \(h_2=0\) is analogous, with \(\operatorname{Cov}(e(h_1Y_I),1)=0\). Hence the terms with \(h_1=0\) or \(h_2=0\) contribute nothing to the above sum, and we may restrict both summations to nonzero integers:
\begin{equation*}
\operatorname{Var}(R_N(f))
\asymp
\frac{1}{N^4}
\sum_{h_1,h_2\in\mathbb{Z}\setminus\{0\}}
\widehat{f}\!\left(\frac{h_1}{N}\right)
\widehat{f}\!\left(\frac{h_2}{N}\right)
\sum_{I,J\in\mathcal{I}_{N}}
\operatorname{Cov}\bigl(e(h_1Y_I),e(-h_2Y_J)\bigr).
\end{equation*}

We split the   sum over \((h_1,h_2)\) into the diagonal part \(h_1=h_2\) and the off-diagonal part \(h_1\neq h_2\), and denote the corresponding contributions by \(\operatorname{Var}^{\mathrm{diag}}(R_N(f))\) and \(\operatorname{Var}^{\mathrm{off}}(R_N(f))\), respectively. Thus
\[
\operatorname{Var}(R_N(f))
=
\operatorname{Var}^{\mathrm{diag}}(R_N(f))
+
\operatorname{Var}^{\mathrm{off}}(R_N(f)).
\]
We   estimate the two parts separately. The key ingredient in both parts is the covariance sum
\[
C_N(h_1,h_2):=\sum_{\substack{I,J\in\mathcal{I}_N\\ I\cap J\neq\varnothing}}
\operatorname{Cov}\bigl(e(h_1Y_I),e(-h_2Y_J)\bigr),
\]
which we analyze in the following subsection.

\subsection{Estimating the covariance sum}
Here and below, we write \(z_h := Z(h\boldsymbol{\theta})\) for notational convenience. 
\begin{lemma}\label{c1}
For \(N \in \mathbb{N}\) and  \(h_1, h_2 \in \mathbb{Z} \setminus \{0\}\), we have 
\[
\boxed{
\begin{aligned}
C_N(h_1, h_2)
&=
\sum_{S=1}^{N}
\biggl(
\sum_{R=1}^{S}
\bigl( z_{h_1-h_2}^{R} - z_{h_1}^{R} z_{-h_2}^{R} \bigr)
\biggr)
\biggl[
2  
\sum_{\substack{A, B \ge 1 \\ A + B \le N - S }}
z_{h_1}^{A} z_{-h_2}^{B}
\\
&\qquad \qquad \qquad \qquad 
+
\sum_{A=1}^{N-S}
(A + 1)(z_{h_1}^{A}+ z_{-h_2}^{A})+1
\biggr].
\end{aligned}
}
\]
\end{lemma}

\begin{proof}
When \(N=0\), we have \(C_N=0\); the summation on the right-hand side is empty. We assume \(N\ge1\), and we partition all ordered pairs of intersecting intervals \(I,J\) into four disjoint classes:
\[
\text{(a) strict crossing},\qquad
\text{(b) }I\supsetneq J,\qquad
\text{(c) }J\supsetneq I,\qquad
\text{(d) }I=J.
\]
For each class, we compute   its contribution to \(C_N(h_1, h_2)\) separately.

\bigskip
\noindent\textbf{\underline{Class (a): strict crossing.}}\quad 
For $1\le A,B, R\le N$ with $A+B+R\le N$, define
\[
\mathcal S_{A,B,R}
:=
\left\{
(I, J) \in \mathcal I_N \times \mathcal I_N
:\
\# (I \setminus J) = A,\
\# (J \setminus I) = B,\
\# (I \cap J) = R
\right\}.
\]
Fix a pair \((I, J) \in \mathcal S_{A,B,R}\). The covariance contribution of this pair is
\[
\operatorname{Cov}\bigl(e(h_1 Y_I), e(-h_2 Y_J)\bigr)
=
z_{h_1}^{A} \, z_{-h_2}^{B}
\bigl( z_{h_1-h_2}^{R} - z_{h_1}^{R} z_{-h_2}^{R} \bigr).
\]
Hence the total contribution of the class is
\[
\sum_{\substack{A, B, R \ge 1 \\ A + B + R \le N}}
\#  \mathcal S_{A,B,R} \cdot
z_{h_1}^{A} \, z_{-h_2}^{B}
\bigl( z_{h_1-h_2}^{R} - z_{h_1}^{R} z_{-h_2}^{R} \bigr).
\]
For fixed \(A, B, R\), a simple counting argument gives 
\[
\#  \mathcal S_{A,B,R} =2( N - A - B - R + 1).
\]
 Therefore, the contribution of Class (a) equals
\[
2
\sum_{\substack{A, B, R \ge 1 \\ A + B + R \le N}}
(N - A - B - R + 1)
\,
z_{h_1}^{A} \, z_{-h_2}^{B}
\bigl( z_{h_1-h_2}^{R} - z_{h_1}^{R} z_{-h_2}^{R} \bigr).
\]
Applying the elementary identity
\(
N-A-B-R+1
=
\sum_{S=R}^{N-A-B}1
\)
and interchanging the summations over \(R\) and \(S\) yields that the contribution is
\[
2
\sum_{S=1}^{N-2}
\sum_{R=1}^{S}
\bigl(z_{h_1-h_2}^{R}-z_{h_1}^{R}z_{-h_2}^{R}\bigr)
\sum_{\substack{A,B\ge1\\ A+B\le N-S}}
z_{h_1}^{A}z_{-h_2}^{B}.
\]
Here the summation over \(S\) may be extended to \(N\), since for \(S>N-2\) the last summation is empty and therefore vanishes.

\bigskip
\noindent\textbf{\underline{Class (b): \(J \subsetneq I\).}}\quad
For \(A,R\ge1\) with \(A+R\le N\), define
\[
\mathcal{S}_{A,R}
:=
\left\{
(I,J)\in\mathcal{I}_N\times\mathcal{I}_N:
\#(I\setminus J)=A,\
\#(J\setminus I)=0,\
\#(I\cap J)=\#J=R
\right\}.
\]
The covariance contribution of a fixed pair \((I,J)\in\mathcal{S}_{A,R}\) is
\[
\operatorname{Cov}\bigl(e(h_1Y_I),e(-h_2Y_J)\bigr)
=
z_{h_1}^{A}\bigl(z_{h_1-h_2}^{R}-z_{h_1}^{R}z_{-h_2}^{R}\bigr);
\]
counting yields
\[
\#\mathcal{S}_{A,R}=(A+1)(N-A-R+1).
\]
Therefore, the contribution of Class (b) equals
\begin{align*}
&~\sum_{\substack{A,R\ge1\\ A+R\le N}}
(A+1)(N-A-R+1)
\,
z_{h_1}^{A}\bigl(z_{h_1-h_2}^{R}-z_{h_1}^{R}z_{-h_2}^{R}\bigr)\\
=&~\sum_{S=1}^{N}
\sum_{R=1}^{S}
\bigl(z_{h_1-h_2}^{R}-z_{h_1}^{R}z_{-h_2}^{R}\bigr)
\sum_{A=1}^{N-S}
(A+1)z_{h_1}^{A}.
\end{align*}

\bigskip
\noindent\textbf{\underline{Class (c): \(I \subsetneq J\).}}\quad
By symmetry, the contribution equals
\[
\sum_{S=1}^{N}
\sum_{R=1}^{S}
\bigl( z_{h_1-h_2}^{R} - z_{h_1}^{R} z_{-h_2}^{R} \bigr)
\sum_{A=1}^{N-S}
(A + 1) z_{-h_2}^{A}.
\]

\bigskip
\noindent\textbf{\underline{Class (d): \(I = J\).}}\quad The contribution equals
\[
\sum_{R=1}^{N}
(N - R + 1)
\bigl( z_{h_1+h_2}^{R} - z_{h_1}^{R} z_{h_2}^{R} \bigr)=
\sum_{S=1}^{N}
\sum_{R=1}^{S}
\bigl( z_{h_1-h_2}^{R} - z_{h_1}^{R} z_{-h_2}^{R} \bigr).
\]

Combining the four contributions, we complete the proof.
\end{proof}

For a formal power series \(G(t)=\sum_{n\ge0}g_nt^n\), we write
\[
[t^N]G(t):=g_N
\]
for the coefficient of \(t^N\) in \(G(t)\). All rational expressions in \(t\) below are understood as formal power series expanded at \(t=0\). With this notation, the covariance sum admits the following closed form.

\begin{lemma}\label{c2}
For \(N\in\mathbb{N}\) and \(h_1,h_2\in\mathbb{Z}\setminus\{0\}\), we have
\begin{equation*}
C_N(h_1,h_2)
=[t^N]\frac{1}{(1-t)^2}
\left(\frac{1}{1-z_{h_1-h_2}t}-\frac{1}{1-z_{h_1}z_{-h_2}t}\right)
\left(\frac{1}{1-z_{h_1}t}+\frac{1}{1-z_{-h_2}t}-1\right)^2.
\end{equation*}
\end{lemma}

\begin{proof}
Write
\begin{align*}
D_S
&:=
\sum_{R=1}^{S}
\bigl(z_{h_1-h_2}^{R}-z_{h_1}^{R}z_{-h_2}^{R}\bigr),\\
E_T
&:=
2\sum_{\substack{A,B\ge1\\ A+B\le T}}
z_{h_1}^{A}z_{-h_2}^{B}
+\sum_{A=1}^{T}(A+1)\bigl(z_{h_1}^{A}+z_{-h_2}^{A}\bigr)
+1.
\end{align*}
By Lemma~\ref{c1}, we obtain
\[
C_N(h_1,h_2)=\sum_{S=1}^{N}D_S E_{N-S}=\sum_{S=0}^{N}D_S E_{N-S},
\]
where the last equality uses \(D_0=0\). Thus the sequence \((C_N)\) is the convolution of the sequences \((D_S) \) and \((E_T) \). 
Consequently, the generating function of \((C_N)\) is the product of the generating functions of \((D_S)\) and \((E_T)\):
\begin{equation*}
\sum_{N=0}^{\infty}C_N(h_1,h_2)t^N
=
 \sum_{S=0}^{\infty}D_S t^S \cdot 
 \sum_{T=0}^{\infty}E_T t^T.
\end{equation*}
We now compute the two generating functions on the right-hand side.
Substituting the definition of \(D_S\) and interchanging the order of summation, we obtain
\[
\begin{aligned}
\sum_{S=0}^{\infty}D_S t^S
&=
\sum_{S=0}^{\infty}\sum_{R=1}^{S}
\bigl(z_{h_1-h_2}^{R}-z_{h_1}^{R}z_{-h_2}^{R}\bigr)t^S=
\sum_{R\ge1}
\bigl(z_{h_1-h_2}^{R}-z_{h_1}^{R}z_{-h_2}^{R}\bigr)
\sum_{S\ge R}t^S\\
&=
\sum_{R\ge1}
\bigl(z_{h_1-h_2}^{R}-z_{h_1}^{R}z_{-h_2}^{R}\bigr)\frac{t^R}{1-t}
=
\frac{1}{1-t}
\left(
\frac{1}{1-z_{h_1-h_2}t}
-
\frac{1}{1-z_{h_1}z_{-h_2}t}
\right).
\end{aligned}
\]
Using the identities
\begin{align*}
2\sum_{T\ge0}\sum_{\substack{A,B\ge1\\ A+B\le T}}
z_{1}^{A}z_{2}^{B}t^T
&=
\frac{2z_{1}z_{2}t^2}{(1-t)(1-z_{1}t)(1-z_{2}t)},\\
\sum_{T\ge0}\sum_{A=1}^{T}(A+1)z^{A}t^T
&=\frac{1}{1-t}\biggl(
\frac{1}{(1-zt)^2}-1\biggr),
\end{align*}
we directly compute the second generating function and obtain
\[
\sum_{T\ge0}E_T t^T
=
\frac{1}{1-t}
\left(
\frac{1}{1-z_{h_1}t}
+
\frac{1}{1-z_{-h_2}t}
-
1
\right)^2.
\]
Combining these completes the proof.
\end{proof}

Let \(\mathbb{D}=\{z\in\mathbb{C}:|z|<1\}\) and \(\overline{\mathbb{D}}=\{z\in\mathbb{C}:|z|\le1\}\) denote the open and closed unit discs, respectively.

\begin{lemma}\label{XY}
Let \(X, Y \in \overline{\mathbb{D}}\). Then for   \(N\in\mathbb{N}\), the following statements hold.

\begin{enumerate}
\item[{(i)}] If \(X = Y\), then
\[
\left|
[t^N] \frac{1}{(1-t)^2(1-Xt)^2}
\right|
\ll
\frac{N}{(|1-X| + N^{-1})^2}.
\]

\item[(ii)]  If \(X\neq Y\), then, for some \(\xi\) lying on the segment \([X,Y]\),
\begin{align*}
&\left|
[t^N]\frac{1}{(1-t)^2(1-Xt)(1-Yt)}
\right|\\[6pt]
\ll~&
\min\left\{
\frac{N}{(|1-\xi|+N^{-1})^2},~
\frac{\frac{N}{|1-X|+N^{-1}}+\frac{N}{|1-Y|+N^{-1}}}{|X-Y|}
\right\}
+\frac{N}{|1-Y|+N^{-1}}.
\end{align*}

\item[(iii)] If 
 \(
|1-X| > N^{-r}\) and  \(|1-Y| > N^{-r} 
\) for \(r>0\),
then
\[
\left|
[t^N] \frac{1}{(1-t)^2(1-Xt)(1-Yt)}
\right|
\ll N^{1+5r}.
\]

\end{enumerate}
\end{lemma}

\begin{proof}
Define
\[
F_N(X):=[t^N]\frac{1}{(1-t)^2(1-Xt)}.
\]
Then
\[
F_N(X)=X^N+2X^{N-1}+\cdots+NX+(N+1)
=\sum_{k=0}^{N}\sum_{j=0}^{N-k}X^j,
\]
which yields the trivial estimate
\(
|F_N(X)|\ll N^2.
\)

If \(X \neq 1\), the closed form
\[
F_N(X)
=
\frac{N+1}{1-X}
-
\frac{X}{(1-X)^2}
+
\frac{X^{N+2}}{(1-X)^2}
\]
yields
\[
|F_N(X)|
\le
\left| \frac{N+1}{1-X} \right|
+
\left| \frac{1 - X^{N+1}}{(1-X)^2} \right|
\ll \frac{N}{|1-X|}.
\]

Combining the two bounds, we obtain, for every \(X \in \mathbb D\),
\begin{equation}\label{X}
|F_N(X)|
\ll \min\left\{ N^2, \frac{N}{|1-X|} \right\}
\ll \frac{N}{|1-X| + N^{-1}}.
\end{equation}

\medskip

(i)  For \(X = Y\),
 an analogous computation to that for \(F_N(X)\) yields
\[
\left|
[t^N] \frac{1}{(1-t)^2(1-Xt)^2}
\right|
=
|F_{N+1}'(X)|=\left|\sum_{k=0}^{N} \sum_{j=1}^{N-k} j X^{j-1}\right|\ll N^{3},
\]
and the closed form of \(F_{N+1}'(X)\) gives
\(
|F_{N+1}'(X)|
\ll
{N}/{|1-X|^2}.
\)
Hence
\[
|F_{N+1}'(X)|
\ll
\frac{N}{(|1-X| + N^{-1})^2}.
\]

\medskip

(ii)  For \(X \neq Y\), by the partial fraction decomposition
\[
\frac{1}{(1-Xt)(1-Yt)}
=
\frac{1}{X-Y}
\left(
\frac{X}{1-Xt} - \frac{Y}{1-Yt}
\right),
\]
we obtain
\begin{equation}\label{XY-1}
\begin{aligned}
&\left|
[t^N] \frac{1}{(1-t)^2(1-Xt)(1-Yt)}
\right|
=
\left|
\frac{X F_N(X) - Y F_N(Y)}{X-Y}
\right|
\\[6pt]
\le~&
\frac{|X| \, |F_N(X) - F_N(Y)|}{|X-Y|}
+
|F_N(Y)|.
\end{aligned}
\end{equation}
By the mean value theorem applied along the line segment \([X,Y]\), there exists  \(\xi\) lying  on the segment \([X,Y]\) such that
\[
F_N(X) - F_N(Y) = F_N'(\xi)(X-Y).
\]
Together with  
\(
|F_N'(Z)| \ll  {N}/{(|1-Z|+N^{-1})^2},
\)
we obtain
\[
|F_N(X) - F_N(Y)|
=
|X-Y| \, |F_N'(\xi)|
\ll
|X-Y|\cdot\frac{N}{(|1-\xi|+N^{-1})^2}.\]
Substituting this into the inequality above yields
\[
\left|
[t^N] \frac{1}{(1-t)^2(1-Xt)(1-Yt)}
\right|
\ll
\frac{N}{(|1-\xi|+N^{-1})^2}
+
\frac{N}{|1-Y|+N^{-1}}.
\]

On the other hand, combining (\ref{X}) and (\ref{XY-1}) gives
\begin{align*}
 &\left|
[t^N] \frac{1}{(1-t)^2(1-Xt)(1-Yt)}
\right|
\ll
\frac{|F_N(x)|+|F_N(Y)|}{|X-Y|}+F_N(Y)\\[6pt]
\ll~ & \frac{\frac{N}{|1-X|+N^{-1}}+\frac{N}{|1-Y|+N^{-1}}}{|X-Y|} +\frac{N}{|1-Y|+N^{-1}}.   
\end{align*}

\medskip

(iii) If \(|X-Y|\ge N^{-4r}\), we infer directly from (ii) that
\[
\left|
[t^N]\frac{1}{(1-t)^2(1-Xt)(1-Yt)}
\right|
\ll N^{1+5r}.
\]
If instead \(|X-Y|<N^{-4r}\), then, for any \(\xi\) on the segment \([X,Y]\), we have
\[
|1-\xi|>N^{-r}-N^{-4r}\gg N^{-r}.
\]
Substituting this into (ii) once more, we reach
\[
\left|
[t^N]\frac{1}{(1-t)^2(1-Xt)(1-Yt)}
\right|
\ll N^{1+5r}.
\]


\end{proof}

\begin{lemma}\label{XYY}
Let \(X,Y\in\overline{\D}\).  Then for \(N\in\mathbb{N}\),  the following statements hold.

\begin{enumerate}

\item[(i)] If \(X = Y\), then
\[
\left|
[t^N] \frac{1}{(1-t)^2(1-Xt)^3}
\right|
\ll
\frac{N}{|1-X|^3}+\frac{N^{2}}{|1-X|^{2}}.
\]

 \item[{(ii)}] If \(X \not=Y\), then
\begin{align*}
&\left|
[t^N] \frac{1}{(1-t)^2(1-Xt)(1-Yt)^{2}}
\right|
\\[6pt]
\ll~&
N \left[
\frac{(|1-X|+N^{-1})^{-1}+(|1-Y|+N^{-1})^{-1}}{|X-Y|^{2}}
+
\frac{(|1-Y|+N^{-1})^{-2}}{|X-Y|}
\right].
\end{align*}
\end{enumerate}

\end{lemma}

\begin{proof}
We first treat the case \(X\neq Y\). We have
\[
F_N^*(X,Y)
:=
[t^N]\frac{1}{(1-t)^2(1-Xt)(1-Yt)}
=
\frac{XF_N(X)-YF_N(Y)}{X-Y},
\]
where the last equality follows from the argument in the proof of Lemma~\ref{XY}. Hence
\begin{align*}
&\left|
[t^N]\frac{1}{(1-t)^2(1-Xt)(1-Yt)^2}
\right|
=
\left|
\frac{\partial F_{N+1}^*(X,Y)}{\partial Y}
\right|\\[6pt]
\le~&
\frac{|F_{N+1}(Y)|+|Y|\,|F_{N+1}'(Y)|}{|X-Y|}
+
\frac{|X|\,|F_{N+1}(X)|+|Y|\,|F_{N+1}(Y)|}{|X-Y|^2}\\[6pt]
\ll~&
N \left[
\frac{(|1-Y|+N^{-1})^{-1}+(|1-Y|+N^{-1})^{-2}}{|X-Y|}
+
\frac{(|1-X|+N^{-1})^{-1}+(|1-Y|+N^{-1})^{-1}}{|X-Y|^2}
\right]\\[6pt]
\ll~&
N \left[
\frac{(|1-X|+N^{-1})^{-1}+(|1-Y|+N^{-1})^{-1}}{|X-Y|^2}
+
\frac{(|1-Y|+N^{-1})^{-2}}{|X-Y|}
\right].
\end{align*}

Similarly, for \(X=Y\), we have
\[
\left|
[t^N]\frac{1}{(1-t)^2(1-Xt)^3}
\right|
=
\left|
\frac{F_{N+1}''(X)}{2}
\right|
\ll
\frac{N}{|1-X|^3}+\frac{N^{2}}{|1-X|^{2}}.
\]
\end{proof}

\subsection{Bounding the diagonal part}
In this subsection, we estimate the contribution of the diagonal terms \(h_1 = h_2\in \mathbb{Z}\setminus{\{0\}}\): 
\begin{align*}
\operatorname{Var}^{\mathrm{diag}}(R_N(f))
&=\frac{1}{N^4}
\sum_{h_1\in \mathbb Z\setminus\{0\}}
\widehat f\!\left(\frac{h_1}{N}\right)
\widehat f\!\left(\frac{h_1}{N}\right)
\sum_{I, J \in \mathcal I_{N}}
\operatorname{Cov}\bigl(e(h_1 Y_I), e(-h_1 Y_J)\bigr)\\
&=\frac{1}{N^{4}}\sum_{h\in\mathbb{Z}\setminus{\{0\}}}\widehat f^{2}\!\left(\frac{h}{N}\right)C_{N}(h,h).
\end{align*}
By Lemma \ref{c2}, we have
\begin{align*}
C_{N}(h,h)&=[t^{N}]\frac{1-|z_h|^{2}}{1-|z_{h}|^{2}t} \frac{t}{(1-t)^{3}}\left(\frac{1}{1-z_ht}+\frac{1}{1-z_{-h}t}-1\right)^{2}\\
&=[t^{N-1}]\frac{1-|z_h|^{2}}{1-|z_{h}|^{2}t}\frac{1}{(1-t)^{3}}\left(\frac{1}{1-z_ht}+\frac{1}{1-z_{-h}t}-1\right)^{2}.
\end{align*}
If \(|z_h|=1\), then \(C_N(h,h)=0\). Henceforth, we shall consider only those \(h\in\mathbb{Z}\setminus\{0\}\) with \(|z_h|\neq1\), for which \(C_N(h,h)\neq0\).

\begin{lemma}\label{c-d}
We have 
\[\left|C_N(h,h)\right|\ll \frac{N^{2}}{|1-z_h|^{2}}.\]
\end{lemma}

\begin{proof}

Set
\[
\mathscr{H}
=
\left\{
(2,0), (0,2), (1,1), (1,0), (0,1), (0,0)
\right\}.
\]
To establish the desired bound, it suffices to bound uniformly over all \((a,b)\in\mathscr{H}\) the quantity \(C_N^*(a,b)\) defined below.  In the following displayed formula, we not only define \(C_N^*(a,b)\)   but also transforms it into a form suitable for the subsequent estimate.
\begin{align*}
C_N^{\ast}(a,b)
&:=\left|[t^N]\frac{1 - |z_h|^2}{1 - |z_h|^2 t}\frac{1}{(1-t)^3 (1 - z_h t)^a (1 - z_{-h} t)^b}\right|\\
&\asymp\sum_{k=0}^{N}
(1 - |z_h|^2) |z_h|^{2(N-k)}\left|
[t^k]
\frac{1}{(1-t)^3 (1 - z_h t)^a (1 - z_{-h} t)^b}
\right|\\
&\le
\max_{0 \le k \le N}
\left|
[t^k]
\frac{1}{(1-t)^3 (1 - z_h t)^a (1 - z_{-h} t)^b}
\right|
\\
&\le \max_{0\le k\le N} \sum_{j=0}^{k}\left|
[t^j] \frac{1}{(1-t)^2 (1 - z_h t)^a(1 - z_{-h} t)^b}
\right|\\
&\le 
\max_{0\le k\le N} \sum_{j=0}^{k}
\max_{0\le j\le N }\left|
[t^j] \frac{1}{(1-t)^2 (1 - z_h t)^a(1 - z_{-h} t)^b}
\right|\\
&\le N\max_{0\le k\le N}
\left|
[t^k] \frac{1}{(1-t)^2 (1 - z_h t)^a(1 - z_{-h} t)^b}
\right|.
\end{align*}

We now proceed case by case.

\medskip
\noindent\textbf{Case 1: \((a,b) = (2,0)\).}
By Lemma~\ref{XY}(i), we obtain
\[
\left|
[t^k] \frac{1}{(1-t)^2 (1 - z_h t)^2}
\right|
\ll
\frac{k}{|1 - z_h|^2},
\]
which yields
\[
C_N^{\ast}(2,0) \ll \frac{N^2}{|1 - z_h|^2}.
\]

\medskip
\noindent\textbf{Case 2: \((a,b) = (0,2)\).}
We argue as in Case 1 to obtain
\[
C_N^{\ast}(0,2) \ll \frac{N^2}{|1 - z_{-h}|^2}=\frac{N^2}{|1 - z_h|^2}.
\]

\medskip
\noindent\textbf{Case 3: \((a,b) = (1,1)\).} We estimate the auxiliary quantities
\[
F_k(\pm h):= [t^k]\frac{1}{(1-t)(1-z_{\pm h} t)}.
\]
Their absolute values satisfy
\[
|F_k(\pm h)|
=
\left|\sum_{j=0}^{k}z_{\pm h}^j\right|
\ll \frac{1}{|1-z_{\pm h}|}
=
\frac{1}{|1-z_h|}.
\]
By the convolution rule for coefficients of a product,  and applying the above bounds, we arrive at
\[\left|
[t^k]\frac{1}{(1-t)^2(1-z_h t)(1-z_{-h} t)}
\right|
\le \left| \sum_{j=0}^{k}  F_j(h) F_{k-j}(-h)\right|
\ll \frac{k}{|1-z_h|^2},
\]
which yields
\[
C_N^{\ast}(1,1)\ll \frac{N^{2}}{|1-z_h|^{2}}.
\]

\medskip
\noindent\textbf{Case 4: \((a,b) = (1,0)\).} By (\ref{X}), we have
\[
\left|
[t^k]\frac{1}{(1-t)^2(1-z_h t)}
\right|
\ll\frac{k}{|1-z_h|}< \frac{k}{|1-z_h|^{2}},\]
which gives 
\[C_N^{\ast}(1,0)\ll \frac{N^{2}}{|1-z_h|^{2}}.\]

\medskip
\noindent\textbf{Case 5: \((a,b) = (0,1)\).}
 We have
\[
C_N^{\ast}(0,1) \ll \frac{N^2}{|1 - z_{-h}|^2}=\frac{N^2}{|1 - z_h|^2}.
\]

\medskip
\noindent\textbf{Case 6: \((a,b) = (0,0)\).}
Taking \(X=0\) in   Lemma~\ref{XY}(i),  we have
\[
\left|
[t^k]\frac{1}{(1-t)^2}
\right|
\ll \frac{k}{|1-z_h|^{2}},\]
and thus 
\[C_N^{\ast}(0,0)\ll \frac{N^{2}}{|1-z_h|^{2}}.\]

\end{proof}

\begin{lemma}\label{diag}
Let \(s \in \mathbb{N}\). For \(N\) sufficiently large that \((1/N)^{1/8}<\eta_{\mathrm{loc}}\), where \(\eta_{\mathrm{loc}}>0\) is the constant from Theorem~\ref{thm:local}, we have
\[\frac{1}{N^2}
\sum_{h \in \mathbb Z \setminus \{0\}}
\left| \widehat f\!\left(\frac{h}{N}\right) \right|^s
\cdot
\frac{1}{|1 - z_h|^2}\ll N^{-\frac{1}{3}}\]
\end{lemma}

\begin{proof}
For $0<\varepsilon <{1}/{18},$  we fix $\lambda= {2}/{3}+\varepsilon$ and  
\(M \ge (2\lambda(1+\varepsilon) + 1/3)/\varepsilon.\) 

As in the expectation estimate in the previous section, we split the relevant sum into two parts, one over the near-exceptional set and the other over the regular set. Here the definitions of these two sets are slightly different from those used before; the near-exceptional set and the regular set are defined respectively by
\[
\Omega_N^{(1)}
:=
\left\{
h \in \mathbb Z \setminus \{0\} : |1 - z_h| < \frac{1}{N^{1/8}}
\right\}~~\text{and } ~\Omega_N^{(2)}
:=
\left\{
h \in \mathbb Z \setminus \{0\} : |1 - z_h| \ge \frac{1}{N^{1/8}}
\right\}.
\]
We then have
\[\frac{1}{N^2}
\sum_{h \in \mathbb Z \setminus \{0\}}
\left| \widehat f\!\left(\frac{h}{N}\right) \right|^2
\cdot
\frac{1}{|1 - z_h|^2}
=
\frac{1}{N^2}
\sum_{h \in \Omega_N^{(1)}}
+
\frac{1}{N^2}
\sum_{h \in \Omega_N^{(2)}}
=:
S_{1,N}
+S_{2,N}.
\]

\paragraph{Estimate of $S_{1,N}$.}  By Theorem \ref{thm:local} and Remark \ref{rem:local},  we obtain 
\[S_{1,N}\asymp\frac{1}{N^2}
\sum_{h \in \Omega_N^{(1)}}
\left| \widehat f\!\left(\frac{h}{N}\right) \right|^s
\cdot
\frac{1}{\rho_h(\boldsymbol{\theta})^2}.\]
Splitting the sum according to \(|h| \le N^{1+\varepsilon}\) and \(|h| > N^{1+\varepsilon}\), we obtain
\begin{align*}
S_{1,N}
&\ll_{\boldsymbol{\theta}}
\frac{1}{N^2}
\sum_{\substack{h \in \Omega_N^{(1)} \\ 0 < |h| \le N^{1+\varepsilon}}}
\left|\widehat f\!\left(\frac{h}{N}\right) \right|^s
\cdot
\frac{1}{\rho_h(\boldsymbol{\theta})^2}
+
\frac{1}{N^2}
\sum_{\substack{h \in \Omega_N^{(1)} \\ |h| > N^{1+\varepsilon}}}
\left(1 + \frac{|h|}{N}\right)^{-sM}
|h|^{2\lambda}\\
&\ll_{\boldsymbol{\theta},M}
\frac{1}{N^2}
\sum_{\substack{h \in \Omega_N^{(1)} \\ 0 < |h| \le N^{1+\varepsilon}}}
\left|\widehat f\!\left(\frac{h}{N}\right) \right|^s
\cdot
\frac{1}{\rho_h(\boldsymbol{\theta})^2}+N^{-sM\varepsilon+2\lambda(1+\varepsilon)}\\
&\le\frac{1}{N^2} \sum_{j=0}^{\left\lceil (1+\varepsilon)\log N \right\rceil} \sum_{2^{j}\le |h|<2^{j+1}}\frac{\mathbf{1}_{\{h\colon \rho_h(\boldsymbol{\theta})\ll \frac{1}{N} \}}}{\rho_h(\boldsymbol{\theta})^2}+N^{-sM\varepsilon+2\lambda(1+\varepsilon)}.\end{align*}
It remains to control the contribution from the first sum.
For $k\in\mathbb{N},$ we define 
\[\mathcal B_{j,k}:=\{h\in\mathbb{Z}\colon 2^{j}\le |h|<2^{j+1}, 2^{-(k+1)}\le\rho_h(\boldsymbol{\theta})<2^{-k} \}.\]
Since $\rho_h(\boldsymbol{\theta})\gg|h|^{-\lambda}$ and  $2^{-j\lambda}\ll2^{-k}\ll (\frac{1}{N})^{1/8}$,  
\begin{align*}
\sum_{2^{j}\le |h|<2^{j+1}}\frac{\mathbf{1}_{\{h\colon \rho_h(\boldsymbol{\theta})\ll \frac{1}{N} \}}}{\rho_h(\boldsymbol{\theta})^2}\asymp\sum_{k: 2^{-j\lambda}\ll2^{-k}\ll (\frac{1}{N})^{1/8}}~\sum_{h\in \mathcal B_{j,k}}2^{2k}
=\sum_{k: 2^{-j\lambda}\ll2^{-k}\ll (\frac{1}{N})^{1/8}} \# \mathcal B_{j,k}\cdot 2^{2k}.
    \end{align*}

A key observation is the following: for  \(h_1, h_2 \in \mathcal B_{j,k}\) with \(h_1 \neq h_2\), 
\[
|h_1 - h_2|^{-\lambda}
\ll \rho_{h_1 - h_2}(\boldsymbol{\theta})
\ll \rho_{h_1}(\boldsymbol{\theta}) + \rho_{h_2}(\boldsymbol{\theta})
< 2^{-(k-1)}.
\]
Indeed, the first inequality follows from the Diophantine condition, the second from the triangle inequality for \(\rho_h\), and the third from the defining property of \(\mathcal B_{j,k}\). 

Consequently,
\[
|h_1 - h_2| \gg 2^{k/\lambda},
\]
which implies that the elements of \(\mathcal B_{j,k}\) are separated by at least a distance of order \(2^{k/\lambda}\). Since \(\mathcal B_{j,k} \subset [-2^{j+1}, 2^{j+1}]\), we obtain
\[
\#  \mathcal B_{j,k}
\ll
\frac{2^{j}}{2^{k/\lambda}}+1.
\]
It follows that 
\begin{align*}
\sum_{2^{j}\le |h|<2^{j+1}}\frac{\mathbf{1}_{\{h\colon \rho_h(\boldsymbol{\theta})\ll \frac{1}{N} \}}}{\rho_h(\boldsymbol{\theta})^2}
\ll \sum_{k: 2^{-j\lambda}\ll2^{-k}\ll (\frac{1}{N})^{1/8}}  2^{2k}(\frac{2^{j}}{2^{k/\lambda}}+1)\asymp2^{2j\lambda} .
    \end{align*}

 Therefore we conclude that
\[S_{1,N}\ll_{\boldsymbol{\theta},M}\frac{1}{N^2} \sum_{j=0}^{\left\lceil (1+\varepsilon)\log N \right\rceil} 2^{2j\lambda}+ N^{-sM\varepsilon+2\lambda(1+\varepsilon)}\ll N^{-2+2\lambda(1+\varepsilon)}+N^{-sM\varepsilon+2\lambda(1+\varepsilon)}\ll N^{-\frac{1}{3}}.\]

\paragraph{Estimate of $S_{2,N}$.} We deduce that 
\begin{align*}
S_{2,N}
\le \frac{1}{N^{2}}\left(\sum_{1\le |h|<N^{1+\varepsilon}} N^{\frac{1}{4}}+\sum_{ |h|\ge N^{1+\varepsilon}} (1+\frac{|h|}{N})^{-2M}N^{\frac{1}{4}}\right)\ll N^{-\frac{1}{3}}.
\end{align*}

\end{proof}

To conclude, combining Lemma~\ref{c-d} and Lemma~\ref{diag}, we obtain
\begin{equation}\label{var:diag}
    \operatorname{Var}^{\mathrm{diag}}(R_N(f)) \ll N^{-1/3}.
\end{equation}

\subsection{Bounding the off-diagonal part}

In this subsection, we estimate the off-diagonal contribution
\[
\operatorname{Var}^{\mathrm{off}}(R_N(f))
=
\frac{1}{N^4}
\sum_{\substack{h_1, h_2 \in \mathbb Z \setminus \{0\} \\ h_1 \neq h_2}}
\widehat f\!\left(\frac{h_1}{N}\right)
\widehat f\!\left(\frac{h_2}{N}\right)
C_N(h_1, h_2).
\]
We decompose the sum into a good part and a bad part.  For \(0 < r < 1\), define the good set
\[
\mathcal G_N
:=
\left\{
(h_1, h_2) \in (\mathbb Z \setminus \{0\})^2 : h_1 \neq h_2,\
|1 - z_{h_1 - h_2}|,\ |1 - z_{h_1} z_{-h_2}|,\ |1 - z_{h_1}|,\ |1 - z_{-h_2}| > N^{-r}
\right\},
\]
and the bad set
\[
\mathcal B_N
:=
\left\{
(h_1, h_2) \in (\mathbb Z \setminus \{0\})^2 : h_1 \neq h_2
\right\}
\setminus
\mathcal G_N.
\]
Accordingly, we write
\[
\operatorname{Var}^{\mathrm{off}}(R_N(f))
=
\operatorname{Var}^{\mathrm{off}}_{\mathcal G}(R_N(f))
+
\operatorname{Var}^{\mathrm{off}}_{\mathcal B}(R_N(f)),
\]
where the two terms correspond to the summation over \(\mathcal G_N\) and \(\mathcal B_N\), respectively.

\medskip
In each part, we estimate the contribution by establishing a bound for \(C_N(h_1,h_2)\). As before, we set
\[
\mathscr{H}
=
\bigl\{
(2,0),\,(0,2),\,(1,1),\,(1,0),\,(0,1),\,(0,0)
\bigr\}.
\]
For \((a,b)\in\mathscr{H}\) and \(\varpi\in\{z_{h_1-h_2},z_{h_1}z_{-h_2}\}\), we define
\[
C_{N}(h_1,h_2;a,b;\varpi)
:=
\left|
[t^{N}]\frac{1}{(1-t)^{2}}
\frac{1}{(1-\varpi t)(1-z_{h_1}t)^{a}(1-z_{-h_2}t)^{b}}
\right|
\]
and
\[
C_{N}(h_1,h_2;a,b)
:=
\sum_{\varpi\in\{z_{h_1-h_2},z_{h_1}z_{-h_2}\}}
C_{N}(h_1,h_2;a,b;\varpi).
\]
It is readily checked that
\begin{align*}
&[t^{N}]\frac{1}{(1-t)^{2}}
\left(
\frac{1}{1-z_{h_1-h_2}t}
-
\frac{1}{1-z_{h_1}z_{-h_2}t}
\right)
\frac{1}{(1-z_{h_1}t)^{a}(1-z_{-h_2}t)^{b}}
\\[6pt]
\le{}& C_{N}(h_1,h_2;a,b),
\end{align*}
and thus, in view of Lemma~\ref{c2}, to establish the desired bound for \(C_N(h_1,h_2)\), it suffices to bound \(C_{N}(h_1,h_2;a,b)\) for every \((a,b)\in\mathscr{H}\).

\subsubsection{Good part}

Our approach consists of two steps: we estimate \(C_{N}(h_1,h_2;a,b)\) separately in the cases \(a+b\le1\) and \(a+b=2\), and combine these estimates to bound the off-diagonal variance over the good set.

\medskip

   We estimate \(C_{N}(h_1,h_2;a,b)\) by considering two cases:

\smallskip

\noindent \underline{\textbf{Case I:}\quad \(a+b\le 1.\)}
If \((a,b)=(0,0)\), then by \eqref{X}, we obtain
\[
C_{N}(h_1,h_2;0,0)
\ll \frac{N}{|1-z_{h_1-h_2}|}+\frac{N}{|1-z_{h_1}z_{-h_2}|}
\le 2N^{1+r}.
\]
If \((a,b)=(1,0)\) or $(0,1)$, then by Lemma \ref{XY}(iii), we have

  \[
  C_{N}(h_1,h_2;1,0)
\ll N^{1+5r}.
  \]

In summary, for \(a+b\le 1\), we have
\[
C_{N}(h_1,h_2;a,b)\ll N^{1+5r}.
\]

\medskip

\noindent \underline{\textbf{Case II:}\quad \(a+b=2.\)}
For \(
\varpi  \in \{z_{h_1-h_2},\, z_{h_1} z_{-h_2}\},
\)
we estimate \(C_{N}(h_1,h_2;a,b; \varpi)\) by dividing it further into two cases: 

\medskip
 
 \underline{\textbf{Case II.1.} } \quad\(|\varpi  - z_{h_1}| \le N^{-4r}\) or \(|\varpi  - z_{-h_2}| \le N^{-4r}\).

 \smallskip

We treat this case via a series of lemmas.

\begin{lemma}\label{X_i}
For  $X=z_{h_1}$ or $z_{-h_2}$, if
\(
|\varpi-X|\le N^{-4r}, 
\)
then 
\[
|\varpi| \le 1 - N^{-4r} ~\text{ and }~  |X| \le 1 - N^{-4r}.
\]
\end{lemma}

\begin{proof}
We treat the case \(X=z_{h_1}\); the other case is analogous.

Suppose for contradiction that one of \(\varpi\) and \(z_{h_1}\) has modulus \(>1-N^{-4r}\). Then the triangle inequality implies that both \(|\varpi|>1-2N^{-4r}\) and \(|z_{h_1}|>1-2N^{-4r}\).

\medskip
We distinguish the cases according to the value of \(\varpi\).

\medskip
\noindent\textbf{Case 1: \(\varpi=z_{h_1-h_2}\).}\quad Then
\[
|z_{h_1-h_2}-z_{h_1}|\le N^{-4r}
\quad\text{and}\quad
|z_{h_1}|,\ |z_{h_1-h_2}|\ge 1-2N^{-4r}.
\]
In this case, we may write
\[
\cos(2\pi h_1\delta)=\varepsilon_1+O(N^{-4r}),
\qquad
\cos(2\pi(h_1-h_2)\delta)=\varepsilon_2+O(N^{-4r})
\]
for some \(\varepsilon_1,\varepsilon_2=\pm1\).
Since cosine decays quadratically near its extrema,
\[
2\pi h_1\delta=k_1\pi+O(N^{-2r}),
\qquad
2\pi(h_1-h_2)\delta=k_2\pi+O(N^{-2r}),
\]
for some integers \(k_1,k_2\). Subtracting yields
\[
2\pi h_2\delta=(k_1-k_2)\pi+O(N^{-2r}),
\]
hence
\[
\cos(2\pi h_2\delta)=\varepsilon_1\varepsilon_2+O(N^{-4r}).
\]

On the other hand, from \(|z_{h_1-h_2}-z_{h_1}|\le N^{-4r}\), we obtain
\[
|e(-h_2\sigma)\cos(2\pi(h_1-h_2)\delta)-\cos(2\pi h_1\delta)|\le N^{-4r}.
\]
Substituting the cosine estimates,
\[
|e(-h_2\sigma)\varepsilon_2-\varepsilon_1|
\le N^{-4r}+O(N^{-4r})
\ll N^{-4r}.
\]
Thus
\[
e(-h_2\sigma)=\varepsilon_1\varepsilon_2+O(N^{-4r}).
\]
Combining this with the estimate for \(\cos(2\pi h_2\delta)\),
\[
z_{-h_2}
=
e(-h_2\sigma)\cos(2\pi h_2\delta)
=
\bigl(\varepsilon_1\varepsilon_2+O(N^{-4r})\bigr)
\bigl(\varepsilon_1\varepsilon_2+O(N^{-4r})\bigr)
=
1+O(N^{-4r}).
\]
When \(N\) is sufficiently large, this contradicts \(|1-z_{-h_2}|\ge N^{-r}\).

\medskip
\noindent\textbf{Case 2: \(\varpi  = z_{h_1} z_{-h_2}\).} Then
\[
|z_{h_1}z_{-h_2}-z_{h_1}|\le N^{-4r},
\]
and since \(|z_{h_1}|\le1\),
\[
|1-z_{-h_2}|\le N^{-4r},
\]
which contradicts \(|1-z_{-h_2}|>N^{-r}\) for \(N\) sufficiently large.
\end{proof}

\begin{lemma}\label{HX}
Suppose that, for $i=1,2,\ldots, s$,
\[
\quad |X_i|\le 1-N^{-4r}.
\]
Let $\{m_i: 1\le i\le  s\}\subset \mathbb{N}$ and
\(H(t)=\sum_{n= 0}^{\infty}h_n t^n\). We have 
\[
\left|
[t^N]H(t)\prod_{i=1}^{s}(1-X_it)^{-m_i}
\right|
\ll
 N^{4r(m_1+\cdots+m_s)}
\sup_{0\le n\le N}|h_n|.
\]
\end{lemma}

\begin{proof}
Expand each factor \((1-X_it)^{-m_i}\) into its negative-binomial series:
\[
(1-X_it)^{-m_i}
=\sum_{n_i= 0}^{\infty}\binom{n_i+m_i-1}{m_i-1}X_i^{n_i}t^{n_i}.
\]
Thus, comparing coefficients of \(t^N\), we obtain
\[
[t^N]H(t)\prod_{i=1}^{s}(1-X_it)^{-m_i}
=
\sum_{\substack{n_1,\dots,n_s\ge 0\\ n_1+\cdots+n_s\le N}}
h_{N-n_1-\cdots-n_s}
\prod_{i=1}^{s}\binom{n_i+m_i-1}{m_i-1}X_i^{n_i}.
\]
Taking absolute values and using \(|h_{N-n_1-\cdots-n_s}|\le \sup_{0\le n\le N}|h_n|\), we get
\[
\left|
[t^N]H(t)\prod_{i=1}^{s}(1-X_it)^{-m_i}
\right|
\le
\sup_{0\le n\le N}|h_n|
\sum_{\substack{n_1,\dots,n_s\ge 0\\ n_1+\cdots+n_s\le N}}
\prod_{i=1}^{s}\binom{n_i+m_i-1}{m_i-1}|X_i|^{n_i}.
\]
Enlarging the constraint \(n_1+\cdots+n_s\le N\) to the product ranges \(0\le n_i\le N\) for each \(i\), we obtain
\[
\left|
[t^N]H(t)\prod_{i=1}^{s}(1-X_it)^{-m_i}
\right|
\le
\sup_{0\le n\le N}|h_n|
\prod_{i=1}^{s}
\sum_{n_i=0}^{N}\binom{n_i+m_i-1}{m_i-1}|X_i|^{n_i}.
\]
Since for each fixed \(i\),
\[
\sum_{n_i=0}^{N}\binom{n_i+m_i-1}{m_i-1}|X_i|^{n_i}\le (1-|X_i|)^{-m_i}
\le N^{4rm_i},
\]
we conclude that
\[
\left|
[t^N]H(t)\prod_{i=1}^{s}(1-X_it)^{-m_i}
\right|
\le
N^{4r(m_1+\cdots+m_s)}
\sup_{0\le n\le N}|h_n|,
\]
as required.
\end{proof}

\begin{lemma}
Suppose that 
\[|\varpi  - z_{h_1}| \le N^{-4r}\quad \text{or}\quad |\varpi  - z_{-h_2}| \le N^{-4r}.\]
    Then for  \((a,b) \in \mathscr{H}\) with $a+b=2$, we have
\[
C_{N}(h_1,h_2;a,b;\varpi)\ll N^{1+12r}.
\]
\end{lemma}
\begin{proof}
The proof is divided into two cases:

\medskip
\noindent\textbf{Case 1:} both \(|\varpi  - z_{h_1}| \le N^{-4r}\) and \(|\varpi  - z_{-h_2}| \le N^{-4r}\) hold.

Applying Lemma~\ref{X_i}, we obtain
\[
|\varpi |, \; |z_{h_1}|, \; |z_{-h_2}| \le 1 - N^{-4r}.
\]
Write $H(t)=1/(1-t)^{2}=\sum_{n=0}^{\infty}(n+1)t^{n}.$  Applying Lemma~\ref{HX} with \(m_1 + \cdots + m_s = 3\) we have  for all \(a+b=2\) that 
\[
\left|
[t^N] H(t)
\frac{1}{(1-\varpi  t)(1-z_{h_1}t)^a (1-z_{-h_2}t)^b}
\right|
\ll N^{1+12r}.
\]

\medskip
\noindent\textbf{Case 2:} exactly one of \(|\varpi  - z_{h_1}| \le N^{-4r}\) and \(|\varpi  - z_{-h_2}| \le N^{-4r}\) holds.

It suffices to treat the case
\[
|\varpi  - z_{h_1}| \le N^{-4r} < |\varpi  - z_{-h_2}|;
\]
the other case is symmetric.

If \((a,b) = (1,1)\), the partial fraction decomposition
\[
\frac{1}{(1-\varpi  t)(1-z_{h_1}t)(1-z_{-h_2}t)}
=
\frac{1}{\varpi  - z_{-h_2}}
\left(
\frac{\varpi }{(1-\varpi  t)(1-z_{h_1}t)}
-
\frac{z_{-h_2}}{(1-z_{-h_2}t)(1-z_{h_1}t)}
\right)
\]
yields
\begin{align*}
&\left|
[t^N] \frac{1}{(1-t)^2}
\frac{1}{(1-\varpi  t)(1-z_{h_1}t)(1-z_{-h_2}t)}
\right| \\
\le~&
\frac{1}{|\varpi  - z_{-h_2}|}
\left(
\left|
[t^N] \frac{1}{(1-t)^2}
\frac{\varpi }{(1-\varpi  t)(1-z_{h_1}t)}
\right|
+
\left|
[t^N] \frac{1}{(1-t)^2}
\frac{z_{-h_2}}{(1-z_{-h_2}t)(1-z_{h_1}t)}
\right|
\right) \\
\ll~ & N^{4r} \cdot N^{1+5r} = N^{1+9r},
\end{align*}
where the last inequality follows from Lemma~\ref{XY}(iii) and \(|\varpi  - z_{-h_2}|^{-1} < N^{4r}\).

If \((a,b) = (0,2)\), applying Lemma~\ref{XYY}(ii)  gives
\[
\left|
[t^N] \frac{1}{(1-t)^2}
\frac{1}{(1-\varpi  t) (1-z_{-h_2}t)^2}
\right|
\ll N^{1+9r}.
\]

If \((a,b) = (2,0)\), applying Lemmas \ref{X_i} and \ref{HX} yields 
\[
\left|
[t^N] \frac{1}{(1-t)^{2}}
\frac{1}{(1-\varpi  t)(1-z_{h_1}t)^2}
\right|
\ll N^{1+12r}.
\]

Combining the two cases, we conclude
\[
C_{N}(h_1, h_2; a, b;\varpi) \ll N^{1+12r}.
\]
\end{proof}

\underline{\textbf{Case II.2.} } \quad  
\(|\varpi  - z_{h_1}| > N^{-4r}\) and \(|\varpi  - z_{-h_2}| > N^{-4r}\). 

\smallskip

Using \eqref{X} if $(a,b)=(0,0)$; Lemma \ref{XY}(iii) if $a+b=1$;
and Lemma \ref{XYY} if $a+b=2$, we obtain 
\[
C_{N}(h_1, h_2; a, b; \varpi) \ll N^{1+12r}.
\]

\medskip

Summarizing, for every $(h_1,h_2)\in\mathcal G_N$ and   $(a,b)\in \mathscr{H}$,
\[
C_{N}(h_1,h_2;a,b)\ll N^{1+12r},
\]
and thus
\[
C_N(h_1,h_2)\ll N^{1+12r}.
\]

\bigskip 

We are ready to bound \(\operatorname{Var}^{\mathrm{off}}_{\mathcal G}(R_N(f))=\frac{1}{N^{4}}\sum_{(h_1,h_2)\in \mathcal G_N} \widehat f\!\left(\frac{h_1}{N}\right)
\widehat f\!\left(\frac{h_2}{N}\right)C_N(h_1, h_2).\)
\begin{lemma}\label{var:og}
For sufficiently large $N$, we have
\[\operatorname{Var}^{\mathrm{off}}_{\mathcal G}(R_N(f))\ll N^{-1+12r}.\]
\end{lemma}

\begin{proof}
We deduce that
\begin{align*}
 \operatorname{Var}^{\mathrm{off}}_{\mathcal G}(R_N(f))&=\frac{1}{N^{4}}\sum_{(h_1,h_2)\in \mathcal G_N} \widehat f\!\left(\frac{h_1}{N}\right)
\widehat f\!\left(\frac{h_2}{N}\right)C_N(h_1, h_2)  \\
& \ll \frac{1}{N^{4}}\sum_{(h_1,h_2)\in (\mathbb{Z}\setminus\{0\})^{2}} \left|\widehat f\!\left(\frac{h_1}{N}\right)
\widehat f\!\left(\frac{h_2}{N}\right)\right|N^{1+12r}\\
&\le N^{-1+12r},
\end{align*}
where the last inequality follows from $\sum_h|\widehat f\!\left(\frac{h}{N}\right)|\ll N.$    
\end{proof}

\subsubsection{Bad part}

In this subsection, we aim to estimate
\[\operatorname{Var}^{\mathrm{off}}_{\mathcal B}(R_N(f))=\frac{1}{N^{4}}\sum_{(h_1,h_2)\in \mathcal B_N} \widehat f\!\left(\frac{h_1}{N}\right)
\widehat f\!\left(\frac{h_2}{N}\right)C_N(h_1, h_2).\]
Recall that the bad set \(\mathcal{B}_N\) consists of all \((h_1,h_2)\in(\mathbb{Z}\setminus\{0\})^2\) with \(h_1\neq h_2\) such that at least one of the following four quantities is \(\le N^{-r}\):
\[
|1-z_{h_1-h_2}|,\quad |1-z_{h_1}z_{-h_2}|,\quad |1-z_{h_1}|,\quad |1-z_{-h_2}|.
\]
We shall treat the contribution of \(\mathcal{B}_N\) in two parts, according to how many of these four quantities are \(\le N^{-r}\).
\medskip

\noindent\underline{\textbf{Part 1:}}\quad Exactly one of the above four quantities is \(\le N^{-r}\).

\medskip
In Part 1, we may assume without loss of generality that the distinguished quantity is \(\le N^{-r}/2\). Indeed, if it lies in \((N^{-r}/2,N^{-r}]\), then since the other three quantities are all \(>N^{-r}\),   the same argument as in the good part applies to show that
\[
C_N(h_1,h_2)\ll N^{1+12r},
\]
and thus the contribution of this part to the variance is \(\ll N^{-1+12r}\).

We then divide the analysis into four subparts.

\medskip
\underline{\textbf{Subpart 1-1:}}\quad  \(|1-z_{h_1}|\le N^{-r}/2.\) 

The triangle inequality gives
\[
|z_{h_1} - z_{h_1-h_2}|\ge |1-z_{h_1-h_2}|-|1-z_{h_1} |>N^{-r} - N^{-r}/2 = N^{-r}/2,
\]
and similarly
\[
|z_{h_1} - z_{h_1} z_{-h_2}| > N^{-r}/2,
\quad 
|z_{h_1}-z_{-h_2}| > N^{-r}/2.
\]

\begin{lemma}\label{P1-1}
For all \((a,b) \in \mathscr{H}\), we have
\[
C_{N}(h_1,h_2;a,b)\ll \frac{N^{1+5r}}{|1 - z_{h_1}|^{2}}.
\]

\end{lemma}

\begin{proof} We consider the various cases according to \((a,b)\).

If \((a,b) = (0,0)\), then by \eqref{X},
\[
C_{N}(h_1, h_2; 0, 0)
\ll \frac{N}{|1 - z_{h_1-h_2}|} + \frac{N}{|1 - z_{h_1} z_{-h_2}|}
\le 2N^{1+r}.
\]

If \((a,b) = (0,1)\), then by Lemma~\ref{XY}(iii),
\[
C_{N}(h_1, h_2; 0, 1) \ll N^{1+5r}.
\]

If \((a,b) = (1,0)\), then by Lemma~\ref{XY}(iii),
\[
C_{N}(h_1, h_2; 1, 0)
\ll \frac{N^{1+2r}}{|1 - z_{h_1}|}.
\]

If \((a,b) = (1,1)\), for $\varpi \in \{z_{h_1-h_2},z_{h_1}z_{-h_2}\},$ the partial fraction decomposition
\[
\frac{1}{(1-\varpi  t)(1-z_{h_1}t)(1-z_{-h_2}t)}
=
\frac{1}{\varpi  - z_{h_1}}
\left(
\frac{\varpi }{(1-\varpi  t)(1-z_{-h_2}t)}
-
\frac{z_{h_1}}{(1-z_{-h_2}t)(1-z_{h_1}t)}
\right)
\]
yields
\begin{align*}
&\left|
[t^N] \frac{1}{(1-t)^2}
\frac{1}{(1-\varpi  t)(1-z_{h_1}t)(1-z_{-h_2}t)}
\right| \\[6pt]
\le~&
\frac{1}{|\varpi  - z_{h_1}|}
\left(
\left|
[t^N] \frac{1}{(1-t)^2}
\frac{\varpi }{(1-\varpi  t)(1-z_{-h_2}t)}
\right|
+
\left|
[t^N] \frac{1}{(1-t)^2}
\frac{z_{h_1}}{(1-z_{-h_2}t)(1-z_{h_1}t)}
\right|
\right) \\[6pt]
 \ll~& N^{r} \cdot (N^{1+5r}+\frac{N^{1+2r}}{|1-z_{h_1}|}) \ll \frac{N^{1+6r}}{|1-z_{h_1}|},
\end{align*}
where the penultimate inequality follows from Lemma~\ref{XY}(ii)-(iii) and \(|\varpi  - z_{h_1}|^{-1} < 2N^{r}\).

If \((a,b) = (2,0)\), then by Lemma \ref{XYY},  we have 
\[C_{N}(h_1, h_2; 2, 0) \ll \frac{N^{1+3r}}{|1-z_{h_1}|^{2}}\]

If \((a,b) = (0,2)\), we distinguish two cases according to the size of \(|z_{-h_2}|\).

\medskip
\noindent\textbf{Case 1: \(|z_{-h_2}| < 1/2\).}
By Lemma~\ref{XY}(iii), for   \(\varpi  \in \{z_{h_1-h_2},\, z_{h_1} z_{-h_2}\}\),
\[
\left|
[t^n] \frac{1}{(1-t)^2}
\frac{1}{(1-\varpi  t)(1-z_{-h_2}t)}
\right|
\ll n^{1+5r}.
\]
Since \(|z_{-h_2}| < 1/2\), we have \(\sum_{n=0}^{N} |z_{-h_2}|^n \ll 1\), and hence by convolution,
\[
\left|
[t^N] \frac{1}{(1-t)^2}
\frac{1}{(1-\varpi  t)(1-z_{-h_2}t)^2}
\right|
\ll N^{1+5r}.
\]

\medskip
\noindent\textbf{Case 2: \(|z_{-h_2}| \ge 1/2\).}

For \(\varpi  = z_{h_1} z_{-h_2}\), Theorem~\ref{thm:local} gives
\[
|\varpi  - z_{-h_2}|
= |z_{-h_2}| \cdot |1 - z_{h_1}|
\asymp |1 - z_{h_1}|.
\]
Thus Lemma~\ref{XYY}(ii) with \(X = \varpi \), \(Y = z_{-h_2}\) yields
\[
\left|
[t^N] \frac{1}{(1-t)^2}
\frac{1}{(1-\varpi  t)(1-z_{-h_2}t)^2}
\right|
\ll
\frac{N^{1+2r}}{|1 - z_{h_1}|^{2}}.
\]

For \(\varpi  = z_{h_1-h_2}\), we distinguish two subcases.

\medskip
 \textit{Subcase 2a: \(|\varpi  - z_{h_1} z_{-h_2}| \le |1 - z_{h_1}|/4\).}
The triangle inequality gives
\[
|\varpi  - z_{-h_2}|
\ge |z_{h_1} z_{-h_2} - z_{-h_2}| - |\varpi  - z_{h_1} z_{-h_2}|
\ge \frac{|1 - z_{h_1}|}{4}.
\]
Thus Lemma~\ref{XYY}(ii) with \(X = \varpi \), \(Y = z_{-h_2}\) yields
\[
\left|
[t^N] \frac{1}{(1-t)^2}
\frac{1}{(1-\varpi  t)(1-z_{-h_2}t)^2}
\right|
\ll
\frac{N^{1+2r}}{|1 - z_{h_1}|^{2}}.
\]

\medskip
 \textit{Subcase 2b: \(|\varpi  - z_{h_1} z_{-h_2}| > |1 - z_{h_1}|/4\).}
We use the identity
\[
|\varpi  - z_{h_1} z_{-h_2}|
= |\sin(2\pi h_1\delta)\sin(2\pi h_2\delta)|
= \sqrt{(1 - |z_{h_1}|^2)(1 - |z_{-h_2}|^2)}.
\]
The hypothesis \(|\varpi  - z_{h_1} z_{-h_2}| > |1 - z_{h_1}|/4\) then gives
\[
1 - |z_{-h_2}| > \frac{|1 - z_{h_1}|}{64}.
\]
Consequently,
\begin{align*}
&\left|
[t^N] \frac{1}{(1-t)^2(1-\varpi  t)(1-z_{-h_2}t)^2}
\right|
\\
 \le~& 
\sup_{0 \le n \le N}
\left|
[t^n] \frac{1}{(1-t)^2(1-\varpi  t)(1-z_{-h_2}t)}
\right|
\cdot
\sum_{n=0}^{N} |z_{-h_2}|^n
\ll \frac{N^{1+5r}}{|1 - z_{h_1}|}.
\end{align*}

In summary, for all \((a,b) \in \mathscr{H}\),
\[
C_{N}(h_1, h_2; a, b) \ll \frac{N^{1+5r}}{|1 - z_{h_1}|^{2}}.
\]
\end{proof}

We now estimate the contribution to  the off-diagonal   variance from those \((h_1, h_2) \in \mathcal B_N\) satisfying \(|1 - z_{h_1}| \le N^{-r}/2\). By Lemma~\ref{P1-1},
\[
C_N(h_1, h_2) \ll \frac{N^{1+5r}}{|1 - z_{h_1}|^{2}}.
\]
Therefore,
\begin{align*}
\operatorname{Var}^{\mathrm{off},1-1}_{\mathcal B}(R_N(f))&:=\frac{1}{N^4}
\sum_{(h_1, h_2) \in \mathcal B_N}
\widehat f\!\left(\frac{h_1}{N}\right)
\widehat f\!\left(\frac{h_2}{N}\right)
\mathbf{1}_{\{|1 - z_{h_1}| \le N^{-r}/2\}}(h_1)
C_N(h_1, h_2)
\\
&\ll
\frac{N^{1+5r}}{N^4}
\sum_{h_1 \in \mathbb{Z} \setminus \{0\}}
\left|\widehat f\!\left(\frac{h_1}{N}\right)\right|
\frac{\mathbf{1}_{\{|1 - z_{h_1}| \le N^{-r}/2\}}(h_1)}{|1 - z_{h_1}|^{2}}
\sum_{h_2 \in \mathbb{Z} \setminus \{0\}}
\left|\widehat f\!\left(\frac{h_2}{N}\right)\right|
\\
& \ll
N^{-2+5r}
\sum_{h_1 \in \mathbb{Z} \setminus \{0\}}
\left|\widehat f\!\left(\frac{h_1}{N}\right)\right|
\frac{1}{|1 - z_{h_1}|^{2}}
\\
& \ll N^{-1/3 + 5r},
\end{align*}
where the penultimate inequality uses
\(
\sum_{h \in \mathbb{Z}} \left|\widehat f\!\left(\frac{h}{N}\right)\right| \ll N,
\)
and the last inequality follows from Lemma~\ref{diag} with $s=1$.

\medskip
\underline{\textbf{Subpart 1-2:}}\quad  \(|1 - z_{-h_2}| \le N^{-r}/2\).

This case is analogous to Subpart 1-1, with the roles of \(h_1\) and \(-h_2\) interchanged. The same argument yields
\[
\operatorname{Var}^{\mathrm{off},1-2}_{\mathcal B}(R_N(f))
\ll N^{-1/3 + 5r}.
\]

\medskip
\underline{\textbf{Subpart 1-3:}}\quad  \(|1 - z_{h_1-h_2}| \le N^{-r}/2\).

The triangle inequality gives
\[
|z_{h_1} - z_{h_1-h_2}|
\ge |1 - z_{h_1-h_2}| - |1 - z_{h_1}|
> N^{-r} - N^{-r}/2
= N^{-r}/2,
\]
and similarly
\[
|z_{h_1-h_2} - z_{-h_2}| > N^{-r}/2.
\]

\begin{lemma}\label{P1-3}
For  \((a,b) \in \mathscr{H}\),
\[
C_{N}(h_1, h_2; a, b)
\ll \frac{N^{1+12r}}{|1 - z_{h_1-h_2}|+N^{-1}}.
\]
\end{lemma}

\begin{proof}
We treat the cases according to \(a+b\).

\medskip
\noindent\textbf{Case 0: \(a+b=0\).}\quad
By \eqref{X},
\[
C_{N}(h_1, h_2; 0, 0)
\ll \frac{N}{|1 - z_{h_1-h_2}|+N^{-1}}.
\]

\medskip
\noindent\textbf{Case 1: \(a+b=1\).}\quad
By Lemma~\ref{XY}(ii),
\[
\left|
[t^N] \frac{1}{(1-t)^2(1 - z_{h_1-h_2}t)(1 - z_{h_1}t)^a (1 - z_{-h_2}t)^b}
\right|
\ll \frac{N^{1+2r}}{|1 - z_{h_1-h_2}|+N^{-1}},
\]
and by Lemma~\ref{XY}(iii),
\[
\left|
[t^N] \frac{1}{(1-t)^2(1 - z_{h_1} z_{-h_2}t)(1 - z_{h_1}t)^a (1 - z_{-h_2}t)^b}
\right|
\ll N^{1+5r}.
\]
Combining the two bounds and using \(|1 - z_{h_1-h_2}|+N^{-1} \le N^{-r}/2 +N^{-1}< 1\) gives
\[
C_{N}(h_1, h_2; a, b)
\ll \frac{N^{1+5r}}{|1 - z_{h_1-h_2}|+N^{-1}}.
\]

\medskip
\noindent\textbf{Case 2: \(a+b=2\).}\quad
For \((a,b) = (1,1)\), the partial fraction decomposition
\begin{align*}
&\frac{1}{(1 - z_{h_1-h_2} t)(1 - z_{h_1} t)(1 - z_{-h_2} t)}
\\
=~&
\frac{1}{z_{h_1-h_2} - z_{h_1}}
\left(
\frac{z_{h_1-h_2}}{(1 - z_{h_1-h_2} t)(1 - z_{-h_2} t)}
-
\frac{z_{h_1}}{(1 - z_{-h_2} t)(1 - z_{h_1} t)}
\right).
\end{align*}
yields
\begin{align*}
&\left|
[t^N] \frac{1}{(1-t)^2}
\frac{1}{(1 - z_{h_1-h_2}t)(1 - z_{h_1}t)(1 - z_{-h_2}t)}
\right| \\
\le~&
\frac{1}{|z_{h_1-h_2} - z_{h_1}|}
\left(
\left|
[t^N] \frac{1}{(1-t)^2}
\frac{z_{h_1-h_2}}{(1 - z_{h_1-h_2}t)(1 - z_{-h_2}t)}
\right|
+
\left|
[t^N] \frac{1}{(1-t)^2}
\frac{z_{h_1}}{(1 - z_{-h_2}t)(1 - z_{h_1}t)}
\right|
\right) \\
\ll~& N^{r} \left( \frac{N^{1+2r}}{|1 - z_{h_1-h_2}|+N^{-1}} + N^{1+5r} \right)
\ll \frac{N^{1+6r}}{|1 - z_{h_1-h_2}|+N^{-1}},
\end{align*}
where the second inequality uses Lemma~\ref{XY}(ii).

For \((a,b) = (2,0)\) or \((0,2)\), Lemma~\ref{XYY}(ii) gives
\[
\left|
[t^N] \frac{1}{(1-t)^2}
\frac{1}{(1 - z_{h_1-h_2}t)(1 - z_{h_1}t)^a (1 - z_{-h_2}t)^b}
\right|
\ll \frac{N^{1+3r}}{|1 - z_{h_1-h_2}|+N^{-1}}.
\]
On the other hand, applying the \textbf{Case II} estimate for the \textbf{Good part} with \(\varpi  = z_{h_1-h_2}\) yields
\[
\left|
[t^N] \frac{1}{(1-t)^2}
\frac{1}{(1 - z_{h_1}z_{-h_2}t)(1 - z_{h_1}t)^a (1 - z_{-h_2}t)^b}
\right|
\ll N^{1+12r}.
\]
Hence, for \(a+b=2\),
\[
C_{N}(h_1, h_2; a, b)
\ll \frac{N^{1+12r}}{|1 - z_{h_1-h_2}|+N^{-1}}.
\]

\medskip
Summarizing the three cases, we conclude that for  \((a,b) \in \mathscr{H}\),
\[
C_{N}(h_1, h_2; a, b)
\ll \frac{N^{1+12r}}{|1 - z_{h_1-h_2}|+N^{-1}}.
\]
\end{proof}

We now estimate the contribution to the off-diagonal variance from those \((h_1, h_2) \in \mathcal B_N\) satisfying \(|1 - z_{h_1-h_2}| \le N^{-r}/2\). By Lemma~\ref{P1-3},
\[
C_N(h_1, h_2) \ll\frac{N^{1+12r}}{|1 - z_{h_1-h_2}|}\le  \frac{N^{1+12r}}{|1 - z_{h_1-h_2}|^{2}}.
\]
Therefore,
\begin{align*}
\operatorname{Var}^{\mathrm{off},1-3}_{\mathcal B}(R_N(f))
&:=
\frac{1}{N^4}
\sum_{(h_1, h_2) \in \mathcal B_N}
\widehat f\!\left(\frac{h_1}{N}\right)
\widehat f\!\left(\frac{h_2}{N}\right)
\mathbf{1}_{\{|1 - z_{h_1-h_2}| \le N^{-r}/2\}}(h_1, h_2)
C_N(h_1, h_2)
\\
&\ll
N^{-3+5r}
\sum_{(h_1, h_2) \in \mathcal B_N}
\left|\widehat f\!\left(\frac{h_1}{N}\right)\right|
\left|\widehat f\!\left(\frac{h_2}{N}\right)\right|
\frac{\mathbf{1}_{\{|1 - z_{h_1-h_2}| \le N^{-r}/2\}}(h_1, h_2)}{|1 - z_{h_1-h_2}|^{2}}
\\
&=
N^{-3+5r}
\sum_{k \in \mathbb{Z} \setminus \{0\}}
\frac{\mathbf{1}_{\{|1 - z_k| \le N^{-r}/2\}}}{|1 - z_k|^{2}}
\sum_{\substack{(h_1, h_2) \in (\mathbb{Z} \setminus \{0\})^{2}\\ h_1 - h_2 = k}}
\left|\widehat f\!\left(\frac{h_1}{N}\right)\right|
\left|\widehat f\!\left(\frac{h_2}{N}\right)\right|.
\end{align*}

It remains to estimate the inner sum. Since \(\widehat f\) decays rapidly, we have
\[
\sum_{\substack{(h_1, h_2) \in \mathcal B_N \\ h_1 - h_2 = k}}
\left|\widehat f\!\left(\frac{h_1}{N}\right)\right|
\left|\widehat f\!\left(\frac{h_2}{N}\right)\right|
\ll
N \left(1+\frac{|k|}{N}\right)^{-M}.
\]
Indeed, writing \(h_1 = h_2 + k\), the inner sum becomes
\[
\sum_{h_2} \left|\widehat f\!\left(\frac{h_2 + k}{N}\right)\right|
\left|\widehat f\!\left(\frac{h_2}{N}\right)\right|.
\]
Using the rapid decay of \(\widehat f\), this is bounded by
\[
\sum_{h_2} \left(1+\frac{|h_2+k|}{N}\right)^{-M-2}
\left(1+\frac{|h_2|}{N}\right)^{-M-2}.
\]
To extract the factor \((1+|k|/N)^{-M}\), we use the elementary inequality
\[
1 + \frac{|k|}{N} \ll \left(1+\frac{|h_2+k|}{N}\right)\left(1+\frac{|h_2|}{N}\right),
\]
which follows from \(|k| \le |h_2+k| + |h_2|\). Hence
\[
\left(1+\frac{|h_2+k|}{N}\right)^{-M}
\left(1+\frac{|h_2|}{N}\right)^{-M}\ll \left(1+\frac{|k|}{N}\right)^{-M},
\]
and therefore
\[
\sum_{h_2} \bigl(1+\tfrac{|h_2+k|}{N}\bigr)^{-M-2}
\bigl(1+\tfrac{|h_2|}{N}\bigr)^{-M-2}
\ll
\bigl(1+\tfrac{|k|}{N}\bigr)^{-M}
\sum_{h_2} \bigl(1+\tfrac{|h_2|}{N}\bigr)^{-2}
\ll
N \bigl(1+\tfrac{|k|}{N}\bigr)^{-M},
\]
where the last step uses \(\sum_{h_2} (1+|h_2|/N)^{-2} \ll N\).

Substituting back,
\[
\operatorname{Var}^{\mathrm{off},1-3}_{\mathcal B}(R_N(f))
\ll
N^{-2+5r}
\sum_{k \in \mathbb{Z} \setminus \{0\}}
\left(1+\frac{|k|}{N}\right)^{-M}
\frac{\mathbf{1}_{\{|1 - z_k| \le N^{-r}/2\}}}{|1 - z_k|^{2}}.
\]
Now Lemma~\ref{diag} applies, provided \(M\) is chosen sufficiently large, say
\[
M \ge \frac{2\lambda(1+\varepsilon) + 1/3}{\varepsilon},
\]
where \(0 < \varepsilon < 1/18\), as in the proof of Lemma~\ref{diag}. Therefore
\[
\frac{1}{N^2}
\sum_{k \in \mathbb{Z} \setminus \{0\}}
\left(1+\frac{|k|}{N}\right)^{-M}
\frac{1}{|1 - z_k|^{2}}
\ll N^{-1/3}.
\]
Hence
\begin{equation}\label{var:1-3}
\operatorname{Var}^{\mathrm{off},1-3}_{\mathcal B}(R_N(f))
\ll N^{-1/3 + 5r}.    
\end{equation}

\medskip
\underline{\textbf{Subpart 1-4:}}\quad  \(|1 - z_{h_1} z_{-h_2}| \le N^{-r}/2\).

By the same argument as in Subpart 1-3, with \(z_{h_1-h_2}\) replaced by \(z_{h_1} z_{-h_2}\), we obtain
\[
C_N(h_1, h_2) \ll \frac{N^{1+12r}}{|1 - z_{h_1} z_{-h_2}|+N^{-1}}.
\]
It remains to estimate
\[
\operatorname{Var}^{\mathrm{off},1-4}_{\mathcal B}(R_N(f))
:=
\frac{1}{N^4}
\sum_{(h_1, h_2) \in \mathcal B_N}
\widehat f\!\left(\frac{h_1}{N}\right)
\widehat f\!\left(\frac{h_2}{N}\right)
\mathbf{1}_{\{|1 - z_{h_1} z_{-h_2}| \le N^{-r}/2\}}(h_1, h_2)
C_N(h_1, h_2).
\]
Substituting the bound on \(C_N(h_1, h_2)\) gives
\begin{align*}
\operatorname{Var}^{\mathrm{off},1-4}_{\mathcal B}(R_N(f))
&\ll
N^{-3+12r}
\sum_{(h_1, h_2) \in \mathcal B_N}
\left|\widehat f\!\left(\frac{h_1}{N}\right)\right|
\left|\widehat f\!\left(\frac{h_2}{N}\right)\right|
\frac{\mathbf{1}_{\{|1 - z_{h_1} z_{-h_2}| \le N^{-r}/2\}}(h_1, h_2)}{|1 - z_{h_1} z_{-h_2}|+N^{-1}}.
\end{align*}

To proceed, we need to understand how the condition \(|1 - z_{h_1} z_{-h_2}| \le N^{-r}/2\) constrains the possible values of \(h_2\) for fixed \(h_1\). This is the content of the following lemma.
\begin{lemma}
\label{df}
Let \(A_0 \in \overline{\mathbb D}\) and \(0 < \eta \le \eta_{\mathrm{loc}}\), where \(\eta_{\mathrm{loc}} > 0\) is the sufficiently small absolute constant from Theorem~\ref{thm:local}. Assume that \(q_1 \neq q_2 \in \mathbb Z \setminus \{0\}\) satisfy
\[
|1 - A_0 z_{-q_i}| \le \eta, \quad i = 1, 2.
\]
Then
\[
\rho_{q_1 - q_2}(\boldsymbol{\theta}) \ll \eta.
\]
The implicit constant is uniform in \(A_0\).
\end{lemma}

\begin{proof}
 By assumption,
\[
|1 - A_0 z_{-q_i}| \le \eta,
\]
so \(|A_0| \ge 1 - \eta\). Since \(\eta \le \eta_{\mathrm{loc}} < 1\), we have \(A_0 \neq 0\), and we may write
\[
A_0 = r e(\phi), \quad 1 - \eta \le r \le 1, \quad \phi \in \mathbb T.
\]

Because \(r \le 1\), the inequality \(|1 - A_0 z_{-q_i}| \le \eta\) implies
\(
|\cos(2\pi q_i\delta)| \ge 1 - \eta.
\)
Since \(|\cos(2\pi x)|\) attains its maximum \(1\) only at \(x \in \{0, 1/2\} \pmod 1\), with quadratic decay near these points, there exist for each \(i\) some \(\varepsilon_i \in \{0, 1\}\) and \(v_i \in \mathbb R\) such that
\[
q_i\delta = \frac{\varepsilon_i}{2} + v_i \pmod 1, \quad |v_i| \le C_1\eta^{1/2},
\]
for an absolute constant \(C_1 > 0\). Hence
\[
\cos(2\pi q_i\delta) = e\!\left(\frac{\varepsilon_i}{2}\right)\cos(2\pi v_i),
\]
and therefore
\[
A_0 z_{-q_i} = r\cos(2\pi v_i) \, e\!\left(\phi - q_i\sigma + \frac{\varepsilon_i}{2}\right).
\]
Since \(r\cos(2\pi v_i) = 1 + O(\eta)\), the inequality \(|1 - A_0 z_{-q_i}| \le \eta\) yields
\[
\left|1 - \widetilde r_i \, e(t_i)\right| \le C_2\eta,
\]
where
\[
\widetilde r_i := r\cos(2\pi v_i), \quad t_i := \phi - q_i\sigma + \frac{\varepsilon_i}{2},
\]
and \(1 - C_2\eta \le \widetilde r_i \le 1\).

We now derive the bound on \(\|t_i\|_{\mathbb T}\) directly. The identity
\[
|1 - \widetilde r_i e(t_i)|^2
=
(1 - \widetilde r_i)^2 + 2\widetilde r_i \bigl(1 - \cos(2\pi t_i)\bigr)
\]
together with \(\widetilde r_i \ge 1/2\) (which holds after decreasing \(\eta_{\mathrm{loc}}\) if necessary) and \(|1 - \widetilde r_i e(t_i)| \le C_2\eta\) gives
\[
1 - \cos(2\pi t_i) \ll_{C_2} \eta^2.
\]
Since \(1 - \cos(2\pi t_i) \asymp \|t_i\|\) uniformly on \(\mathbb T\), we conclude
\[
\|t_i\|\le C_3\eta
\]
for some absolute constant \(C_3 > 0\). Hence there exists \(u_i \in \mathbb R\) such that
\[
q_i\sigma = \phi + \frac{\varepsilon_i}{2} + u_i \pmod 1, \quad |u_i| \le C_3\eta.
\]

Combining the two coordinates,
\[
q_i\boldsymbol{\theta} = \kappa_{\varepsilon_i}(\phi) + (u_i, v_i), \quad i = 1, 2,
\]
where the center \(\kappa_{\varepsilon_i}(\phi)\) is defined by
\[
\kappa_{\varepsilon_i}(\phi)
:=
\left(\phi + \frac{\varepsilon_i}{2}, \frac{\varepsilon_i}{2}\right) \pmod 1.
\]
Thus \(\kappa_0(\phi) = (\phi, 0)\) and \(\kappa_1(\phi) = (\phi + 1/2, 1/2)\), which are precisely the translates of the two exceptional points in
\[
\mathcal K = \left\{(0,0), \left(\tfrac{1}{2}, \tfrac{1}{2}\right)\right\}.
\]

We now compare the two centers. If \(\varepsilon_1 = \varepsilon_2\), then \(\kappa_{\varepsilon_1}(\phi) - \kappa_{\varepsilon_2}(\phi) = (0,0)\). If \(\varepsilon_1 \neq \varepsilon_2\), then
\[
\kappa_{\varepsilon_1}(\phi) - \kappa_{\varepsilon_2}(\phi) = \left(\frac{1}{2}, \frac{1}{2}\right) \pmod 1.
\]
Thus in all cases, \(\kappa_{\varepsilon_1}(\phi) - \kappa_{\varepsilon_2}(\phi) \in \mathcal K\). Since \(\mathcal K\) is a subgroup of \(\mathbb T^2\) and every element of \(\mathcal K\) has zero difference frequency, the two centers contribute nothing to \(\rho_{q_1 - q_2}\). Consequently,
\[
\rho_{q_1 - q_2}(\boldsymbol{\theta})
\ll
\operatorname{dist}_{\mathcal K}\bigl((u_1 - u_2, v_1 - v_2)\bigr)
\ll
|u_1 - u_2| + (v_1 - v_2)^2.
\]
Using \(|u_i| \le C_3\eta\) and \(|v_i| \le C_1\eta^{1/2}\), we obtain
\[
|u_1 - u_2| \le 2C_3\eta \ll \eta,
\quad 
(v_1 - v_2)^2 \le (2C_1\eta^{1/2})^2 = 4C_1^2\eta \ll \eta.
\]
Therefore,
\[
\rho_{q_1 - q_2}(\boldsymbol{\theta}) \ll \eta.
\]
This completes the proof.
\end{proof}

The following lemma records a preliminary observation that plays a crucial role in the estimate of \(\operatorname{Var}^{\mathrm{off},1-4}_{\mathcal B}(R_N(f))\).

\begin{lemma}\label{>0}

For every \((h_1, h_2) \in \mathcal B_N\),
\[
|1 - z_{h_1} z_{-h_2}| > 0.
\]
    
\end{lemma}
\begin{proof}
 Suppose for contradiction that \(|1 - z_{h_1} z_{-h_2}| = 0\), i.e.,
\[
z_{h_1} z_{-h_2} = 1.
\]
By definition,
\[
z_{h_1} z_{-h_2}
=
e(h_1\sigma)\cos(2\pi h_1\delta) \cdot e(-h_2\sigma)\cos(2\pi h_2\delta)
=
e\bigl((h_1-h_2)\sigma\bigr)\cos(2\pi h_1\delta)\cos(2\pi h_2\delta).
\]
Hence
\[
e\bigl((h_1-h_2)\sigma\bigr)\cos(2\pi h_1\delta)\cos(2\pi h_2\delta) = 1.
\]
Taking absolute values gives
\[
|\cos(2\pi h_1\delta)| = |\cos(2\pi h_2\delta)| = 1,
\]
so there exist integers \(k_1, k_2\) such that
\[
h_1\delta = \frac{k_1}{2}, \quad h_2\delta = \frac{k_2}{2}.
\]
Substituting these back into the equation yields
\[
e\bigl((h_1-h_2)\sigma\bigr) = 1,
\]
so there exists an integer \(k_3\neq 0\) such that
\[
(h_1-h_2)\sigma = \frac{k_3}{2}.
\]
Thus there is a nontrivial rational linear relation among \(1,\sigma,\delta\), contradicting the assumed linear independence of \(1,\sigma,\delta\) over \(\mathbb{Q}\). 
\end{proof}

\begin{lemma}\label{var:1-4}
We have
\[
\operatorname{Var}^{\mathrm{off},1-4}_{\mathcal B}(R_N(f)) \ll N^{-2/3+12r}.
\]
\end{lemma}

\begin{proof}
Recall that
\begin{align*}
\operatorname{Var}^{\mathrm{off},1-4}_{\mathcal B}(R_N(f))
&\ll
N^{-3+12r}
\sum_{(h_1, h_2) \in \mathcal B_N}
\left|\widehat f\!\left(\frac{h_1}{N}\right)\right|
\left|\widehat f\!\left(\frac{h_2}{N}\right)\right|
\frac{\mathbf{1}_{\{|1 - z_{h_1} z_{-h_2}| \le N^{-r}/2\}}(h_1, h_2)}{|1 - z_{h_1} z_{-h_2}|+N^{-1}}.
\end{align*}
Fix \(h_1\). We estimate the inner sum over \(h_2\),
\[
S(h_1)
:=
\sum_{h_2\colon (h_1, h_2) \in \mathcal B_N}
\left|\widehat f\!\left(\frac{h_2}{N}\right)\right|
\frac{\mathbf{1}_{\{|1 - z_{h_1} z_{-h_2}| \le N^{-r}/2\}}}{|1 - z_{h_1} z_{-h_2}|+N^{-1}}.
\]
For $0<\varepsilon < {1}/{3},$  write $\lambda= {2}/{3}+\varepsilon.$  We split \(S(h_1)\) into two parts, according to whether \(|h_2|\le N^{1+\varepsilon}\) or \(|h_2|>N^{1+\varepsilon}\).

\medskip
\noindent\textbf{Part 1: \(|h_2| \le N^{1+\varepsilon}\).}
By Lemma~\ref{>0}, for  \((h_1, h_2) \in \mathcal B_N\),
\[
 |1 - z_{h_1} z_{-h_2}| > 0.
\]
Hence there exists a unique integer \(k \ge 0\) such that
\[
2^{-(k+1)} \le |1 - z_{h_1} z_{-h_2}| < 2^{-k}<N^{-r}/2 \quad (k \ge r \log_2 N).
\]

For \(j, k \in \mathbb N\), define
\[
\mathcal B_{j,k}(h_1)
:=
\left\{
h_2 \in \mathbb Z \setminus \{0\} :
2^j \le |h_2| < 2^{j+1},\
2^{-(k+1)} \le |1 - z_{h_1} z_{-h_2}|  < 2^{-k}
\right\}.
\]
By Lemma~\ref{df}, for any two distinct \(h_2, h_2' \in \mathcal B_{j,k}(h_1)\),
\[
2^{-j\lambda}\ll|h_2 - h_2'|^{-\lambda}
\ll \rho_{h_2 - h_2'}(\boldsymbol{\theta})
\ll 2^{-k}.
\]
Hence $k\ll j\lambda$ and  \(|h_2 - h_2'| \gg 2^{k/\lambda}\). Since \(\mathcal B_{j,k}(h_1) \subset [-2^{j+1}, 2^{j+1}]\), we obtain
\[
\#  \mathcal B_{j,k}(h_1)
\ll
\frac{2^j}{2^{k/\lambda}} + 1.
\]

We now estimate the inner sum over \(h_2\) with \(|h_2| \le N^{1+\varepsilon}\):
\begin{align*}
S_1(h_1)
&:=
\sum_{\substack{h_2 : (h_1, h_2) \in \mathcal B_N \\ 0 < |h_2| \le N^{1+\varepsilon}}}
\left|\widehat f\!\left(\frac{h_2}{N}\right)\right|
\frac{\mathbf{1}_{\{d(h_1, h_2) \le N^{-r}/2\}}}{d(h_1, h_2) + N^{-1}}\\ 
&\ll
\sum_{j=0}^{\lceil (1+\varepsilon)\log N \rceil}
\sum_{\log N \ll k \ll j\lambda}
\sum_{h_2 \in \mathcal B_{j,k}(h_1)}
\left|\widehat f\!\left(\frac{h_2}{N}\right)\right| \cdot 2^k \\
&\ll
\sum_{j=0}^{\lceil (1+\varepsilon)\log N \rceil}
\sum_{\log N \ll k \ll j\lambda}
\#  \mathcal B_{j,k}(h_1) \cdot 2^k\\
&\ll \sum_{j=0}^{\lceil (1+\varepsilon)\log N \rceil}
\sum_{\log N \ll k \ll j\lambda}
\left( \frac{2^j}{2^{k/\lambda}} + 1 \right) 2^k\\
&\asymp
 N^{(1+\varepsilon)}+N\log N.
\end{align*}

Therefore, for sufficiently large $N,$
\[
S_1(h_1) \ll N^{(1+\varepsilon)\lambda}.
\]

\medskip
\noindent\textbf{Part 2: \(|h_2| > N^{1+\varepsilon}\).}
By the rapid decay of \(\widehat f\),
\[
\left|\widehat f\!\left(\frac{h_2}{N}\right)\right| \ll \left(1+\frac{|h_2|}{N}\right)^{-M}
\]
for every \(M > 0\). Moreover, the factor
\[
\frac{\mathbf{1}_{\{|1 - z_{h_1} z_{-h_2}| \le N^{-r}/2\}}}{|1 - z_{h_1} z_{-h_2}|+N^{-1}}
\]
is bounded by \(N\). Hence
\[
S_2(h_1)
:=
\sum_{\substack{h_2 : (h_1, h_2) \in \mathcal B_N \\ |h_2| > N^{1+\varepsilon}}}
\left|\widehat f\!\left(\frac{h_2}{N}\right)\right|
\frac{\mathbf{1}_{\{|1 - z_{h_1} z_{-h_2}| \le N^{-r}/2\}}}{|1 - z_{h_1} z_{-h_2}|+N^{-1}}
\ll
N \sum_{|h_2| > N^{1+\varepsilon}} \left(1+\frac{|h_2|}{N}\right)^{-M}.
\]
Taking \(M\) sufficiently large (say \(M > 2/\varepsilon\)), the sum is \(\ll_{M} N^{-M\varepsilon+2} \ll 1\).

\medskip
Combining the two parts,
\[
S(h_1) \ll N^{1+\varepsilon}.
\]
Substituting back,
\[
\operatorname{Var}^{\mathrm{off},1-4}_{\mathcal B}(R_N(f))
\ll
N^{-2+\varepsilon+12r}
\sum_{h_1 \in \mathbb Z \setminus \{0\}}
\left|\widehat f\!\left(\frac{h_1}{N}\right)\right|.
\]
By the rapid decay of \(\widehat f\),
\[
\sum_{h_1 \in \mathbb Z \setminus \{0\}}
\left|\widehat f\!\left(\frac{h_1}{N}\right)\right| \ll N.
\]
Therefore,
\begin{equation*}
 \operatorname{Var}^{\mathrm{off},1-4}_{\mathcal B}(R_N(f))
\ll
N^{-1+\varepsilon+12r}\le N^{-2/3+12r}.   
\end{equation*}
\end{proof}

\medskip

\noindent\underline{\textbf{Part 2:}}\quad At least two of the following  are \(\le N^{-r}\):
\[
|1-z_{h_1-h_2}|,\quad |1-z_{h_1}z_{-h_2}|,\quad |1-z_{h_1}|,\quad |1-z_{-h_2}|.
\]

\medskip
\underline{\textbf{Subpart 2-1:}}\quad  \(|1 - z_{h_1} z_{-h_2}|, |1 - z_{h_1-h_2}| \le N^{-r}\), and \(|1 - z_{h_1}|, |1 - z_{-h_2}| \ge N^{-r}\).

In this subpart, the two ``bad'' conditions hold simultaneously, while \(z_{h_1}\) and \(z_{-h_2}\) are both bounded away from \(1\).
Recall that
\[
C_N(h_1,h_2)
=
[t^N]\frac{1}{(1-t)^2}
\left(
\frac{1}{1-z_{h_1-h_2}t}
-
\frac{1}{1-z_{h_1}z_{-h_2}t}
\right)
\left(
\frac{1}{1-z_{h_1}t}
+
\frac{1}{1-z_{-h_2}t}
-
1
\right)^2.
\]
For \(\varpi\in\{z_{h_1-h_2},z_{h_1}z_{-h_2}\}\), we define
\[
C_N(h_1,h_2;\varpi)
:=
\left|
[t^N]\frac{1}{(1-t)^2}
\frac{1}{1-\varpi t}
\left(
\frac{1}{1-z_{h_1}t}
+
\frac{1}{1-z_{-h_2}t}
-
1
\right)^2
\right|.
\]
Hence
\[
|C_N(h_1,h_2)|
\le
C_N(h_1,h_2;z_{h_1-h_2})
+
C_N(h_1,h_2;z_{h_1}z_{-h_2}).
\]
It suffices to bound the two terms on the right-hand side of the last inequality. To this end, we treat these two terms as in Subparts 1-3 and 1-4, respectively, and then apply \eqref{var:1-3} and Lemma~\ref{var:1-4} to obtain
\[
\operatorname{Var}^{\mathrm{off},2-1}_{\mathcal{B}}(R_N(f))
\ll
N^{-1/3+5r} + N^{-2/3+12r}
\ll
N^{-1/3+12r}.
\]

\medskip
\underline{\textbf{Subpart 2-2:}}\quad  Remaining cases.

To treat the remaining cases of Part 2 in a unified manner, we need to understand how the following four quantities interact:
\[
|1-z_{h_1-h_2}|,\quad
|1-z_{h_1}z_{-h_2}|,\quad
|1-z_{h_1}|,\quad
|1-z_{-h_2}|.
\]
The next lemma shows that, apart from the configuration treated in Subpart 2-1, if any two of these quantities are small, then so are the other two.
\begin{lemma}\label{4b}
Let \(h_1,h_2\in\mathbb Z\setminus\{0\}\) with \(h_1\neq h_2\). Define
\[
d_1:=|1-z_{h_1-h_2}|,\quad
d_2:=|1-z_{h_1}z_{-h_2}|,\quad
d_3:=|1-z_{h_1}|,\quad
d_4:=|1-z_{-h_2}|.
\]
Suppose at least two of the \(d_i\) are \(\le N^{-r}\). Then all four are \(\ll N^{-r}\), except in the case where \(d_1,d_2\le N^{-r}\) and \(d_3,d_4\ge N^{-r}\).
\end{lemma}
 
\begin{proof} We say that a point \(z\) is bad if \(|1-z|\ll N^{-r}\). The proof of the lemma proceeds by examining the following cases.

\medskip

\noindent\textbf{(1)}\quad \(z_{h_1}\) and \(z_{-h_2}\) bad \(\Longrightarrow\) \(z_{h_1-h_2}\) and \(z_{h_1} z_{-h_2}\) bad.

Since
\[
(h_1 - h_2)\boldsymbol{\theta} = h_1\boldsymbol{\theta} + (-h_2)\boldsymbol{\theta},
\]
and \(\mathcal K + \mathcal K = \mathcal K\), the local difference-frequency estimate gives
\[
|1 - z_{h_1-h_2}| \ll N^{-r}.
\]
Moreover,
\[
1 - z_{h_1} z_{-h_2} = (1 - z_{h_1}) + z_{h_1}(1 - z_{-h_2}),
\]
so
\[
|1 - z_{h_1} z_{-h_2}| \le |1 - z_{h_1}| + |1 - z_{-h_2}| \ll N^{-r}.
\]

\medskip
\noindent\textbf{(2)}\quad  \(z_{h_1}\) and \(z_{h_1-h_2}\) bad \(\Longrightarrow\) \(z_{-h_2}\) and \(z_{h_1} z_{-h_2}\) bad.

Since
\[
-h_2\boldsymbol{\theta} = (h_1 - h_2)\boldsymbol{\theta} - h_1\boldsymbol{\theta},
\]
and \(\mathcal K - \mathcal K = \mathcal K\), the local difference-frequency estimate gives
\[
|1 - z_{-h_2}| \ll N^{-r}.
\]
And \(|1 - z_{h_1} z_{-h_2}| \ll N^{-r}\) follows as in \textbf{(1)}.

\medskip
\noindent\textbf{(3)}\quad  \(z_{-h_2}\) and \(z_{h_1-h_2}\) bad \(\Longrightarrow\) \(z_{h_1}\) and \(z_{h_1} z_{-h_2}\) bad.

This is symmetric to \textbf{(2)}, with \(h_1\) and \(-h_2\) interchanged.

\medskip
\noindent\textbf{(4)}\quad  \(z_{h_1}\) and \(z_{h_1} z_{-h_2}\) bad \(\Longrightarrow\) \(z_{-h_2}\) and \(z_{h_1-h_2}\) bad.

We use the identity
\[
z_{h_1} z_{-h_2} - z_{h_1} = z_{h_1}(z_{-h_2} - 1).
\]
Since \(|1 - z_{h_1}| \ll N^{-r}\), we have \(|z_{h_1}| \ge 1/2\) after decreasing the implicit constant in \(N^{-r}\) if necessary. Hence
\[
|1 - z_{-h_2}|
=
\frac{|z_{h_1} z_{-h_2} - z_{h_1}|}{|z_{h_1}|}
\le
2\bigl(|1 - z_{h_1} z_{-h_2}| + |1 - z_{h_1}|\bigr)
\ll N^{-r}.
\]
Then \textbf{(1)} gives \(|1 - z_{h_1-h_2}| \ll N^{-r}\).

\medskip
\noindent\textbf{(5)}\quad  \(z_{-h_2}\) and \(z_{h_1} z_{-h_2}\) bad \(\Longrightarrow\) \(z_{h_1}\) and \(z_{h_1-h_2}\) bad.

This is symmetric to \textbf{(4),} with \(h_1\) and \(-h_2\) interchanged.
\end{proof}

\medskip
By Lemma~\ref{4b}, the points \(h_1\boldsymbol{\theta}\) and \(-h_2\boldsymbol{\theta}\) both lie in a cusp neighborhood of \(1\). Hence we may write
\begin{equation}\label{lco}
h_1\boldsymbol{\theta}
=
\mathbf{k}_1+\left(u_1(h_1),\frac{v_1(h_1)}{2}\right),
\qquad
-h_2\boldsymbol{\theta}
=
\mathbf{k}_2+\left(u_2(h_2),\frac{v_2(h_2)}{2}\right),
\end{equation}
where the local scale is chosen sufficiently small. For notational simplicity, for \(i=1,2\), we abbreviate \(u_i(h_i)\) and \(v_i(h_i)\) by \(u_i\) and \(v_i\), respectively.

\begin{lemma}\label{lco-1}
In the local coordinates \eqref{lco}, for  \(1 \le n \le N\),
\[
\boxed{
|z_{h_1-h_2}^n - (z_{h_1} z_{-h_2})^n|
\ll
\frac{|v_1v_2|}{s_1 + s_2 + N^{-1}}.
}
\]
\end{lemma}

\begin{proof}
By the difference-of-powers identity,
\[
z_{h_1-h_2}^n - (z_{h_1} z_{-h_2})^n
=
\bigl(z_{h_1-h_2} - z_{h_1} z_{-h_2}\bigr)
\sum_{k=0}^{n-1} z_{h_1-h_2}^{n-1-k} (z_{h_1} z_{-h_2})^k.
\]
By \eqref{lco},
\[
z_{h_1-h_2} = e(u_1 + u_2)\cos(\pi(v_1 - v_2)),
\quad 
z_{h_1} z_{-h_2} = e(u_1 + u_2)\cos(\pi v_1)\cos(\pi v_2).
\]
Hence
\[
z_{h_1-h_2} - z_{h_1} z_{-h_2}
=
e(u_1 + u_2)\bigl[\cos(\pi(v_1 - v_2)) - \cos(\pi v_1)\cos(\pi v_2)\bigr].
\]
Using the identity
\[
\cos(\pi(v_1 - v_2)) - \cos(\pi v_1)\cos(\pi v_2) = \sin(\pi v_1)\sin(\pi v_2),
\]
and \(|\sin(\pi x)| \le \pi|x|\), we obtain
\[
|z_{h_1-h_2} - z_{h_1} z_{-h_2}|
\ll |v_1 v_2|.
\]

If \(v_1 v_2 = 0\), then \(z_{h_1-h_2} = z_{h_1} z_{-h_2}\), which yields the desired result.

If \(v_1 v_2 > 0\), then
\[
|z_{h_1-h_2}^n - (z_{h_1} z_{-h_2})^n|
\ll |v_1 v_2| \sum_{k=0}^{n-1} |z_{h_1} z_{-h_2}|^k
\ll |v_1 v_2| \min\left\{\frac{1}{1 - |z_{h_1} z_{-h_2}|}, n\right\}.
\]
Since
\[
1 - |z_{h_1} z_{-h_2}|
=
1 - |\cos(\pi v_1)\cos(\pi v_2)|
=
1 - |\cos(\pi v_1)| \cdot |\cos(\pi v_2)|,
\]
and \(1 - |\cos(\pi v_i)| \asymp v_i^2 \) for \(i = 1, 2\), we obtain
\[
1 - |z_{h_1} z_{-h_2}| \asymp v_1^{2} + v_2^{2}.
\]
Therefore,
\[
|v_1 v_2|\min\left\{\frac{1}{1 - |z_{h_1} z_{-h_2}|}, n\right\}
\ll
\frac{|v_1 v_2|}{v_1^{2} + v_2^{2} + n^{-1}}
\le
\frac{|v_1 v_2|}{v_1^{2} + v_2^{2} + N^{-1}}.
\]
\end{proof}

Recall that
\[
\mathscr H
=
\{(2,0),(0,2),(1,1),(1,0),(0,1),(0,0)\}.
\]

\begin{lemma}
\label{bad-ab}
For $(a,b)\in\mathscr{H}$, we have
\[\sum_{n=0}^{N}\left|
[t^n]
\frac{1}
{(1-t)^2(1-z_{h_1}t)^a(1-z_{-h_2}t)^b}
\right|
\ll
\frac{N^{2}}{(|1-z_{h_1}|+N^{-1})^{a}(|1-z_{-h_2}|+N^{-1})^{b}}\]
\end{lemma}

\begin{proof}
If $(a,b)=(0,0)$
\[\sum_{n=0}^{N}\left|
[t^n]
\frac{1}
{(1-t)^2}
\right| \ll \sum_{n=0}^{N}  n \asymp N^{2}.\]

If $(a,b)=(1,0),$ (\ref{X}) gives
\[\sum_{n=0}^{N}\left|
[t^n]
\frac{1}
{(1-t)^2(1-z_{h_1}t)}
\right|
\ll \sum_{n=0}^{N}
\frac{N}{(|1-z_{h_1}|+N^{-1})}\asymp \frac{N^{2}}{(|1-z_{h_1}|+N^{-1})}.\]
By symmetry, the same bound holds for \((a,b) = (0,1)\), with \(z_{h_1}\) replaced by \(z_{-h_2}\).

If \((a,b) = (1,1)\), the geometric series expansions give
\[
\frac{1}{(1-t)^2} = \sum_{k \ge 0} (k+1) t^k,
\quad 
\frac{1}{1 - z_{h_1} t} = \sum_{k \ge 0} z_{h_1}^k t^k,
\quad 
\frac{1}{1 - z_{-h_2} t} = \sum_{k \ge 0} z_{-h_2}^k t^k.
\]
Multiplying the three power series and extracting the coefficient of \(t^n\),
\[
[t^n]
\frac{1}{(1-t)^2(1 - z_{h_1} t)(1 - z_{-h_2} t)}
=
\sum_{\substack{j, k, l \ge 0 \\ j + k + l = n}}
(j+1) \, z_{h_1}^k \, z_{-h_2}^l.
\]
Reorganizing the sum in terms of the partial sums
\[
S_j(z_{h_1}) := \sum_{k=0}^{j} z_{h_1}^k,
\quad 
S_j(z_{-h_2}) := \sum_{l=0}^{j} z_{-h_2}^l,
\]
we obtain
\[
[t^n]
\frac{1}{(1-t)^2(1 - z_{h_1} t)(1 - z_{-h_2} t)}
=
\sum_{j=0}^{n} S_j(z_{h_1}) S_{n-j}(z_{-h_2}).
\]
Since the partial sums satisfy
\[
|S_j(z_{h_1})| \ll \frac{1}{|1 - z_{h_1}| + N^{-1}},
\quad 
|S_j(z_{-h_2})| \ll \frac{1}{|1 - z_{-h_2}| + N^{-1}}
\]
for \(0 \le j \le N\), we have that 
\[
\left|
[t^n]
\frac{1}{(1-t)^2(1 - z_{h_1} t)(1 - z_{-h_2} t)}
\right|
\ll
\frac{n}{(|1 - z_{h_1}| + N^{-1})(|1 - z_{-h_2}| + N^{-1})}.
\]
Summing over \(n = 0, \dots, N\) and using \(\sum_{n=0}^{N}(n+1) \asymp N^2\), we conclude
\[
\sum_{n=0}^{N}
\left|
[t^n]
\frac{1}{(1-t)^2(1 - z_{h_1} t)(1 - z_{-h_2} t)}
\right|
\ll
\frac{N^2}{(|1 - z_{h_1}| + N^{-1})(|1 - z_{-h_2}| + N^{-1})}.
\]

If $(a,b)=(2,0),$ Lemma \ref{XY}(i) gives 
\[\sum_{n=0}^{N}
\left|
[t^n]
\frac{1}{(1-t)^2(1 - z_{h_1} t)^{2}}
\right|
\ll\frac{N^{2}}{(|1-z_{h_1}|+N^{-1})^{2}},\]
and, by symmetry, the same bound holds for \((a,b)=(0,2)\) with \(z_{h_1}\) replaced by \(z_{-h_2}\).  
\end{proof}

\begin{lemma}\label{cov-bad}
We have
\begin{multline*}
C_N(h_1, h_2)
\ll
\frac{N^2 |v_1v_2|}{v_1^{2} + v_2^{2} + N^{-1}}
\biggl[
\frac{1}{(|1 - z_{h_1}| + N^{-1})^2}
+
\frac{1}{(|1 - z_{-h_2}| + N^{-1})^2}
\\
+
\frac{1}{(|1 - z_{h_1}| + N^{-1})(|1 - z_{-h_2}| + N^{-1})}
+
\frac{1}{|1 - z_{h_1}| + N^{-1}}
+
\frac{1}{|1 - z_{-h_2}| + N^{-1}}
\biggr].
\end{multline*}
\end{lemma}

\begin{proof}
For $(a,b)\in\mathscr{H}$, by Lemmas \ref{lco-1} and \ref{bad-ab}, we have
\begin{align*}
C_N(h_1, h_2; a, b)
~\le~&
\sup_{1 \le n \le N}
\left|
[t^n]
\left(
\frac{1}{1 - z_{h_1-h_2} t}
-
\frac{1}{1 - z_{h_1} z_{-h_2} t}
\right)
\right|
\\
&\quad 
\cdot
\sum_{n=0}^{N}
\left|
[t^n]
\frac{1}{(1-t)^2(1 - z_{h_1} t)^a (1 - z_{-h_2} t)^b}
\right|
\\
\ll~&
\frac{N^2 |v_1v_2|}{(v_1^{2} v_2^{2}+ N^{-1})(|1 - z_{h_1}| + N^{-1})^a (|1 - z_{-h_2}| + N^{-1})^b}.
\end{align*}
\end{proof}

We are now in a position to estimate the contribution of Subpart 2-2 to the off-diagonal variance:
\[
\operatorname{Var}^{\mathrm{off},2-2}_{\mathcal{B}}(R_N(f))
=
\frac{1}{N^4}
\sideset{}{^*}\sum_{(h_1,h_2)\in\mathcal{B}_N}
\widehat{f}\!\left(\frac{h_1}{N}\right)
\widehat{f}\!\left(\frac{h_2}{N}\right)
C_N(h_1,h_2),
\] 
 here and below, a starred summation indicates that the pair \((h_1,h_2)\) is required to satisfy the additional restrictions of Subpart 2-2. 

First, with Lemma~\ref{cov-bad} at our disposal, we use the coarse bound 
\[
C_N(h_1, h_2) \ll N^4.
\]
Together with the Schwartz decay of \(\widehat f\), this implies that the variance contribution from pairs satisfying
\[
|h_1| > N^{1+\varepsilon}
\quad \text{or}\quad 
|h_2| > N^{1+\varepsilon}
\]
is \(O_{f}(N^{-1})\).   
It remains to estimate
\[
\widetilde{\operatorname{Var}}^{\mathrm{off},2-2}_{\mathcal B}(R_N(f))
=
\frac{1}{N^4}
\sideset{}{^*}\sum_{\substack{(h_1, h_2) \in \mathcal B_N\\
1\le |h_1|,|h_2|\le N^{1+\varepsilon}}}
\widehat f\!\left(\frac{h_1}{N}\right)
\widehat f\!\left(\frac{h_2}{N}\right)
C_N(h_1, h_2).
\]

For \(k, l \in \mathbb N\) and \(i \in \{1, 2\}\), define
\begin{align*}
\Omega_{k,l}^{(i)}
&:=
\bigl\{
h_i \in \mathbb Z \setminus \{0\} :
1 \le |h_i| \le N^{1+\varepsilon},\
|1 - z_{h_i}| < N^{-r},
\\
&\quad \quad \quad \quad 
2^{-(l+1)} \le \rho_{h_i}(\boldsymbol{\theta}) < 2^{-l},\
2^{-(k+1)} \le v_i^2 \le 2^{-k}
\bigr\}.
\end{align*}

The following lemma provides upper bounds for the cardinalities of the these sets.

\begin{lemma}
\label{Omega-b}
For  \(k, l \in \mathbb N\) and \(i \in \{1, 2\}\),
\[
\#  \Omega_{k,l}^{(i)}
\ll
\min\left\{
1 + \frac{N^{1+\varepsilon}}{2^{l/\lambda}},
\left(1 + \frac{2^{-l}}{N^{-\lambda(1+\varepsilon)}}\right)
\left(1 + \frac{2^{-k/2}}{N^{-\lambda(1+\varepsilon)/2}}\right)
\right\}.
\]
\end{lemma}

\begin{proof}

For \(p, q \in \Omega_{k,l}^{(i)}\) with \(p \neq q\), the definition of \(\rho_h(\boldsymbol{\theta})\) yields
\[
|p - q|^{-\lambda}
\ll \rho_{p - q}(\boldsymbol{\theta})
\ll \rho_p(\boldsymbol{\theta}) + \rho_q(\boldsymbol{\theta})
< 2^{-(l-1)}.
\]
Hence
\[
|p - q| \gg 2^{l/\lambda}.
\]
It follows that
\[
\#  \Omega_{k,l}^{(i)}
\ll 1 + \frac{N^{1+\varepsilon}}{2^{l/\lambda}}.
\]

Further, by Theorem~\ref{thm:local} and \eqref{loc-asy}, we have
\[
\rho_{h_i}(\boldsymbol{\theta})
\asymp |1 - z_{h_i}|
\asymp |u_i| + |v_i|^2.
\]
This implies that for \(h_i \in \Omega_{k,l}^{(i)}\),
\[
0 \le |u_i| \ll 2^{-l},
\quad 
|v_i| \asymp 2^{-k/2}.
\]
Put \(d_N := N^{-\lambda(1+\varepsilon)}\), and partition the region into half-open cells of side lengths \(d_N\) and \(d_N^{1/2}\) in the \(u\)- and \(v\)-directions, respectively. The number of such cells is at most
\[
\ll
\left(1 + \frac{2^{-l}}{N^{-\lambda(1+\varepsilon)}}\right)
\left(1 + \frac{2^{-k/2}}{N^{-\lambda(1+\varepsilon)/2}}\right).
\]

We now show that each cell contains at most one point from \(  \Omega_{k,l}^{(i)}\). Suppose for contradiction that two distinct integers \(p, q \in \Omega_{k,l}^{(i)}\) occupy the same cell. Then
\[
|u(p) - u(q)| < d_N,
\quad 
|v(p) - v(q)| < d_N^{1/2}.
\]
By the local cusp comparison (Theorem~\ref{thm:local}),
\[
\rho_{p - q}(\boldsymbol{\theta})
\asymp |u(p) - u(q)| + |v(p) - v(q)|^2
\le C d_N
\]
for some absolute constant \(C > 0\). Since \(d_N = N^{-\lambda(1+\varepsilon)}\), for \(N\) sufficiently large we have
\[
C d_N < c_{\boldsymbol{\theta},\varepsilon} N^{-\lambda},
\]
and hence
\[
\rho_{p - q}(\boldsymbol{\theta}) < c_{\boldsymbol{\theta},\varepsilon} N^{-\lambda}.
\]
On the other hand, by definition of \(\rho_{p-q}\),
\[
\rho_{p - q}(\boldsymbol{\theta}) \ge c_{\boldsymbol{\theta},\varepsilon} N^{-\lambda},
\]
a contradiction.

\end{proof}

For \(h_i \in \Omega_{k,l}^{(i)}\), the defining conditions give
\[
|1 - z_{h_i}| < N^{-r},
\quad 
2^{-(l+1)} \le \rho_{h_i}(\boldsymbol{\theta}) < 2^{-l},
\quad 
2^{-(k+1)} \le v_i^2 \le 2^{-k}.
\]
By the local comparison (Theorem~\ref{thm:local}),
\[
|1 - z_{h_i}| \asymp \rho_{h_i}(\boldsymbol{\theta})  \gg N^{-\lambda}.
\]
Combining this with \(\rho_{h_i}(\boldsymbol{\theta}) < 2^{-l}\) and \(|1 - z_{h_i}| < N^{-r}\), we obtain
\[
N^{-\lambda} \ll 2^{-l} \ll N^{-r}.
\]

Similarly, since \(\rho_{h_i}(\boldsymbol{\theta}) \asymp |u_i| + v_i^2 \ge v_i^2\), we have
\[
v_i^2 \ll 2^{-l_i},
\]
and therefore
\[
2^{-k}\ll 2^{-l_i}\ll N^{-r}.
\]

We now estimate
\[
S_N := \frac{1}{N^2}
\sum_{\substack{(h_1, h_2) \in \mathcal B_N \\ 1 \le |h_1|, |h_2| \le N^{1+\varepsilon}}}
\widehat f\!\left(\frac{h_1}{N}\right)
\widehat f\!\left(\frac{h_2}{N}\right)
\frac{ |v_1 v_2|}{v_1^2 + v_2^2 + N^{-1}}
\frac{1}{(|1 - z_{h_1}| + N^{-1})^2}.
\]

\begin{lemma}\label{var:2-2}
We have 
\[S_N\ll N^{-\min\{{r}/{2}-8\varepsilon, \, {2}/{3}-3\varepsilon\}}.\]
\end{lemma}

\begin{proof}
We deduce that
\begin{align*}
S_N
&\le
\frac{1}{N^2}
\sum_{\substack{(h_1, h_2) \in \mathcal B_N \\ 1 \le |h_1|, |h_2| \le N^{1+\varepsilon}}}
\frac{|v_1 v_2|}{v_1^2 + v_2^2 + N^{-1}}
\frac{1}{|1 - z_{h_1}|^2}
\\
&\ll
\frac{1}{N^2}
\sum_{\substack{(h_1, h_2) \in \mathcal B_N \\ 1 \le |h_1|, |h_2| \le N^{1+\varepsilon}}}
\frac{|v_1 v_2|}{v_1^2 + v_2^2 + N^{-1}}
\frac{1}{\rho_{h_1}(\boldsymbol{\theta})^2},
\end{align*}
where the second step uses \(\rho_{h_1}(\boldsymbol{\theta}) \asymp |1 - z_{h_1}|\) (Theorem~\ref{thm:local}).

We now decompose the sum into dyadic layers according to
\[
\rho_{h_i}(\boldsymbol{\theta}) \asymp 2^{-l_i},
\quad 
v_i^2 \asymp 2^{-k_i},
\quad i = 1, 2,
\]
with the ranges
\[
N^{-\lambda} \ll 2^{-l_i} \ll N^{-r},
\quad 
2^{-k_i} \ll 2^{-l_i}.
\]
Thus
\begin{align*}
S_N
&\ll
\frac{1}{N^2}
\sum_{\substack{l_1, l_2 \\ N^{-\lambda} \ll 2^{-l_i} \ll N^{-r}}}
\sum_{\substack{k_1, k_2 \\ 2^{-k_i} \ll 2^{-l_i}}}
\#  \Omega_{k_1, l_1}^{(1)}
\cdot
\#  \Omega_{k_2, l_2}^{(2)}
\cdot
\frac{(2^{-k_1} 2^{-k_2})^{1/2}}{2^{-k_1} + 2^{-k_2} + N^{-1}}
\cdot
2^{2 l_1}.
\end{align*}

By Lemma~\ref{Omega-b}, the cardinalities satisfy
\[
\#  \Omega_{k_i, l_i}^{(i)}
\ll
\min\left\{
1 + \frac{N^{1+\varepsilon}}{2^{l_i/\lambda}},
\left(1 + \frac{2^{-l_i}}{N^{-\lambda(1+\varepsilon)}}\right)
\left(1 + \frac{2^{-k_i/2}}{N^{-\lambda(1+\varepsilon)/2}}\right)
\right\}.
\]
Substituting these bounds into the above expression yields
\begin{align*}
S_N
&\ll
\frac{1}{N^2}
\sum_{\substack{l_1, l_2 \\ N^{-\lambda} \ll 2^{-l_i} \ll N^{-r}}}
\sum_{\substack{k_1, k_2 \\ 2^{-k_i} \ll N^{-r}}}
\left(1 + \frac{N^{1+\varepsilon}}{2^{l_1/\lambda}}\right)
\\
&\quad \quad 
\times
\left(1 + \frac{2^{-l_2}}{N^{-\lambda(1+\varepsilon)}}\right)
\left(1 + \frac{2^{-k_2/2}}{N^{-\lambda(1+\varepsilon)/2}}\right)
\cdot
\frac{(2^{-k_1} 2^{-k_2})^{1/2}}{2^{-k_1} + 2^{-k_2} + N^{-1}}
\cdot
2^{2 l_1}.
\end{align*}
Expanding the product of the three brackets, we obtain a sum of \(2\times2\times2=8\) terms. We estimate each term separately. We list the eight terms in increasing binary order, where the first bracket corresponds to the least significant bit.
\medskip

\noindent \textbf{Term 1.} We now estimate the first term \(S_{N,1}\): 

\begin{align*}
S_{N,1}
&:=
\frac{1}{N^2}
\sum_{\substack{l_1, l_2 \\ N^{-\lambda} \ll 2^{-l_i} \ll N^{-r}}}
\sum_{\substack{k_1, k_2 \\ 2^{-k_i} \ll 2^{-l_i}}}
\frac{(2^{-k_1} 2^{-k_2})^{1/2}}{2^{-k_1} + 2^{-k_2} + N^{-1}}
\cdot
2^{2 l_1}.
\end{align*}

We first sum over \(l_1\) and \(l_2\). Since the range of each \(l_i\) has length \(\asymp \log N\) and the summand depends only on \(l_1\), we have
\[
\sum_{l_1} 2^{2 l_1} \asymp N^{2\lambda}.
\]
 The sum over \(l_2\) contributes a factor \(\asymp \log N\). Hence
\[
\sum_{\substack{l_1, l_2 \\ N^{-\lambda} \ll 2^{-l_i} \ll N^{-r}}} 2^{2 l_1} \asymp N^{2\lambda} \log N.
\]
Substituting this into \(S_{N,1}\) and extracting the factor \(N^{-2}\), we obtain
\begin{align*}
S_{N,1}
&\asymp
\frac{\log N}{N^{2-2\lambda}}
\sum_{\substack{k_1, k_2 \\ 2^{-k_i} \ll 2^{-l_i}}}
\frac{(2^{-k_1} 2^{-k_2})^{1/2}}{2^{-k_1} + 2^{-k_2} + N^{-1}}.
\end{align*}
It remains to estimate the sum over \(k_1, k_2\), which we shall do below.

\medskip
\noindent\textbf{Sublemma 1.}
\[
\sum_{\substack{k_1, k_2 \\ 2^{-k_i} \ll 2^{-l_i}}}
\frac{(2^{-k_1} 2^{-k_2})^{1/2}}{2^{-k_1} + 2^{-k_2} + N^{-1}}
\ll (\log N)^{2}.
\]

\begin{proof}
We split the sum into three cases according to whether both, exactly one, or neither of \(2^{-k_1}, 2^{-k_2}\) exceeds \(N^{-1}\).

\medskip
\noindent{Case 1:} \(2^{-k_1} > N^{-1}\) and \(2^{-k_2} > N^{-1}\).

In this case, the elementary inequality
\(
(ab)^{1/2} \le a + b
\)
gives
\[
\sum_{\substack{k_1, k_2 \\ N^{-1}<2^{-k_i} \ll 2^{-l_i} }}
\frac{(2^{-k_1} 2^{-k_2})^{1/2}}{2^{-k_1} + 2^{-k_2} + N^{-1}}
\ll
\sum_{k_1 \asymp \log N} \sum_{k_2 \asymp \log N} 1
\asymp (\log N)^{2},
\]
where \(k_2 \asymp \log N\) since
\(N^{-1} < 2^{-k_i} \ll 2^{-l_i} \ll N^{-r}.\)

\medskip
\noindent{Case 2:} exactly one of \(2^{-k_1}, 2^{-k_2}\) is \(> N^{-1}\).

Without loss of generality, assume \(2^{-k_1} > N^{-1}\).
\[\sum_{k_1: N^{-1}<2^{-k_1}\ll 2^{-l_1}}~ \sum_{k_2: 2^{-k_2}<N^{-1}} \frac{(2^{-k_1} 2^{-k_2})^{1/2}}{2^{-k_1} + 2^{-k_2} + N^{-1}} \le \sum_{k_1: N^{-1}<2^{-k_1}\ll N^{-r}}~ \sum_{k_2: 2^{-k_2}<N^{-1}}  2^{k_1/2}2^{-k_2/2}\asymp 1.\]

\medskip
\noindent{Case 3:} \(2^{-k_1}  \le  N^{-1}\) and \(2^{-k_2} \le  N^{-1}\).
\[\sum_{\substack{k_1, k_2 \\ 2^{-k_i} \le N^{-1} }}\frac{(2^{-k_1} 2^{-k_2})^{1/2}}{2^{-k_1} + 2^{-k_2} + N^{-1}} \le N \sum_{\substack{k_1, k_2 \\ 2^{-k_i} \le N^{-1} }} (2^{-k_1} 2^{-k_2})^{1/2} \asymp 1. \]
\end{proof}

Thus, by  {Sublemma 1}, 
\[S_{N,1}\ll \frac{(\log N)^{3}}{N^{2-2\lambda}}.\]

\noindent \textbf{Term 2.} We estimate  
\begin{align*}
 S_{N,2}:=\frac{1}{N^2}
\sum_{\substack{l_1, l_2 \\ N^{-\lambda} \ll 2^{-l_i} \ll N^{-r}}}
\sum_{\substack{k_1, k_2 \\ 2^{-k_i} \ll 2^{-l_{i}}}}
\frac{N^{1+\varepsilon}}{2^{l_1/\lambda}}\cdot\frac{(2^{-k_1} 2^{-k_2})^{1/2}}{2^{-k_1} + 2^{-k_2} + N^{-1}}
\cdot
2^{2 l_1}.
\end{align*}
Since \(\lambda = 2/3 + \varepsilon > 1/2\), we have \(2 - 1/\lambda > 0\), and
\[
\sum_{l_1: N^{-\lambda} \ll 2^{-l_1} \ll N^{-r}} N^{1+\varepsilon} 2^{l_1(2 - 1/\lambda)}
\asymp
N^{1+\varepsilon} \cdot N^{2\lambda - 1}
=
N^{2\lambda + \varepsilon}.
\]
The sum over \(l_2\) contributes a factor \(\asymp \log N\), and by  {Sublemma 1},
\[
S_{N,2}
\ll
\frac{1}{N^2}
\cdot
N^{2\lambda + \varepsilon}
\cdot
\log N
\cdot
(\log N)^{2}
=
\frac{(\log N)^3}{N^{2 - 2\lambda - \varepsilon}}.
\]

\noindent \textbf{Term 3.} By {Sublemma 1}, we have 
\begin{align*}
 S_{N,3}&:=\frac{1}{N^2}
\sum_{\substack{l_1, l_2 \\ N^{-\lambda} \ll 2^{-l_i} \ll N^{-r}}}
\sum_{\substack{k_1, k_2 \\ 2^{-k_i} \ll 2^{-l_{i}}}}
\frac{2^{-l_2}}{N^{-\lambda(1+\varepsilon)}}\cdot\frac{(2^{-k_1} 2^{-k_2})^{1/2}}{2^{-k_1} + 2^{-k_2} + N^{-1}}
\cdot
2^{2 l_1}\\
&\ll \frac{(\log N)^{2}}{N^{2+r-3\lambda-\lambda\varepsilon}}.
\end{align*}

\noindent \textbf{Term 4.} By {Sublemma 1}, we have
\begin{align*}
 S_{N,4}&:=\frac{1}{N^2}
\sum_{\substack{l_1, l_2 \\ N^{-\lambda} \ll 2^{-l_i} \ll N^{-r}}}
\sum_{\substack{k_1, k_2 \\ 2^{-k_i} \ll 2^{-l_{i}}}}
\frac{2^{-k_2/2}}{N^{-\lambda(1+\varepsilon)/2}}\cdot\frac{(2^{-k_1} 2^{-k_2})^{1/2}}{2^{-k_1} + 2^{-k_2} + N^{-1}}
\cdot
2^{2 l_1}\\
&\ll \frac{(\log N)^{3}}{N^{2+\frac{r}{2}-\frac{5\lambda}{2}-\frac{\lambda\varepsilon}{2}}}.
\end{align*}

\noindent \textbf{Term 5.}  {Sublemma 1} yields 
\begin{align*}
 S_{N,5}&:=\frac{1}{N^2}
\sum_{\substack{l_1, l_2 \\ N^{-\lambda} \ll 2^{-l_i} \ll N^{-r}}}
\sum_{\substack{k_1, k_2 \\ 2^{-k_i} \ll 2^{-l_{i}}}}
\frac{N^{(1+\lambda)(1+\varepsilon)}}{2^{l_2}2^{l_1/\lambda}}\cdot\frac{(2^{-k_1} 2^{-k_2})^{1/2}}{2^{-k_1} + 2^{-k_2} + N^{-1}}
\cdot
2^{2 l_1}\\
&\ll \frac{(\log N)^{4}}{N^{2+r-3\lambda-(\lambda+1)\varepsilon}}.
\end{align*}

\noindent \textbf{Term 6.}  {Sublemma 1} gives 
\begin{align*}
 S_{N,6}&:=\frac{1}{N^2}
\sum_{\substack{l_1, l_2 \\ N^{-\lambda} \ll 2^{-l_i} \ll N^{-r}}}
\sum_{\substack{k_1, k_2 \\ 2^{-k_i} \ll 2^{-l_{i}}}}
\frac{N^{(1+\frac{\lambda}{2})(1+\varepsilon)}}{2^{k_2/2}2^{l_1/\lambda}}\cdot\frac{(2^{-k_1} 2^{-k_2})^{1/2}}{2^{-k_1} + 2^{-k_2} + N^{-1}}
\cdot
2^{2 l_1}\\
&\ll \frac{(\log N)^{3}}{N^{2+\frac{r}{2}-\frac{5\lambda}{2}-(1+\frac{\lambda}{2})\varepsilon}}.
\end{align*}

\noindent \textbf{Term 7.} We estimate
\begin{align*}
 S_{N,7}&:=\frac{1}{N^2}
\sum_{\substack{l_1, l_2 \\ N^{-\lambda} \ll 2^{-l_i} \ll N^{-r}}}
\sum_{\substack{k_1, k_2 \\ 2^{-k_i} \ll 2^{-l_{i}}}}
\frac{N^{\frac{3\lambda}{2}(1+\varepsilon)}}{2^{k_2/2}2^{l_2}}\cdot\frac{(2^{-k_1} 2^{-k_2})^{1/2}}{2^{-k_1} + 2^{-k_2} + N^{-1}}\cdot 2^{2l_1}.
\end{align*}

\medskip
\noindent\textbf{Sublemma 2.}
\[
\sum_{\substack{k_1, k_2 \\ 2^{-k_i} \ll 2^{-l_{i}}}}
\frac{2^{-k_1/2} 2^{-k_2}}{2^{-k_1} + 2^{-k_2} + N^{-1}}
\ll \log N(2^{-l_1/2}+N^{-1}).
\]

\begin{proof}

As before, we split the sum into three cases.

\medskip
\noindent{Case 1:} \(2^{-k_1}> N^{-1}\) and \(2^{-k_2} > N^{-1}\).
We deduce 
\begin{align*}
  \sum_{\substack{k_1, k_2 \\ 2^{-k_i} \ll 2^{-l_{i}}}}
\frac{2^{-k_1/2} 2^{-k_2}}{2^{-k_1} + 2^{-k_2} + N^{-1}}
&\le \sum_{\substack{k_1\le k_2 \\ 2^{-k_i} \ll 2^{-l_{i}}}}
\frac{2^{-k_1/2} 2^{-k_2}}{2^{-k_1} + 2^{-k_2} + N^{-1}}+\sum_{\substack{k_1\ge k_2 \\ 2^{-k_i} \ll 2^{-l_{i}}}}
\frac{2^{-k_1/2} 2^{-k_2}}{2^{-k_1} + 2^{-k_2} + N^{-1}}\\
&\ll \sum_{\substack{k_1\le k_2 \\ 2^{-k_i} \ll 2^{-l_{i}}}} 2^{k_1/2}2^{-k_2}+\sum_{\substack{k_1\ge k_2 \\ 2^{-k_i} \ll 2^{-l_{1}}}}2^{-k_1/2}\\
&\asymp\sum_{k_1: 2^{-k_1} \ll 2^{-l_{1}}}2^{-k_1/2}+\log N\sum_{k_1: 2^{-k_1} \ll 2^{-l_{1}}}2^{-k_1/2}\\
&\asymp2^{-l_1/2}\log N.
\end{align*}

\medskip
\noindent{Case 2}: exactly one of \(2^{-k_1}, 2^{-k_2}\) is \(> N^{-1}\).
If \(2^{-k_1} > N^{-1} \ge 2^{-k_2}\), then \(k_1 < k_2\); if \(2^{-k_2} > N^{-1} \ge 2^{-k_1}\), then \(k_1 \ge k_2\). In either subcase, the argument as for Case 1 gives the summation is
\(
\ll 2^{-l_1/2} \log N.
\)

\medskip
\noindent{Case 3:} \(2^{-k_1}\le N^{-1}\) and \(2^{-k_2} \le N^{-1}\).

A direct calculation gives
\[\sum_{\substack{k_1, k_2 \\ 2^{-k_i} \le N^{-1}}}
\frac{2^{-k_1/2} 2^{-k_2}}{2^{-k_1} + 2^{-k_2} + N^{-1}}\le N\sum_{\substack{k_1, k_2 \\ 2^{-k_i} \le N^{-1}}}
2^{-k_1/2} 2^{-k_2}\asymp N^{-1/2}.\]
\end{proof}

Hence, by {Sublemma 2}, we have
\begin{align*}
    S_{N,7}&\ll \frac{N^{\frac{3\lambda}{2}(1+\varepsilon)}\log N}{N^2}\sum_{\substack{l_1, l_2 \\ N^{-\lambda} \ll 2^{-l_i} \ll N^{-r}}}
\frac{(2^{-l_1/2}+N^{-1})2^{2l_1}}{2^{l_2}}\\
& \ll \frac{N^{\frac{3\lambda}{2}(1+\varepsilon)}\log N}{N^2}N^{-r}\cdot (N^{3\lambda/2}+N^{2\lambda-1})\\
&\ll \frac{\log N}{N^{2+r-3\lambda(1+\varepsilon)}}.
\end{align*}

\noindent \textbf{Term 8.}  {Sublemma 2} gives
\begin{align*}
 S_{N,8}&:=\frac{1}{N^2}
\sum_{\substack{l_1, l_2 \\ N^{-\lambda} \ll 2^{-l_i} \ll N^{-r}}}
\sum_{\substack{k_1, k_2 \\ 2^{-k_i} \ll 2^{-l_{i}}}}
\frac{N^{(1+\frac{3\lambda}{2})(1+\varepsilon)}}{2^{l_1/\lambda}2^{k_2/2}2^{l_2}}\cdot\frac{(2^{-k_1} 2^{-k_2})^{1/2}}{2^{-k_1} + 2^{-k_2} + N^{-1}}\cdot 2^{2l_1}\\
&\ll \frac{N^{(1+\frac{3\lambda}{2})(1+\varepsilon)}\log N}{N^2}\sum_{\substack{l_1, l_2 \\ N^{-\lambda} \ll 2^{-l_i} \ll N^{-r}}}\frac{(2^{-l_1/2}+N^{-1})2^{2l_1}}{2^{l_1/\lambda}2^{l_2}}\\
&\ll \frac{N^{(1+\frac{3\lambda}{2})(1+\varepsilon)}\log N}{N^{2+r}}\sum_{l_1: N^{-\lambda} \ll 2^{-l_1} \ll N^{-r}}(2^{l_1(\frac{3}{2}-\frac{1}{\lambda})}+N^{-1}2^{l_1(2-\frac{1}{\lambda})})\\
&\ll \frac{N^{(1+\frac{3\lambda}{2})(1+\varepsilon)}\log N}{N^{2+r}}N^{\frac{3\lambda}{2}-1}=\frac{\log N}{N^{2+r-3\lambda-(1+\frac{3\lambda}{2})\varepsilon}}.
\end{align*}

To conclude, since  $\lambda=2/3+\varepsilon,$ we have
\[S_N=S_{N,1}+\cdots+S_{N,8}\ll N^{-\min\{ {r}/{2}-8\varepsilon, \, {2}/{3}-3\varepsilon\}}.\]

\end{proof}

By Lemma~\ref{var:2-2}, we have
\begin{multline*}
\frac{1}{N^2}
\sum_{\substack{(h_1, h_2) \in \mathcal B_N \\ 1 \le |h_1|, |h_2| \le N^{1+\varepsilon}}}
\widehat f\!\left(\frac{h_1}{N}\right)
\widehat f\!\left(\frac{h_2}{N}\right)
\frac{|v_1 v_2|}{v_1^2 + v_2^2 + N^{-1}}
\\
\times
\frac{1}{(|1 - z_{-h_2}| + N^{-1})^2}
\ll
N^{-\min\{r/2 - 8\varepsilon,\ 2/3 - 3\varepsilon\}}.
\end{multline*}
Applying \(ab \le \frac{1}{2}(a^2 + b^2)\) yields
\begin{multline*}
\frac{1}{N^2}
\sum_{\substack{(h_1, h_2) \in \mathcal B_N \\ 1 \le |h_1|, |h_2| \le N^{1+\varepsilon}}}
\widehat f\!\left(\frac{h_1}{N}\right)
\widehat f\!\left(\frac{h_2}{N}\right)
\frac{|v_1 v_2|}{v_1^2 + v_2^2 + N^{-1}}
\\
\times
\frac{1}{(|1 - z_{h_1}| + N^{-1})(|1 - z_{-h_2}| + N^{-1})}
\ll
N^{-\min\{r/2 - 8\varepsilon,\ 2/3 - 3\varepsilon\}}.
\end{multline*}
Since
\[
\frac{1}{|1 - z_{h_i}| + N^{-1}} \le \frac{1}{(|1 - z_{h_i}| + N^{-1})^2}
\quad (i = 1, 2),
\]
we also have
\begin{multline*}
\frac{1}{N^2}
\sum_{\substack{(h_1, h_2) \in \mathcal B_N \\ 1 \le |h_1|, |h_2| \le N^{1+\varepsilon}}}
\widehat f\!\left(\frac{h_1}{N}\right)
\widehat f\!\left(\frac{h_2}{N}\right)
\frac{|v_1 v_2|}{v_1^2 + v_2^2 + N^{-1}}
\\
\times
\left(
\frac{1}{|1 - z_{h_1}| + N^{-1}}
+
\frac{1}{|1 - z_{-h_2}| + N^{-1}}
\right)
\ll
N^{-\min\{r/2 - 8\varepsilon,\ 2/3 - 3\varepsilon\}}.
\end{multline*}
Combining these estimates with Lemma~\ref{cov-bad}, we obtain
\begin{align*}
\widetilde{\operatorname{Var}}^{\mathrm{off},2-2}_{\mathcal B}(R_N(f))
~=~&
\frac{1}{N^4}
\sideset{}{^*}\sum_{\substack{(h_1, h_2) \in \mathcal B_N \\ 1 \le |h_1|, |h_2| \le N^{1+\varepsilon}}}
\widehat f\!\left(\frac{h_1}{N}\right)
\widehat f\!\left(\frac{h_2}{N}\right)
C_N(h_1, h_2)
\\[6pt]
\ll~&
N^{-\min\{r/2 - 8\varepsilon,\ 2/3 - 3\varepsilon\}}.
\end{align*}
Hence
\[
\operatorname{Var}^{\mathrm{off},2-2}_{\mathcal B}(R_N(f))
\ll
N^{-\min\{r/2 - 8\varepsilon,\ 2/3 - 3\varepsilon\}}.
\]

Finally, we obtain
\begin{align} 
\operatorname{Var}^{\mathrm{off}}_{\mathcal B}(R_N(f))
~\ll~&
\sum_{i=1}^{4} \operatorname{Var}^{\mathrm{off},1-i}_{\mathcal B}(R_N(f))
+
\sum_{i=1}^{2} \operatorname{Var}^{\mathrm{off},2-i}_{\mathcal B}(R_N(f))\nonumber
\\[6pt]
\ll~&
N^{-1/3+12r}
+
N^{-\min\{r/2 - 8\varepsilon,\ 2/3 - 3\varepsilon\}}.
\label{var:ob}   
\end{align}

\section{Proof of Theorem \ref{thm:intro-smooth-moments}}

In Section~\ref{FM}, we prove that
\[
\mathbb{E} R_N(f) \longrightarrow \int_{\mathbb R} f(x) \dd x
\quad (N \to \infty).
\]
In Section~\ref{SM}, we show that for \(0 < \varepsilon < 1/18\),
\begin{align*}
\operatorname{Var}\bigl(R_N(f)\bigr)
&\ll_{f,\boldsymbol{\theta}}
\operatorname{Var}^{\mathrm{diag}}(R_N(f))
+
\operatorname{Var}^{\mathrm{off}}_{\mathcal G}(R_N(f))
+
\operatorname{Var}^{\mathrm{off}}_{\mathcal B}(R_N(f))
\\[6pt]
&\ll
N^{-1/3}
+
N^{-1+12r}
+
\left(
N^{-1/3+12r}
+
N^{-\min\{r/2 - 8\varepsilon,\ 2/3 - 3\varepsilon\}}
\right),
\end{align*}
where the last inequality follows from \eqref{var:diag}, Lemma~\ref{var:og}, and \eqref{var:ob}.
Choosing
\[
r \in \left(18\varepsilon,\ \frac{1/3 - \varepsilon}{12}\right),
\]
we obtain
\[
\operatorname{Var}\bigl(R_N(f)\bigr) \ll_{f,\boldsymbol{\theta}} N^{-\varepsilon}.
\]

\section{Proof of Theorem \ref{thm:ap}}

This section is devoted to the proof of Theorem~\ref{thm:ap}. We shall need a number-theoretic lower bound for products of fractional parts of algebraic numbers, which is a direct consequence of Schmidt's Subspace Theorem. The following lemma provides the key estimate that will be used repeatedly in the variance estimates below.

\begin{lemma}
\label{lem:arithmetic-multiplicative-algebraic}
Let $\alpha_1,\alpha_2$ be real algebraic numbers such that
\[
1,\ \alpha_1,\ \alpha_2
\]
are linearly independent over $\mathbb Q$. Then, for every $\varepsilon > 0$, there exists a constant $c = c(\alpha_1, \alpha_2, \varepsilon) > 0$ such that
\begin{equation*}
\|q\alpha_1\| \, \|q\alpha_2\| \ge c |q|^{-1-\varepsilon}
\quad \text{for all } q \in \mathbb Z \setminus \{0\}.
\end{equation*}
\end{lemma}

\begin{proof}
We now apply Schmidt's Subspace Theorem \cite{S89} to prove the required lower bound. Recall that the theorem states the following: if $L_0, \dots, L_{n-1}$ are linearly independent linear forms in $n$ variables with algebraic coefficients, then for any $\varepsilon > 0$, the integer solutions $\mathbf{x} \in \mathbb Z^n \setminus \{0\}$ satisfying
\[
|L_0(\mathbf{x}) \cdots L_{n-1}(\mathbf{x})| < |\mathbf{x}|^{-\varepsilon}
\]
lie in a finite union of proper rational subspaces of $\mathbb R^n$.

For \(i=1,2\), choose
\(p_i(q)\in\mathbb Z\) such that
\[
|q\alpha_i-p_i(q)|
=
\|q\alpha_i\|.
\]
Define the linearly independent algebraic linear forms
\[
L_0(x_0, x_1, x_2) = x_0,~
L_1(x_0, x_1, x_2) = \alpha_1 x_0 - x_1, ~ 
L_2(x_0, x_1, x_2) = \alpha_2 x_0 - x_2,
\]
and set
\[
\mathbf{x}_q := (q, p_1(q), p_2(q)) \in \mathbb Z^3.
\]
A direct computation gives
\begin{equation*}
\prod_{i=0}^{2} |L_i(\mathbf{x}_q)| = q \, \|q\alpha_1\| \, \|q\alpha_2\|.
\end{equation*}
Moreover, we claim that
\[
H(\mathbf{x}_q) := \max\{1, q, |p_1(q)|, |p_2(q)|\} \asymp_{\alpha_1, \alpha_2} q.
\]
To prove the claim, we assume for contradiction that for infinitely many $q$,
\[
\prod_{i=0}^{2}|L_i(\mathbf x_q)| < H(\mathbf{x}_q)^{-\varepsilon}
\]
for some $\varepsilon > 0$. Then the Schmidt's Subspace Theorem ($n=3$)
implies that all such vectors $\mathbf{x}_q$ lie in a finite union of proper rational subspaces of $\mathbb R^3$. Hence at least one proper rational subspace contains infinitely many of the $\mathbf{x}_q$.
We   show that this is impossible. 

Suppose there exist integers $A, B, C$, not all zero, such that
\[
A q + B p_1(q) + C p_2(q) = 0
\]
for infinitely many $q$. Dividing by $q$ and taking the limit along this infinite subsequence, we obtain
\[
A + B \lim_{q  } \frac{p_1(q)}{q} + C \lim_{q } \frac{p_2(q)}{q} = 0.
\]
Since $p_i(q)/q \to \alpha_i$, this gives
\[
A + B\alpha_1 + C\alpha_2 = 0,
\]
contradicting the assumed linear independence of $1, \alpha_1, \alpha_2$ over $\mathbb Q$. 

Consequently, for all sufficiently large $q$,
\begin{equation*}
q \, \|q\alpha_1\| \|q\alpha_2\| \gg_{\alpha_1, \alpha_2, \varepsilon} q^{-\varepsilon}.
\end{equation*}
The finitely many remaining small $q$ are handled by adjusting the implicit constant, and the case of negative $q$ follows from the symmetry $\| - q\alpha_i\|= \|q\alpha_i\|$. 
\end{proof}

Apply Lemma~\ref{lem:arithmetic-multiplicative-algebraic} with
\[
\alpha_1=\delta,\qquad \alpha_2=\sigma.
\]
For every \(\varepsilon>0\), we obtain
\begin{equation*}
X_qY_q:=\|q\delta\|\,\|q\sigma\|
\gg_{\sigma,\delta,\varepsilon}|q|^{-1-\varepsilon},
\end{equation*}
for all \(q\in\mathbb{Z}\setminus\{0\}\).

Since \(1,\sigma,\delta\) are linearly independent over \(\mathbb{Q}\), at least one of \(\sigma\) and \(\delta\) must be irrational. Without loss of generality, assume that \(\delta\) is irrational. Then \(X_q>0\) and
\[
Y_q\gg_{\sigma,\delta,\varepsilon}\frac{|q|^{-1-\varepsilon}}{X_q}.
\]

For arbitrary integers \(m,n\in\mathbb{Z}\), we have the elementary estimates
\[
|2h\delta-n|\ge X_{2h},\qquad
\left|h\sigma-\frac{n}{2}-m\right|
=\frac{1}{2}\bigl|2h\sigma-(n+2m)\bigr|
\ge\frac{1}{2}Y_{2h}.
\]
Hence, by the definition of \(\rho_h(\boldsymbol{\theta})\),
\begin{equation}
\rho_h(\boldsymbol{\theta})
\ge X_{2h}^2+\frac{1}{2}Y_{2h}
\gg_{\sigma,\delta,\varepsilon}
X_{2h}^2+\frac{|h|^{-1-\varepsilon}}{2X_{2h}}.
\label{eq:rho-lower-bound-XY}
\end{equation}

Now consider the function
\[
\phi(x)=x^2+\frac{L}{2x},\qquad x>0,
\]
where \(L>0\) is a constant. A simple computation shows that \(\phi\) attains its minimum at \(x^3=L/4\), with
\[
\min_{x>0}\phi(x)=3\cdot 2^{-4/3}L^{2/3}.
\]
Applying this with \(x=X_{2h}\) and \(L=|h|^{-1-\varepsilon}\), we obtain from \eqref{eq:rho-lower-bound-XY} that
\[
\rho_h(\boldsymbol{\theta})
\gg_{\sigma,\delta,\varepsilon}
|h|^{-\frac{2}{3}(1+\varepsilon)}.
\]
Thus \(\boldsymbol{\theta}\) is of \(2/3\)-Diophantine type. This completes the proof of Theorem~\ref{thm:ap}.

\section*{Appendix: Proofs for Subsection   \ref{subsec:DC}}

We now present the proofs  deferred from Subsection \ref{subsec:DC}.

\begin{proof}[Proof of Theorem \ref{thm:local}] 

Near a point at which \(Z(\mathbf{x})\) is close to \(1\), only two possibilities can occur: either both \(e(x)\) and \(\cos(2\pi y)\) are close to \(1\), or both are close to \(-1\). These correspond, respectively, to \(\mathbf{x}=(x,y)\in\mathbb{T}^2\) lying close to \(\mathbf{k}_1\) or to \(\mathbf{k}_2\).

We first consider a neighborhood of the exceptional point \(\mathbf{k}_1=(0,0)\). By the Taylor expansions
\[
e(x)=1+2\pi i x+o(x),\qquad
\cos(2\pi y)=1-2\pi^2 y^2+o(y^2),
\]
we have, for sufficiently small \(\mathbf{x}=(x,y)\),
\begin{equation}\label{loc-asy}
|1-Z(\mathbf{x})|\asymp |x|+y^2<2\eta_{\mathrm{loc}}.
\end{equation}
Choose \(\eta_{\mathrm{loc}}>0\) small enough so that this estimate holds whenever \(\mathbf{x}\) lies in the corresponding \(\eta_{\mathrm{loc}}\)-neighborhood of \(\mathbf{k}_1\); a similar estimate holds near \(\mathbf{k}_2\) by translating the origin to that point.

Now put \(\mathbf{x}=h\boldsymbol{\theta}=(h\sigma,h\delta)\). If \(\mathbf{x}\) is close to \(\mathbf{k}_1\), then
\[
|1-Z(h\boldsymbol{\theta})|
\asymp \min_{m,n\in\mathbb{Z}}\left\{|h\sigma-m|+|h\delta-n|^2\right\}.
\]
Since \(|h\delta-n|^2=\frac14|2h\delta-2n|^2\), and the factor \(1/4\) is absorbed into the comparison constants, this is equivalent to
\[
|1-Z(h\boldsymbol{\theta})|
\asymp \min_{m,n\in\mathbb{Z}}\left\{|h\sigma-m|+|2h\delta-2n|^2\right\}.
\]
Making the change of variables \(m'=m-n\), this becomes
\[
|1-Z(h\boldsymbol{\theta})|
\asymp \min_{m',n\in\mathbb{Z}}\left\{\left|h\sigma-\frac{2n}{2}-m'\right|+|2h\delta-2n|^2\right\}.
\]
Similarly, if \(\mathbf{x}\) is close to \(\mathbf{k}_2\), we have
\[
|1-Z(h\boldsymbol{\theta})|
\asymp \min_{m,n\in\mathbb{Z}}\left\{\left|h\sigma-m-\frac12\right|+\left|h\delta-n-\frac12\right|^2\right\},
\]
which, after replacing the odd integer by \(2n+1\) and shifting \(m\) appropriately, yields
\[
|1-Z(h\boldsymbol{\theta})|
\asymp \min_{m,n\in\mathbb{Z}}\left\{\left|h\sigma-\frac{2n+1}{2}-m\right|+|2h\delta-(2n+1)|^2\right\}.
\]
Combining the even and odd cases for the integer \(n\) yields precisely \(\rho_h(\boldsymbol{\theta})\).

Hence, in a sufficiently small neighborhood of \(\mathcal{K}\),
\[
|1-Z(h\boldsymbol{\theta})|\asymp \rho_h(\boldsymbol{\theta}),
\]
with constants depending only on the choice of \(\eta_{\mathrm{loc}}\). 
\end{proof}

\begin{proof}[Proof of Theorem \ref{LI}]
Suppose, for contradiction, that there exist \(q_1,q_2\in\mathbb{Z}\), not both zero, satisfying
\[
q_1\sigma+q_2\delta\in\mathbb{Z}.
\]
Let \(d=\gcd(|q_1|,|q_2|)\) and write \(q_1=da\), \(q_2=db\) with \(\gcd(a,b)=1\). Then
\[
d(a\sigma+b\delta)\equiv0\pmod1.
\]
Since \((a,b)\) is primitive, this implies that there exists \(t\in\mathbb{T}\) such that
\[
d\boldsymbol{\theta}\equiv(bt,-at)\pmod1.
\]
If \(t\) were rational, then \(kd\boldsymbol{\theta}\equiv\mathbf{0}\pmod1\) for some \(k\ge1\), giving \(\rho_{kd}(\boldsymbol{\theta})=0\), contradicting the Diophantine condition, which requires \(\rho_h(\boldsymbol{\theta})>0\) for all \(h\neq0\). Hence \(t\) is irrational.

By Dirichlet's approximation theorem, there exist infinitely many integer pairs \((k,\ell_k)\) such that
\[
|kt-\ell_k|\le k^{-1}.
\]
Set \(h=dk\). Choose integers \(r,s\) with
\[
d\sigma=bt+r,\qquad d\delta=-at+s.
\]
Then, with \(h=dk\),
\[
h\sigma=b\ell_k+b\varepsilon_k+kr,\qquad
h\delta=-a\ell_k-a\varepsilon_k+ks,
\]
where \(\varepsilon_k:=kt-\ell_k\) satisfies \(|\varepsilon_k|\le k^{-1}\).

In the definition of \(\rho_h\), take
\[
n_k=-2a\ell_k+2ks,\qquad
m_k=(a+b)\ell_k+k(r-s).
\]
A direct substitution yields
\[
2h\delta-n_k=-2a\varepsilon_k,\qquad
h\sigma-\frac{n_k}{2}-m_k=b\varepsilon_k.
\]
Hence
\[
\rho_h(\boldsymbol{\theta})
\le 4a^2\varepsilon_k^2+|b|\,|\varepsilon_k|
\ll_{a,b} k^{-1}
\ll_d h^{-1}.
\]
Since this holds for infinitely many \(h\), the \(\lambda\)-Diophantine condition would imply
\[
c_{\boldsymbol{\theta}}h^{-\lambda}\ll h^{-1},
\]
which is impossible for infinitely many arbitrarily large \(h\) when \(\lambda<1\).
\end{proof}

\begin{proof}[Proof of Theorem \ref{full}]
For $\varepsilon_0 \in \mathbb{Q}^{+}$ (the set of positive rational numbers), define the auxiliary set
\[
D_\lambda^*(\varepsilon_0) := \left\{ \boldsymbol{\theta} \in \mathbb{T}^2 : \rho_h(\boldsymbol{\theta}) \ge |h|^{-(\lambda+\varepsilon_0)} \text{ for all } h \in \mathbb{Z}\setminus\{0\} \right\}.
\]
It is straightforward to verify that
\[
\bigcap_{\varepsilon_0 \in\mathbb{Q}^{+}} D_\lambda^*(\varepsilon_0) \subset D_\lambda.
\]
Hence it suffices to show that $\bigcap_{\varepsilon_0\in\mathbb{Q}^{+}} D_\lambda^*(\varepsilon_0)$ has full Lebesgue measure, or equivalently, that the complement of $D_\lambda^*(\varepsilon_0)$ has measure zero for each $\varepsilon_0 \in\mathbb{Q}^{+}$.

To this end, write the complement as
\begin{align*}
D_\lambda^*(\varepsilon_0)^c
&= \left\{ \boldsymbol{\theta} \in \mathbb{T}^2 : \rho_h(\boldsymbol{\theta}) < |h|^{-(\lambda+\varepsilon_0)} \text{ for infinitely many } h \in \mathbb{Z}\setminus\{0\} \right\} \\
&= \bigcap_{H=1}^{\infty} \bigcup_{\substack{h \in \mathbb{Z} \\ |h| \ge H}} \mathcal{E}_h,
\end{align*}
where
\[
\mathcal{E}_h := \left\{ \boldsymbol{\theta} \in \mathbb{T}^2 : \rho_h(\boldsymbol{\theta}) < |h|^{-(\lambda+\varepsilon_0)} \right\}.
\]
We now estimate the Lebesgue measure of $\mathcal E_h$. By Theorem~\ref{thm:local}, for sufficiently large $|h|$,
\[
\mathcal{E}_h \subset
\left\{ \boldsymbol{\theta} \in \mathbb{T}^2 : \|h\sigma\| + \|h\delta\|^2 \ll |h|^{-(\lambda+\varepsilon_0)} \right\}
\cup
\left\{ \boldsymbol{\theta} \in \mathbb{T}^2 : \left\|h\sigma-\frac12\right\| + \left\|h\delta-\frac12\right\|^2 \ll |h|^{-(\lambda+\varepsilon_0)} \right\}.
\]
Consider the first set in the above union. On the torus, the conditions $\|h\sigma\| < r$ and $\|h\delta\| < r^{1/2}$ define regions of measure $\asymp r$ and $\asymp r^{1/2}$, respectively. Hence its measure is
\[
\ll |h|^{-3(\lambda+\varepsilon_0)/2}.
\]
The second set has the same measure as the first, since it is obtained by a translation on the torus. Therefore,
\[
\operatorname{Leb}(\mathcal{E}_h) \ll |h|^{-3(\lambda+\varepsilon_0)/2}.
\]

Since $\lambda \ge 2/3$, we have
\[
\sum_{h \in \mathbb{Z}\setminus\{0\}} \operatorname{Leb}(\mathcal{E}_h)
\ll \sum_{h \neq 0} |h|^{-3(\lambda+\varepsilon_0)/2}
< \infty.
\]
The first Borel--Cantelli lemma then implies that
\[
\operatorname{Leb}\left( D_\lambda^*(\varepsilon_0)^c \right) = 0.
\]
\vskip -.8cm
\end{proof}

\bigskip
\noindent {\bf  Acknowledgements.}  This work is supported by National Key R$\&$D Program of China (No. 2024YFA1013700) and NSFC (Nos. 12671113 and 12331005), and the Fundamental Research Funds for the Central Universities.

\bigskip

\noindent School of Mathematics and Statistics, Huazhong University of Science and Technology, Wuhan 430074, China.

\noindent\textit{Email address:} \texttt{tanbo@hust.edu.cn}

\bigskip

\noindent School of Mathematics and Statistics, Huazhong University of Science and Technology, Wuhan 430074, China.

\noindent\textit{Email address:} \texttt{zhangzhen\_@hust.edu.cn}
\bigskip

\noindent School of Mathematics and Statistics, Wuhan University of Technology, Wuhan 430070, China.

\noindent\textit{Email address:} \texttt{zhouql@whut.edu.cn}

\end{document}